\documentclass[12pt]{article}
\usepackage{amsmath}
\usepackage{amsfonts}
\usepackage{amsthm}
\usepackage{amscd}
\usepackage{amssymb}
\usepackage{color}

\usepackage{graphicx}
\graphicspath{ {images/} }

\newtheorem{theorem}{Theorem}
\newtheorem{lemma}[theorem]{Lemma}

\newtheorem{observation}[theorem]{Observation}
\newtheorem{problem}[theorem]{Problem}
\newtheorem{example}[theorem]{Example}

\newtheorem{corollary}[theorem]{Corollary}
\newtheorem{conjecture}[theorem]{Conjecture}

\theoremstyle{definition}
\newtheorem{definition}[theorem]{Definition}
\theoremstyle{remark}
\newtheorem{remark}[theorem]{Remark}
\theoremstyle{comment}

\numberwithin{equation}{section}
\numberwithin{theorem}{section}
\numberwithin{figure}{section}

\newcommand{\1}{\mathbf{1}}
\newcommand{\ESA}{{\rm ESA}}

\def\1{\mathbf{1}}

\def\f2{\mathbb{F}_2}

\def\supp{\hskip0.02cm{\rm supp}\hskip0.01cm}

\newcommand{\ep}{\varepsilon}

\newcommand{\sign}{{\rm sign}\hskip0.02cm}

\newcommand{\bbN}{\mathbb{N}}

\newcommand{\one}{\mathbf{1}}

\newcommand{\PP}{{\mathcal{P}}}
\newcommand{\al}{\alpha}
\newcommand{\be}{\beta}
\newcommand{\g}{\gamma}
\newcommand{\de}{\delta}
\newcommand{\e}{\varepsilon}

\newcommand{\la}{\lambda}

\newcommand{\om}{\omega}

\newcommand{\bbZ}{\mathbb{Z}}

\newcommand{\RR}{\operatorname{{\tt Ref}_n}}
\newcommand{\RRnu}{\operatorname{{\tt Ref}_{n,\nu}}}
\newcommand{\card}{\operatorname{card}}

\newcommand{\disp}{\displaystyle}
\newcommand{\lb}{\label}

\newcommand{\wtw}{if and only if}

\newcommand{\buoo}{without loss of generality}
\newcommand{\Buo}{Without loss of generality }
\newcommand{\Buoo}{Without loss of generality}

\newcommand{\DEF}{\buildrel {\mbox{\tiny def}}\over =}

\newcommand\remove[1]{}

\begin{document}

\title{A new approach to low-distortion embeddings of finite metric spaces into non-superreflexive Banach spaces}

\author{Mikhail~I.~Ostrovskii and Beata~Randrianantoanina}

\date{~}

\maketitle



\begin{abstract}
The main goal of this paper is to develop a new embedding method
which we use to show that some finite metric spaces admit
low-distortion embeddings into all non-superreflexive spaces. This
method is based on the theory of equal-signs-additive
sequences developed by Brunel and Sucheston (1975-1976). We also
show that some of the low-distortion embeddability results
obtained using this method cannot be obtained using the method
based on the factorization between the summing basis and the unit
vector basis of $\ell_1$, which was used by Bourgain (1986) and
Johnson and Schechtman (2009).
\end{abstract}

{\small \noindent{\bf Keywords:} diamond graph;
equal-signs-additive sequence; metric characterization;
superreflexive Banach space

}



\section{Introduction}

One of the basic problems of the theory of metric embeddings is:
given some Banach space or a natural class $\mathcal{P}$ of Banach
spaces find classes of metric spaces which admit low-distortion
embeddings into each Banach space of the class $\mathcal{P}$. The
main goal of this paper is to develop a new embedding method which
can be used to show that some finite metric spaces admit
low-distortion embeddings into all non-superreflexive spaces
(Theorem~\ref{T:Main}). This method is based on the theory of
equal-signs-additive sequences (ESA) developed by Brunel and
Sucheston \cite{BS74,BS75,BS76}. We show in
Theo\-rem~\ref{no-diamonds-in-X-Delta} that some of the
low-distortion embeddability results obtained using this method
cannot be obtained using the method
 based on the factorization between the
summing basis and the unit vector basis of $\ell_1$, which was
used by Bourgain \cite{Bou86} and Johnson and Schechtman
\cite{JS09}, see Corollary \ref{C:UsedBJS}.

The problem mentioned at the beginning of the previous paragraph
can be regarded as one side of the problem of metric
characterization of the class $\mathcal{P}$. Recall that, in the
most general sense, a \emph{metric characterization} of a class of
Banach spaces  is a characterization which refers only to the
metric structure of a Banach space and does not involve the linear
structure. The study of metric characterizations became an active
research direction in mid-1980s, in the work of Bourgain
\cite{Bou86} and Bourgain, Milman, and Wolfson \cite{BMW86} (see
also Pisier \cite[Chapter 7]{Pis86}). The work on metric
characterization of isomorphic invariants of Banach spaces
determined by their finite-dimensional subspaces, and on
generalization of the obtained theory to general metric spaces
became known as the {\it Ribe program}, see \cite{Bal13,Nao12}.
The type of metric characterizations which is closely related to
the present  paper is the following:

\begin{definition}[\cite{Ost13a}]\label{D:TestSpace}
Let $\mathcal{P}$ be a class of Banach spaces and let
$T=\{T_\alpha\}_{\alpha\in A}$ be a set of metric spaces. We say
that $T$ is a set of {\it test-spaces} for $\mathcal{P}$ if the
following two conditions are equivalent for a Banach space $X$:
{\bf (1)} $X\notin\mathcal{P}$; {\bf (2)} The spaces
$\{T_\alpha\}_{\alpha\in A}$ admit bilipschitz embeddings into $X$
with uniformly bounded distortions.
\end{definition}

There are several known different  sets of finite test-spaces for
superreflexivity of Banach spaces, including: the set  of all
finite binary trees (Bourgain \cite{Bou86}, see also
\cite{Mat99,Klo14}),   the set of diamond graphs, and the set of
Laakso graphs (Johnson and Schechtman \cite{JS09}, see also
\cite{Ost11}).  In \cite{Ost14,OO15+,LNOO15+} it was shown that
these sets of  test-spaces are independent in the  sense that the
respective families of metric spaces do not admit bilipschitz
embeddings one into another with uniformly bounded distortions.

There are also metric characterizations of superreflexivity using
only one metric test-space. Baudier \cite{Bau07} proved that
 the infinite binary tree is such a test-space, many other one-element test-spaces for superreflexivity were described in
\cite{Ost14}. See \cite{Ost16} for a survey on metric
characterizations of superreflexivity.

The first main result of the present paper  is a construction of
bilipschitz embeddings with a uniform bound on distortions of
diamond graphs with arbitrary finite number of branches  into any
non-superreflexive Banach space. Multibranching diamonds are a
generalization of usual (binary) diamond graphs. Their embedding
properties were first studied  in \cite{LR10}.

\begin{definition}[cf. \cite{LR10}]\label{D:BranchDiam}
For  any integer  $k\ge 2$, we define $D_{0,k}$ to be the graph
consisting of two vertices joined by one edge. For any $n\in\bbN$,
if the graph $D_{n-1,k}$  is already defined, the graph $D_{n,k}$
is defined as the graph obtained from $D_{n-1,k}$ by replacing
each edge $uv$ in $D_{n-1,k}$ by a set of $k$  independent paths
of length $2$ joining $u$ and $v$. We endow $D_{n,k}$ with the
shortest path distance. We call $\{D_{n,k}\}_{n=0}^\infty$ {\it
diamond graphs of branching $k$}, or {\it diamonds of branching}
$k$.
\end{definition}

We prove

\begin{theorem}\label{T:Main} For every $\ep>0$, any non-superreflexive Banach
space $X$, and any $n,k\in \mathbb{N}$, $k\ge 2$, there exists a
bilipschitz embedding of $D_{n,k}$ into $X$ with distortion at
most $8+\ep$.
\end{theorem}

  In particular, Theorem~\ref{T:Main} together with the result of \cite{JS09} implies that the set of all diamond graphs of arbitrary finite branching
is a set of test-spaces for superreflexivity.

To prove Theorem~\ref{T:Main} we develop a novel technique of
constructing low-distortion embeddings of finite metric spaces
into non-superreflexive Banach spaces. This technique, which we
consider   the main contribution of the present paper, relies on
the concept of equal-sign-additive (ESA) sequences developed by
Brunel and Sucheston \cite{BS74,BS75,BS76} in their deep study of
superreflexivity. Our construction relies on  ESA basic sequences
and on, now standard, use of independent random variables, to
identify in any  non-superreflexive Banach space an element $x$
with multiple well-separated (exact) metric midpoints between $x$
and $0$,  with an additional property that the selected metric
midpoints between $x$ and $0$ have a structure sufficiently
similar to the element $x$, so that the procedure of selecting
multiple  well-separated metric midpoints  can be iterated the
desired number of times. The construction and the proof are
presented in Section~\ref{S:D1k}. We have not attempted to find
the best distortion constant in Theorem~\ref{T:Main}. We do not
expect that $8+\e$ is   best possible. In Section~\ref{S:BS}, we
briefly recall the definitions and results from
\cite{BS74,BS75,BS76} that we use.

It is clear that our techniques work for somewhat larger families
of graphs. In particular, in Section~\ref{S:Laakso} we outline a
proof of an analogue of Theorem~\ref{T:Main} for a set of Laakso
graphs with arbitrary finite branching (cf.
Definition~\ref{D:multiLaakso} and Theorem~\ref{T:Laakso}).
However in more general cases the technical details become much
more complicated. We decided to focus our attention in this paper
on the construction of low-distortion embeddings in the case of
multibranching diamonds, so that the main ideas of the
construction are more transparent, and because, as we explain
below,   this case cannot be proved  by using previously known
methods.  Also, in recent years diamond graphs
of high branching have appeared naturally in different contexts,
cf. \cite{BCDKRSZ16+,LR10,OR16}.

The next main result of the present  paper
(Theorem~\ref{no-diamonds-in-X-Delta}) shows that the new
technique  that we develop is inherently different from the known
before method of constructing metric embeddings into
non-superreflexive Banach spaces (Bourgain \cite{Bou86} and
Johnson-Schechtman \cite{JS09}). Their method   is based
on the following result which emerged in the following sequence of
papers: Pt\'ak \cite{Pta59}, Singer \cite{Sin62}, Pe\l czy\'nski
\cite{Pel62}, James \cite{Jam64}, Milman-Milman \cite{MM65}.
Denote by $\|\cdot\|_1$  the standard norm on $\ell_1$, and by $\|\cdot\|_s$   the summing norm on $\ell_1$, that is,
\[\|(a_i)_{i=1}^\infty\|_s\DEF\sup_k\left|\sum_{i=1}^k
a_i\right|.\]
It is clear that $(\ell_1,\|\cdot\|_s)$ is a normed
space, but not a Banach space.

\begin{theorem} A Banach space $X$ is nonreflexive if and only if the
identity operator $I:(\ell_1,\|\cdot\|_1)\to(\ell_1,\|\cdot\|_s)$ factors
through $X$ in the following sense: there are bounded linear
operators $S:(\ell_1,\|\cdot\|_1)\to X$ and $T:S(\ell_1)\to
(\ell_1,\|\cdot\|_s)$ such that $I=TS$. Furthermore, if $X$ is
nonreflexive, then there is a factorization $I=TS$ through $X$, as above,
such that the product $\|T\|\cdot\|S\|$ is bounded by a constant
$\Pi$ which does not depend on $X$.
\end{theorem}

The following corollary is immediate:

\begin{corollary}\label{C:UsedBJS} If a metric space $M$ admits an embedding of distortion $D$ into
$\ell_1$, such that the distances induced by the $\ell_1$ norm and
the summing norm on the image of $M$ are $C$-equivalent, then $M$
admits an embedding into an arbitrary nonreflexive space with
distortion at most $D\cdot \Pi\cdot C$. If, in addition, $M$ is
finite, then the above assumption implies that for every $\ep>0$,
$M$ embeds into any non-superreflexive space with distortion at
most $D\cdot \Pi\cdot C+\ep$.
\end{corollary}

All known results on embeddings of families of finite metric
spaces into all non-superreflexive Banach spaces with uniformly
bounded distortions are based on Corollary~\ref{C:UsedBJS}. We
show that the set of all diamonds of all finite branchings does
not satisfy the assumption of Corollary~\ref{C:UsedBJS}, and thus
the method of \cite{Bou86,JS09} of constructing low-distortion
embeddings cannot be used to prove Theorem~\ref{T:Main}. To see
this, first observe that the assumption of
Corollary~\ref{C:UsedBJS} is equivalent (with modified constants)
to the assumption: there exists an embedding $f:M\to \ell_1$ such
that
\begin{equation}\lb{factor}
\forall u,v\in M\quad \|f(u)-f(v)\|_1\le d_M(u,v)< C\cdot \|f(u)-f(v)\|_s.
\end{equation}
We prove the following result (see
Section~\ref{S:NoMultBrXdelta}).

\begin{theorem} \label{no-diamonds-in-X-Delta}
For every $C> 1$ there exists $k(C)\in\bbN$ such that if for
some $k\in\mathbb{N}$ and every $n\in\mathbb{N}$ there exists an
embedding $f_n:D_{n,k}\to\ell_1$ satisfying
\begin{equation*}
\forall u,v\in D_{n,k}\quad \|f_n(u)-f_n(v)\|_1\le
d_{D_{n,k}}(u,v)< C\cdot \|f_n(u)-f_n(v)\|_s,
\end{equation*}
then $k\le k(C)$.
\end{theorem}

\begin{remark}
We note that Theorem~\ref{no-diamonds-in-X-Delta} does not exclude the possibility that $\forall k\in\bbN\ \exists C=C(k)>1$ so that for all $n\in\bbN$ there exists an embedding from $D_{n,k}$ into $\ell_1$ that satisfies condition  \eqref{factor}. Theorem~\ref{no-diamonds-in-X-Delta} only implies that if such numbers $C(k)$ exist for all $k\in\bbN$ then they would not be uniformly bounded.

We do not know whether such numbers $C(k)$ exist for all $k\in\bbN$. Johnson and Schechtman \cite{JS09} proved that $C(2)$ exists, but we don't even know whether $C(3)$ exists.
\end{remark}

From another perspective, the results of the present paper can be
viewed as a step in a generalization of results on existence of
low-distortion embeddings of finite metric spaces into $\ell_1$,
to existence of such embeddings into any non-superreflexive Banach
space. Starting with seminal works \cite{LLR95,AR98,GNRS04}, due
to their numerous important applications, the study of
low-distortion metric embeddings has become a very active area of
research also in theoretical computer science, for more
information we refer the reader to the books
\cite{DL97,Mat02,WS11}, the surveys \cite{Ind01,Lin02}, and the
list of open problems \cite{MatN-open-list} that has been very
important   in the development of the subject.

Here we just want to mention the following, still open, well-known conjecture.

\begin{conjecture}[Planar Conjecture]\label{C:planar} Any metric supported on a (finite) planar graph (that is a shortest-path metric on any planar graph whose edges have arbitrary weights)
can be embedded into $\ell_1$ with
constant distortion.
\end{conjecture}

Gupta, Newman, Rabinovich, and Sinclair say that the Planar Conjecture was a motivation for their work \cite{GNRS04}. Recall that it is well-known that planar graphs are characterized by the condition that they do not contain the complete graph  $K_5$ nor the complete bipartite graph $K_{3,3}$ as a minor ($H$ is a minor of $G$ if it can be obtained from
$G$ via a sequence of edge contractions, edge deletions, and
vertex deletions; note that all graphs are considered with arbitrarily assigned weights on  edges);
we refer to \cite{Die00} for graph theory terminology and background.

As a step towards a solution of the Planar Conjecture, Gupta, Newman, Rabinovich, and  Sinclair \cite{GNRS04} proved that   all (finite) graphs that do not contain the complete graph $K_4$ as a minor can be embedded into $\ell_1$ with
 distortion at most $14$. The graphs excluding $K_4$ as a minor are also known as {\it series-parallel graphs}.
   Recall, that the graph $G=(V,E)$ is called
 {\it  series-parallel   with terminals $s, t\in V$} if $G$ is either a single
edge $(s, t)$, or $G$ is a series combination or a parallel combination of two series parallel
graphs $G_1$ and $G_2$ with terminals $s_1, t_1$ and $s_2, t_2$. The {\it series combination}
of  $G_1$ and $G_2$  is formed by setting $s=s_1, t=t_2$ and identifying $s_2=t_1$;
the {\it parallel combination} is formed by identifying $s=s_1=s_2$, $t=t_1=t_2$.

Gupta, Newman, Rabinovich, and  Sinclair \cite[p.~235]{GNRS04} formulated the following generalization of the Planar Conjecture.

\begin{conjecture}[Forbidden-minor embedding conjecture]\label{C:GNRS}
For any finite set $L$ of graphs, there exists a constant $C_L<\infty$ so that every metric on any graph that does not contain any member of the set $L$ as a minor can be embedded into $\ell_1$ with
 distortion at most $C_L$.
\end{conjecture}

Conjectures~\ref{C:planar} and \ref{C:GNRS} remain open despite very active work on them, cf. e.g. \cite{CJLV08,LS09,LM10,LR10,LP13,LS13,Sid13,CCN15}  and their references.

Chakrabarti,  Jaffe,   Lee, and  Vincent \cite{CJLV08} improved
the upper bound obtained in \cite{GNRS04} by proving that every
series parallel graph can be embedded into $\ell_1$ with
distortion at most 2. Lee and Raghavendra \cite{LR10} proved that
2 is best possible  - it is the supremum of $\ell_1$-distortions
of the family of all multibranching diamonds $D_{n,k}$, for all
$n,k\in\bbN$, with uniform weights on all edges, that is, the same
family of graphs that we study in Theorem~\ref{T:Main}.

Several methods of constructing low-distortion embeddings of finite metric spaces into $\ell_1$ are now available. However these methods rely on special geometric properties of $\ell_1$, and it is not known whether there exist methods applicable in other classes of Banach spaces. In particular, Johnson and
Schechtman \cite[Remark 6]{JS09} suggested
  the following problem.

\begin{problem}\label{P:SerParSR} Let $X$ be any non-superreflexive Banach space. Is it true
that all series-parallel graphs admit  bilipschitz
embeddings into $X$ with uniformly bounded distortions?
\end{problem}

Theorems~\ref{T:Main} and \ref{T:Laakso} can be seen as a step towards a solution of Problem~\ref{P:SerParSR}.
\medskip

We suggest the following analogue of Conjecture~\ref{C:GNRS}.
\begin{problem} Do there exist a non-superreflexive  Banach space $X$ and a  finite graph $G$ such that the family of all finite graphs which exclude $G$ as a minor   is not embeddable into $X$  with
uniformly bounded distortions?
\end{problem}

To
the best of our knowledge this problem is   open.

\section{Preliminaries}

Throughout the paper we try to use standard terminology and
notation. We refer to \cite{Die00} for graph theoretical
terminology and to \cite{Ost13} for terminology of the theory of
metric embeddings.

In this section we recall the results of Brunel and Sucheston
about equal-signs-additive (ESA) sequences, that we will use in an
essential way. In the second part of this section we describe the
notation that we will use for vertices of multi-branching
diamonds.

\subsection{Equal signs
additive (ESA) sequences}\label{S:BS}

Our main construction relies on the following notions that were
introduced by Brunel and Sucheston in their deep study of
superreflexivity.

\begin{definition}[{\cite[p.~83--84]{BS75}, \cite[p.~287-288]{BS76}}]\label{D:ESAetc}
Let $\{e_i\}_{i=1}^\infty$ be a sequence in a normed space
$(X,\|\cdot\|)$.
\medskip

{\bf (1)} The norm $\|\cdot\|$ is called {\it
equal-signs-additive} (ESA) on  $\{e_i\}_{i=1}^\infty$ if for any
finitely non-zero sequence $\{a_i\}$ of real numbers such that
$a_ka_{k+1}\ge 0$, we have
\begin{equation*}\tag{ESA}
\left\|\sum_{i=1}^{k-1}a_ie_i+(a_k+a_{k+1})e_k+\sum_{i=k+2}^\infty
a_ie_i\right\|=\left\|\sum_{i=1}^{\infty}a_ie_i\right\|.
\end{equation*}

\medskip

{\bf (2)} The norm $\|\cdot\|$ is called {\it subadditive} (SA) on
$\{e_i\}_{i=1}^\infty$ if for any finitely non-zero sequence
$\{a_i\}$ of real numbers, we have
\begin{equation*}\label{E:SA}\tag{SA}
\left\|\sum_{i=1}^{k-1}a_ie_i+(a_k+a_{k+1})e_k+\sum_{i=k+2}^\infty
a_ie_i\right\|\le\left\|\sum_{i=1}^{\infty}a_ie_i\right\|.
\end{equation*}
\medskip

{\bf (3)} The norm $\|\cdot\|$ is called {\it invariant under
spreading} (IS) on $\{e_i\}_{i=1}^\infty$ if for any finitely
non-zero sequence $\{a_i\}$ of real numbers, and for any
(increasing) subsequence $\{k_i\}_{i=1}^\infty$ in $\mathbb{N}$,
we have
\begin{equation*}\label{E:IS}\tag{IS}
\left\|\sum_{i=1}^{\infty}a_ie_i\right\|=\left\|\sum_{i=1}^{\infty}a_ie_{k_i}\right\|.
\end{equation*}
\medskip
\end{definition}

If the norm is understood, we will simply say that the  sequence
$\{e_i\}_{i=1}^\infty$ is ESA, SA, or IS, respectively.

Brunel and Sucheston proved  the following relationships between
the above notions.

\begin{lemma}[{\cite[Lemma~1]{BS76}}]\lb{lemmaESA}
A sequence is \ESA\ \wtw\ it is  SA, and that every ESA sequence is also IS .
\end{lemma}

 Moreover Brunel and Sucheston    discovered the following deep result.

\begin{theorem}[\cite{BS75}]\label{T:BS}  For each nonreflexive space $X$ there exists a Banach space $E$ with
 an \ESA\ basis that is finitely
representable in $X$.
\end{theorem}

Since this theorem is not explicitly stated in \cite{BS75} (and in
\cite[Lemma 11.33]{Pis16} the statement is slightly different), we
describe how to get it from the argument presented there.

 By
\cite{Pta59} (see also \cite{Jam64,MM65,Pel62,Sin62}), since $X$ is not reflexive, there exist: a sequence $\{x_i\}_{i=1}^\infty$ in $B_X$
(the unit ball of $X$), a number $0<\theta<1$, and a sequence of functionals
$\{f_i\}_{i=1}^\infty\subset B_{X^*}$, so that
\[f_n(x_k)=\begin{cases} \theta &\hbox{ if }n\le k\\
0 &\hbox{ if }n>k.
\end{cases}\]
Following \cite[Proposition 1]{BS74} we build on the sequence
$\{x_i\}$ the spreading model $\widetilde X$ (the term {\it
spreading model} was not used in \cite{BS74}, it was introduced
later, see \cite[p.~359]{Bea79}). The natural basis
$\{e_i\}_{i=1}^\infty$ in $\widetilde X$ is IS. The space
$\widetilde X$ is finitely representable in $X$, see
\cite[p.~83]{BS75}. Now one can use the procedure described in
\cite[Proposition~2.2 and Lemma~2.1]{BS75}, and obtain a Banach
space $E$ which is finitely representable in $\widetilde X$ and
has an \ESA\ basis. (Actually, the fact that we get a basis was
not verified in \cite{BS75}, this was done in
\cite[Proposition~1]{BS76}).

\subsection{Labelling of the vertices
of the diamond   $D_{n,k}$.}\label{S:DescrD}

Recall that we stated the formal definition of multi-branching diamond graphs (diamonds)  $D_{n,k}$ in the Introduction (Definition~\ref{D:BranchDiam}).
In this section we describe a system of labels for their vertices that we will use in the proof of Theorem~\ref{T:Main}.

First note, that there are two standard normalizations for the
shortest-path metric on diamonds: in one of them each edge has
length $1$, and in the other each edge of $D_{n,k}$ has length
(weight) $2^{-n}$, so that the distance between the top and and
the bottom vertex is equal to $1$. We shall use the $2^{-n}$
normalization of diamond graphs. Observe that in this
normalization the natural embedding of $D_{n,k}$ into $D_{n+1,k}$
is isometric.

We will call one of the vertices of $D_{0,k}$  the {\it bottom} and the
other the {\it top}. We define the  {\it
bottom}  and the {\it top} of $D_{n,k}$ as vertices which evolved from
the bottom and the top of $D_{0,k}$, respectively. A {\it subdiamond} of
$D_{n,k}$ is a subgraph which evolved from an edge of some
$D_{m,k}$ for $0\le m\le n$. The {\it top} and {\it bottom of a
subdiamond} of $D_{n,k}$  are defined as the vertices of the
subdiamond which are the closest   to the top and bottom of
$D_{n,k}$, respectively. The {\it height of the subdiamond} is the
distance between its top and its bottom.

We will say that a  vertex of $D_{n,k}$  is at the {\it level}
$\la$, if its distance from the bottom vertex is equal to $\la$.
Then $B_n\DEF\{\frac{t}{2^n} : 0\le t \le 2^n\}$ is the set of all
possible levels. For each $\lambda\in B_n$ we consider its dyadic
expansion
\begin{equation}\lb{dyadic}
\la=\sum_{\al=0}^{s(\la)} \frac{\la_\al}{2^\al},
\end{equation}
where $0\le s(\la)\le n$, $\la_\al\in\{0,1\}$ for each
$\al\in\{0,\dots,s(\la)-1\}$, and
$\la_{s(\lambda)}=1$ for all $\lambda\ne 0$. We will use the convention
$s(0)=0$. Note that $1\in B_n$ is the only value of $\la\in B_n$
with $\la_0\ne 0$.

We will label each vertex of the diamond $D_{n,k}$  by its level
$\la$, and by an ordered $s(\la)$-tuple of numbers from the set
$\{1,\dots,k\}$. We will refer to this  $s(\la)$-tuple of numbers
as  the  {\it  label of the branch of the vertex}. We define
labels inductively on the value of $s(\la)$ of the level  $\la$ of
vertices,  as follows, cf. Figure~\ref{F:diamond-label}:
\begin{itemize}

\item $s(\la)=0$: The bottom vertex is labelled $v^{(n)}_0$, and the top
vertex is labelled $v^{(n)}_1$.

\item   $s(\la)=1$: There are $k$ vertices at the level $\frac12$, and they are
labelled by $v^{(n)}_{\frac12,j_1}$, where $j_1\in \{1,\dots,k\}$ is the label of the path in $D_{1,k}$ (see Definition~\ref{D:BranchDiam}) to which the vertex $v^{(n)}_{\frac12,j_1}$
belongs.

\item $s(\la)=l+1$, where $1\le l<n$: Suppose that for
all $\mu\in B_n$ with $s(\mu)\le l$, all vertices at level $\mu$
have been labelled, and
let $\la\in B_n$ be such that $s(\la)=l+1$.

Then $\la_{l+1}=1$, and  there exist unique values $\kappa,\mu\in
B_n$ with $s(\kappa)<s(\mu)= l$, and $\e\in\{1,-1\}$ so that
\begin{equation*}\lb{level-lb}
\la=\kappa + \e\frac{1}{2^{l+1}}=\mu-\e\frac{1}{2^{l+1}}.
\end{equation*}

If a vertex $v$ of the diamond $D_{n,k}$ is at the level $\la$, then
there exist a unique vertex $u$ at the level $\kappa$, and a
unique vertex $w$ at the level $\mu$, so that
$d(u,w)=\frac{1}{2^l}$ and
\begin{equation}\lb{level-dist}
d(v,u)=d(v,w)=\frac{1}{2^{l+1}}.
\end{equation}
Note also that if the vertices $u$ at the level $\kappa$, and $w$
at the level $\mu$, are such that $d(u,w)=\frac{1}{2^l}$, then
there are exactly $k$ vertices in $D_{n,k}$ satisfying
\eqref{level-dist}. These $k$ vertices will be labelled by
$v^{(n)}_{\la;
j_1,\dots,j_{s(\mu)},j_{s(\la)}}$,\label{P:DescrLabel} where
$j_{s(\la)}\in \{1,\dots,k\}$, and  $(j_1,\dots,j_{s(\mu)})$ is  the label of the branch of $w$, i.e. $w=v^{(n)}_{\mu;
j_1,\dots,j_{s(\mu)}}$. Note that $s(\la)=s(\mu)+1$. Moreover, in
the situation described above $u=v^{(n)}_{\kappa;
j_1,\dots,j_{s(\kappa)}}$, where $( j_1,\dots,j_{s(\kappa)})$ is
an initial segment of $(j_1,\dots,j_{s(\mu)})$.
\end{itemize}

\begin{figure}[h]
\centering
\includegraphics[scale=0.7]{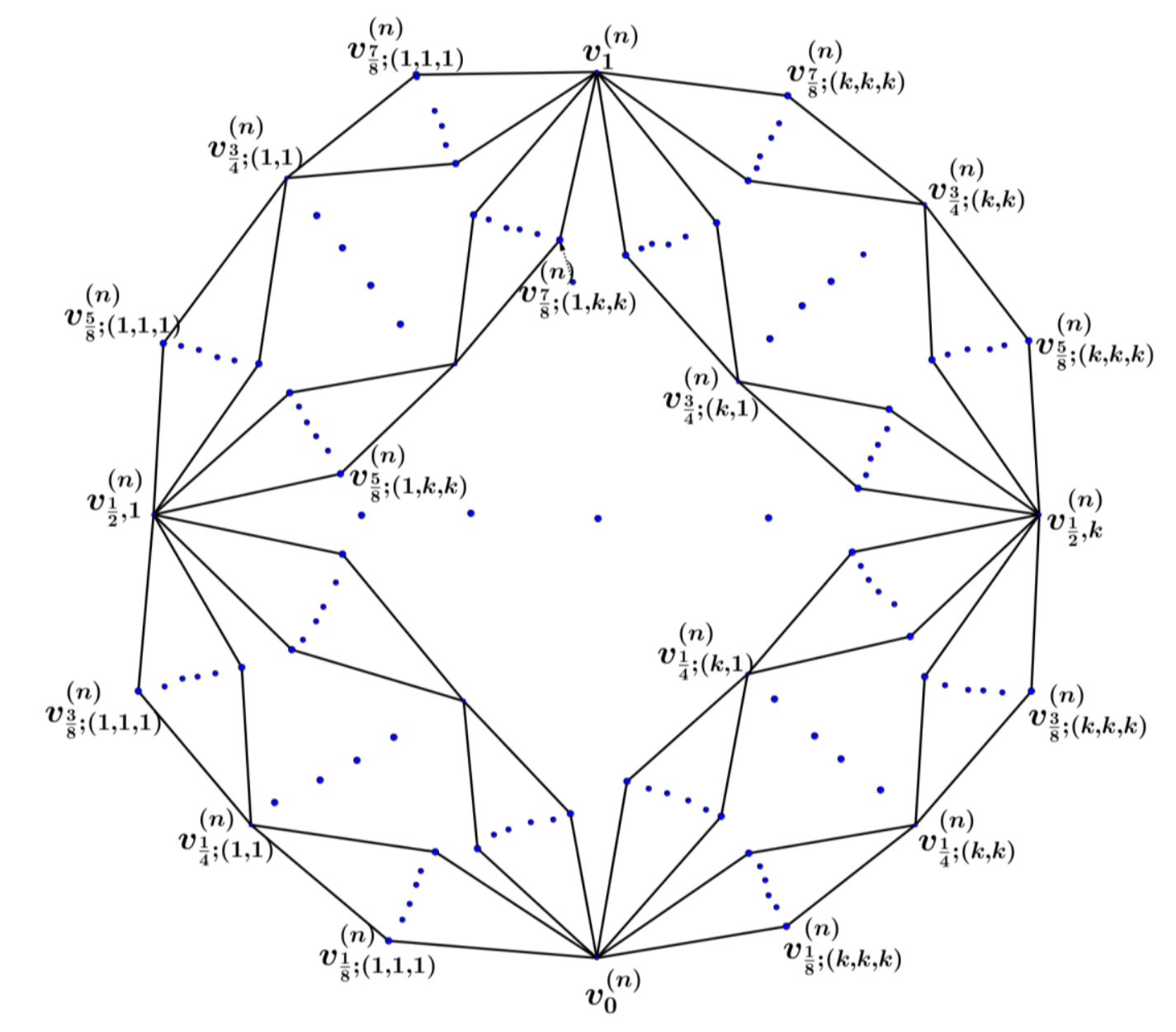}
\caption{Labelling of the diamond}\label{F:diamond-label}
\end{figure}

The following observations are easy consequences of our method of
labelling of vertices:

\begin{observation}\label{O:vertex}
If two vertices  are connected by an edge  in  $D_{n,k}$, then  the  absolute value of the difference between their
 levels is equal to $2^{-n}$. In particular the distance between two vertices that are connected by an edge is equal to  the  absolute value of the difference between their
 levels.
\end{observation}

 The last statement is generalized in the next observation.

\begin{observation}\label{O:DistTopBot}
For all $\mu,\la\in B_n$, with $s(\mu)<s(\la)$, and for every
$s(\la)$-tuple \[(j_1,\dots,j_{s(\la)}),\] there exists a geodesic
path in $D_{n,k}$ that connects the bottom and the top vertex of
$D_{n,k}$ and passes through both  the vertices
$v^{(n)}_{\mu;j_1,\dots,j_{s(\mu)}}$ and
$v^{(n)}_{\la;j_1,\dots,j_{s(\la)}}$. Thus
\begin{equation*}
d_{D_{n,k}}(v^{(n)}_{\mu;j_1,\dots,j_{s(\mu)}},
v^{(n)}_{\la;j_1,\dots,j_{s(\la)}})=|\lambda-\mu|.
\end{equation*}

In particular, the distance from any vertex in any subdiamond of $D_{n,k}$ to the bottom or the
top of the subdiamond is equal to the absolute value of the difference between the
corresponding levels.
\end{observation}

\begin{observation}\label{O:NestSubdiam} For every $\la\in B_n$ with $\la\ne 1$, and every $\tau\in\{0,\dots,s(\la)\}$, the vertex $v=v^{(n)}_{\la;
j_1,\dots,j_{s(\la)}}\in D_{n,k}$ belongs to the subdiamond
$\Sigma_\tau(v)$ of height $2^{-\tau}$ uniquely determined by its
bottom and top as follows, cf. Figure~\ref{F:diamond-example}:

\begin{itemize}

\item The bottom is a vertex at the level
$R_{\tau}(\la)\DEF\sum_{\al=0}^\tau(\lambda_\alpha/2^{\alpha})$
labelled by the corresponding initial segment of the label of $v$,
namely $(j_1,\dots,j_{s(R_{\tau}(\la))})$

\item The top is a vertex at the level $R_{\tau}(\la)+2^{-\tau}$,
labelled by the corresponding initial segment of the label of $v$,
namely $(j_1,\dots,j_{s(R_{\tau}(\la)+2^{-\tau})})$

\end{itemize}

The subdiamonds  $\Sigma_\tau(v)$ form a nested sequence in the
following sense:
\[ \Sigma_0(v)\supset \Sigma_1(v)\supset \dots\supset \Sigma_{s(\la)}(v),\]
and if $\la\ne 1$, then $v$ is the bottom vertex of
$\Sigma_{s(\la)}(v)$.
\end{observation}

\begin{figure}[h]
\centering
\includegraphics[scale=0.9]{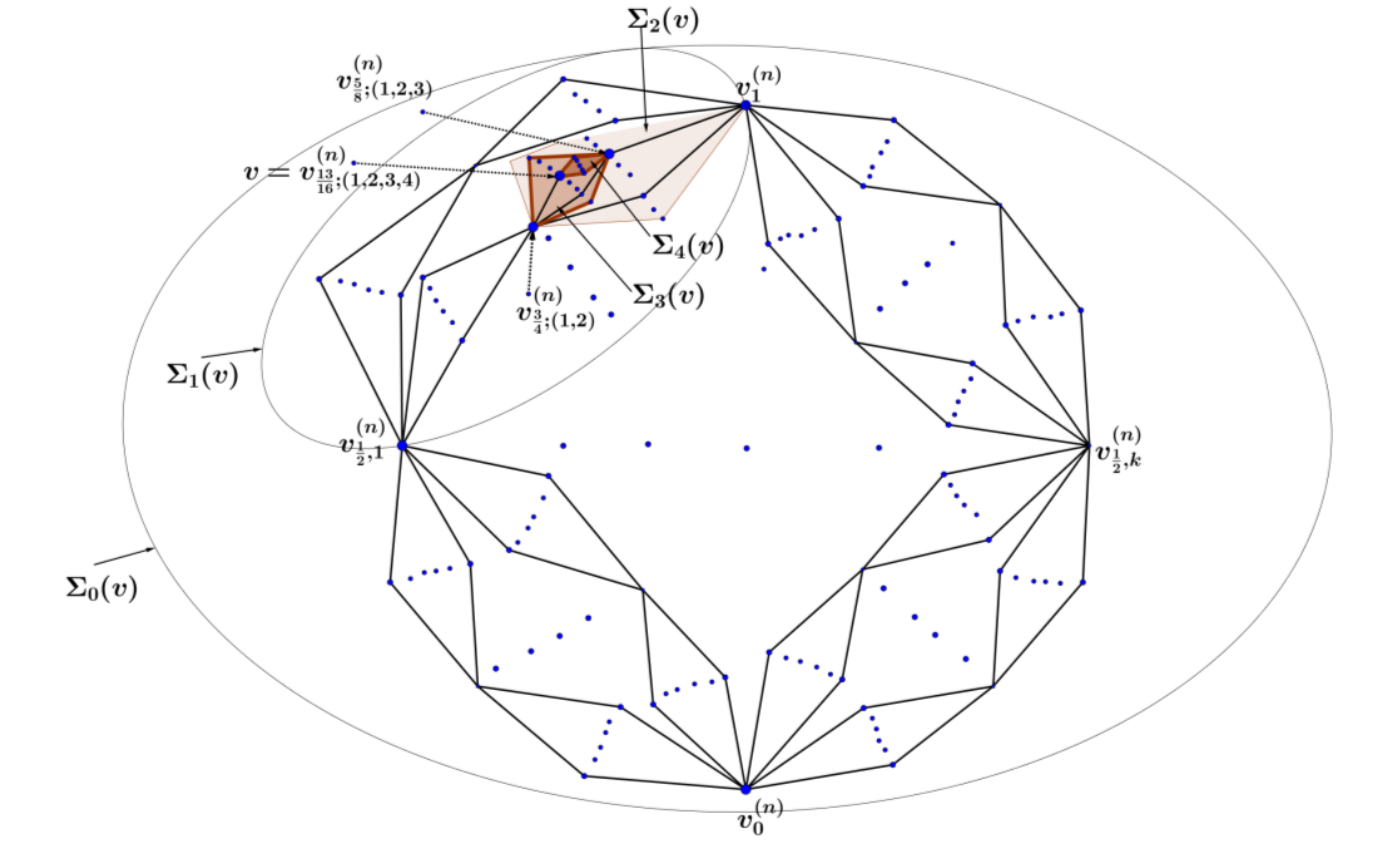}
\caption{Subdiamonds in Observation~\ref{O:NestSubdiam}, cf. Example~\ref{ex-subd}.}\label{F:diamond-example}
\end{figure}

\begin{example} \lb{ex-subd}
If $\la= \frac12+\frac14+\frac1{16}$, and
$v=v^{(n)}_{\la,1,2,3,4}$, then bottom and top vertices of
$\Sigma_\tau(v)$ are, respectively, see Figure~\ref{ex-subd}:

-- in $\Sigma_0(v)$: $v^{(n)}_0$ and $v^{(n)}_1$,

-- in $\Sigma_1(v)$: $v^{(n)}_{\frac12,1}$ and $v^{(n)}_1$,

-- in $\Sigma_2(v)$: $v^{(n)}_{\frac12+\frac14,1,2}$ and $v^{(n)}_1$,

--  in $\Sigma_3(v)$: $v^{(n)}_{\frac12+\frac14,1,2}$ and $v^{(n)}_{\frac12+\frac14+\frac18,1,2,3}$\ ,

--  in  $\Sigma_4(v)$: $v^{(n)}_{\frac12+\frac14+\frac1{16},1,2,3,4}=v$ and $v^{(n)}_{\frac12+\frac14+\frac18,1,2,3}$\ .
\end{example}

\section{Embedding diamonds into  spaces with an \ESA\ basis -- Proof of
Theorem~\ref{T:Main}}\label{S:D1k}

By Theorem~\ref{T:BS}    of Brunel and Sucheston, in order to prove
Theorem~\ref{T:Main} it suffices to find, for each
$n,k\in\mathbb{N}$, a bilipschitz   embedding with distortion at most $8$  of
$D_{n,k}$ into an arbitrary Banach space with an ESA basis.  This section is devoted to a construction of such  embeddings.

Recall that a {\it metric midpoint} (or a {\it midpoint}) between
points $u$ and $v$ in a metric space $(Y,d_Y)$ is a point $w\in Y$
so that $d_Y(u,w)=d_Y(w,v)=\frac12d_Y(u,v)$.

We note that the diamond $D_{n,k}$ has numerous midpoints between
many pairs of   points, in particular there are $k$ midpoints
between the top and the bottom vertex,    $k$ midpoints between
the top and each vertex at level $\frac12$,  $k$ midpoints between
the bottom and each vertex at level $\frac12$, and so on. In fact
the recursive  construction of the diamond $D_{n,k}$ can be viewed
as adding  $k$ midpoints between every pair of existing points
that are connected by an edge. For this reason, to construct an
embedding  of $D_{n,k}$ into a Banach space $X$, we need to
develop a method of constructing elements in $X$ that have
multiple well-separated metric midpoints that themselves also have
multiple well-separated metric midpoints, and so on, for several
iterations. The general construction is rather technical, so we
prefer to start with the simple but, hopefully, illuminating case
of $n=1$. It is worth mentioning that one can easily find a
bilipschitz embedding with distortion $\le 2+\ep$ of $D_{1,k}$
into an arbitrary infinite-dimensional Banach space (including
superreflexive spaces). The usefulness of the construction
described below is in the existence of a suitable iteration, that
leads to a low-distortion embedding of $D_{n,k}$.

\subsection{Warmup: Embedding of $D_{1,k}$ into  spaces with an \ESA\
basis}\label{S:1,k}

Recall that $D_{1,k}$ consists of the bottom vertex, the top
vertex at distance $1$
from the bottom vertex, and $k$ midpoints between the top and the
bottom, that are at distance $1$ from each other.

We will work with finitely supported elements  of $X$, whose
coefficients in their basis representations are $0$ and $\pm 1$.
We shall write $+1$ as $+$ and $-1$ as $-$.

First we consider an element $h=e_1+e_2-e_3-e_4$, i.e.
$h=(++--00\dots).$ To simplify the notation, we omit brackets and
0's that appear at the end of sequences of coefficients for basis
expansions of every finitely supported element in $X$, i.e. we
write
$$h=++--.$$

Since the basis  $\{e_i\}_{i=1}^\infty$ is ESA (recall that by
Lemma~\ref{lemmaESA},    ESA is equivalent to SA, and implies IS),
we conclude that
\begin{equation}\label{E:Norm_h}\|h\|=2\|+-\|,\end{equation} and, by IS  of the basis,  the elements
$$ h_+= 0+-0, \ \ \ h_-=+00-$$
are both metric midpoints between $h$ and $0$. Further, we have
$$\|h_+-h_-\|=\|-+-+\|\stackrel{\rm SA}{\ge}\|-+\|\stackrel{\eqref{E:Norm_h}}{=}\frac12\|h\|.$$

Thus there are two well-separated metric midpoints between $h$ and
$0$. We can use $h$ and the ESA property of the basis to construct
an element in $X$ such that there are $M$ well-separated metric
midpoints between this element and $0$, where $M$ is any natural
number.

Indeed, consider an element $x_1$ equal to the sum of $2^M$
shifted disjoint copies of $h$, i.e., if  $S$ denotes the shift
operator on the basis ($Se_i\DEF e_{i+1}$),
\begin{equation*}
\begin{split}
x_1&=\sum_{\nu=0}^{2^M-1} S^{4\nu}(h)\\
&=++--++--++--++--....++--.
\end{split}
\end{equation*}

 By IS and ESA of the basis we have
\begin{equation*}
\begin{split}
\|x_{1}\|&=2\|\sum_{\nu=0}^{2^M-1} S^{2\nu}(e_1-e_2)\|\\
&=2\|\underbrace{+-+-....+-}_{2^M {\text {\small\  pairs}}}\|.
\end{split}
\end{equation*}

Let $r_1,\dots,r_M$ denote the (natural analogues of) Rademacher
functions on $\{0,\dots,2^M-1\}$. We assume that $M\ge k$ and define the element $m_j$,
$j=1,\dots,k$, as the sum of $2^M$ disjoint blocks, where each
block $++--$ of $x_1$ is replaced either by $0+-0$, i.e. by $h_+$,
if the corresponding value of $r_j$ is $1$, or by $+00-$, i.e. by
$h_-$, if the corresponding value of $r_j$ is $-1$, that is ($h_+$
and $h_-$ are defined above)
$$m_j=\sum_{\nu=0}^{2^M-1} S^{4\nu}(h_{r_j(\nu)}).  $$

Since, for all $1\le j\le M$, each block of $m_j$ contains exactly
one $+$ and one $-$, and the position of the $+$ is always before
$-$, by IS and SA of the basis we have
\begin{equation*}
\begin{split}
\|m_j\|&=\|\sum_{\nu=0}^{2^M-1} S^{4\nu}(e_1-e_2)\|=\frac12
\|x_{1}\|.
\end{split}
\end{equation*}

The same estimate holds for $x_1-m_j$, and so
$\|x_1-m_j\|=\frac12\|x_1\|$. Thus, for all $1\le j\le k$, the
vector $m_j$ is a metric midpoint between $x_1$ and $0$.\medskip

To compute the distance between different midpoints $m_i$ and
$m_j$, we note that $h_+-h_-=-+-+$. Since $i\ne j$, for half of
the values of $\nu$, we have $r_i(\nu)=r_j(\nu)$. For these values
of $\nu$, the $\nu$-th block in $m_i-m_j$ is $0000$. For one
quarter of values of $\nu$, we have $r_i(\nu)=1, r_j(\nu)=-1$. For
these values the block is $h_+-h_-=-+-+$. For the remaining one
quarter of values of $\nu$, we have $r_i(\nu)=-1, r_j(\nu)=1$, and
the block becomes $+-+-$. By SA of the basis, we can replace all
blocks $+-+-$ by $0000$ without increasing the norm. Thus, by IS
and SA of the basis we obtain
\begin{equation*}
\begin{split}
\|m_i-m_j\|&=\|\sum_{\nu=0}^{2^M-1} S^{4\nu}(h_{r_i(\nu)}-h_{r_j(\nu)})\|\\
&\ge\|\sum_{\nu=0}^{2^{M-2}-1} S^{4\nu}(h_+-h_-)\|\\
&=\|\sum_{\nu=0}^{2^{M-1}-1} S^{2\nu}(-e_1+e_2)\|\\
&\ge\frac14\|x_1\|,
\end{split}
\end{equation*}
where the last inequality follows from the triangle inequality,
since, by IS,  for every $N\in\bbN$,
\begin{equation}\lb{triangle}
\begin{split}
\|\sum_{\nu=0}^{N-1} S^{2\nu}(e_1-e_2)\|&=\frac12\left[\|\sum_{\nu=0}^{N-1} S^{2\nu}(e_1-e_2)\|+\|\sum_{\nu=N}^{2N-1} S^{2\nu}(e_1-e_2)\|\right] \\
&\ge\frac12\,\|\sum_{\nu=0}^{2N-1} S^{2\nu}(e_1-e_2)\|.
\end{split}
\end{equation}

Thus the metric midpoints $\{m_i\}_{i=1}^k$ between $x_1$ and $0$
are well-separated, and therefore the embedding of the diamond
$D_{1,k}$ into $X$ that sends the bottom vertex to $0$, the top
vertex to $x_1$, and the $k$ vertices at the level $\frac12$  of
$D_{1,k}$ to the $k$ midpoints $\{m_i\}_{i=1}^k$, has distortion
at most $4$.

The most important feature of this construction is that it can be
iterated without large increase of distortion, as we demonstrate
below.

\subsection{Description of the embedding of $D_{n,k}$ into a space with an \ESA\
basis}\label{S:n,k}

Our next goal is to define a low-distortion embedding of $D_{n,k}$
into a space $X$ with an \ESA\ basis $\{e_i\}_{i=1}^\infty$.
We want to find an element in $X$, that we will denote by $x^{(n)}_1$, that has
 at least $k$ well-separated (exact) metric midpoints, with an additional property that the selected $k$ metric midpoints of $x^{(n)}_1$
have a structure sufficiently similar to the element $x^{(n)}_1$,
so that the procedure of selecting $k$ good well-separated metric
midpoints can be iterated $n$ times, cf. Remark~\ref{rem-idea}
below. To achieve this goal we will  generalize    the
construction described in Section~\ref{S:1,k}.
\medskip

We shall continue using the notation $+$ for $+1$ and $-$ for $-1$
(with the hope that in each case it will be clear from the context whether we use
this convention or we use $+$ and $-$ to denote algebraic
operations). We define the element $h^{(n)}$, by
\begin{equation*}\label{def-block}
\begin{split}
h^{(n)}&=\sum_{l=1}^{2^n}e_l-\sum_{l=2^n+1}^{2^{n+1}}e_l\\
&=\underbrace{+\dots+}_{2^n}\underbrace{-\dots-}_{2^n}.
\end{split}
\end{equation*}

The element  $h^{(n)}$ is supported on the interval $[1,2^{n+1}]$.
We denote the support of the positive part of  $h^{(n)}$, that is,
the interval $[1,2^{n}]$, by $I^{(n)}$. Note that $\card(I^{(n)})=2^n$.

We will denote by ${\tt Ref}_n$ the reflection about the center of the interval, on the interval $[1,2^{n+1}]$, that is, for $j\in [1,2^{n+1}]$,
$${\tt Ref}_n(j)\DEF 2^{n+1}-j+1.$$

Note that, in this notation,
$h^{(n)}=\one_{I^{(n)}}-\one_{\RR(I^{(n)})},$ where $\one_A$
denotes the indicator function of the set $A$.

We define $h^{(n)}_+$ and
$h^{(n)}_-$ by
\begin{equation*}\label{h+}
h^{(n)}_+=\underbrace{0\dots0}_{2^{n-1}}\underbrace{+\dots+}_{2^{n-1}\atop{I^{(n)}_{+}}}\underbrace{-\dots-}_{2^{n-1}\atop{\RR(I^{(n)}_{+})}}\underbrace{0\dots0}_{2^{n-1}},
\end{equation*}
and
\begin{equation*}\label{h-}
h^{(n)}_-=\underbrace{+\dots+}_{2^{n-1}\atop{I^{(n)}_{-}}}\underbrace{0\dots0}_{2^{n-1}}\underbrace{0\dots0}_{2^{n-1}}\underbrace{-\dots-}_{2^{n-1}\atop{\RR(I^{(n)}_{-})}}.
\end{equation*}
We  denote the supports of the positive parts of $h^{(n)}_+$ and
$h^{(n)}_-$ by $I^{(n)}_+$ and $I^{(n)}_-$, respectively. Note
that intervals $I^{(n)}_+$ and $I^{(n)}_-$ are disjoint, are
contained in $I^{(n)}$,
$\card(I^{(n)}_+)=\card(I^{(n)}_-)=2^{n-1}$,  and the interval
$I^{(n)}_-$ precedes the interval $I^{(n)}_+$, i.e. the right
endpoint of $I^{(n)}_-$ is less than the left endpoint of
$I^{(n)}_+$. Moreover
$$h^{(n)}_+=\one_{I^{(n)}_+}-\one_{\RR(I^{(n)}_+)},\ \ \ h^{(n)}_-=\one_{I^{(n)}_-}-\one_{\RR(I^{(n)}_-)}.$$

Clearly, $h^{(n)}=h^{(n)}_++h^{(n)}_-$. Note that by IS and ESA
of the basis, we have
\begin{equation*}\label{normh+}
\big\|h^{(n)}_+\big\|=\big\|h^{(n)}_-\big\|=\frac12\big\|h^{(n)}\big\|=2^{n-1}\|e_1-e_2\|.
\end{equation*}

For any $1<\al\le n$, and $\e_i=\pm1$, for $1\le i\le\al$, if $h^{(n)}_{\e_1,\dots,\e_{\al-1}}$ is already defined,    $I^{(n)}_{\e_1,\dots,\e_{\al-1}}$ denotes the support of the positive part  of $h^{(n)}_{\e_1,\dots,\e_{\al-1}}$, and $h^{(n)}_{\e_1,\dots,\e_{\al-1}}=\one_{I^{(n)}_{\e_1,\dots,\e_{\al-1}}}- \one_{\RR(I^{(n)}_{\e_1,\dots,\e_{\al-1}})}$, we define $I^{(n)}_{\e_1,\dots,\e_{\al-1},+}$ to be the subinterval consisting  of $2^{n-\al}$ largest coordinates of $I^{(n)}_{\e_1,\dots,\e_{\al-1}}$, and we define
 \begin{equation*}\label{defh++}
 h^{(n)}_{\e_1,\dots,\e_{\al-1},+}\DEF  \one_{I^{(n)}_{\e_1,\dots,\e_{\al-1},+}}- \one_{\RR(I^{(n)}_{\e_1,\dots,\e_{\al-1},+})}.
 \end{equation*}

 We define $I^{(n)}_{\e_1,\dots,\e_{\al-1},-}=I^{(n)}_{\e_1,\dots,\e_{\al-1}}\setminus I^{(n)}_{\e_1,\dots,\e_{\al-1},+}$, and
 \begin{equation*}
 \begin{split}
 h^{(n)}_{\e_1,\dots,\e_{\al-1},-}&\DEF \one_{I^{(n)}_{\e_1,\dots,\e_{\al-1},-}}- \one_{\RR(I^{(n)}_{\e_1,\dots,\e_{\al-1},-})}\\
 &= h^{(n)}_{\e_1,\dots,\e_{\al-1}}-h^{(n)}_{\e_1,\dots,\e_{\al-1},+}.
 \end{split}
 \end{equation*}
Thus  the supports of  $h^{(n)}_{\e_1,\dots,\e_{\al-1},+}$ and
$h^{(n)}_{\e_1,\dots,\e_{\al-1},-}$ are disjoint, have the same
cardinality ($=2^{n-\al}$), and their union is equal to  the
support of  $h^{(n)}_{\e_1,\dots,\e_{\al-1}}$. In other words, the
intervals $I^{(n)}_{\e_1,\dots,\e_{\al-1},+}$ and
$I^{(n)}_{\e_1,\dots,\e_{\al-1},-}$ are disjoint, are contained in
$I^{(n)}_{\e_1,\dots,\e_{\al-1}}$,
 the interval $I^{(n)}_{\e_1,\dots,\e_{\al-1},-}$ precedes the interval $I^{(n)}_{\e_1,\dots,\e_{\al-1},+}$,  and
\begin{equation}\lb{cardI}
 \begin{split}
 \card(I^{(n)}_{\e_1,\dots,\e_{\al-1},+})=\card(I^{(n)}_{\e_1,\dots,\e_{\al-1},-})=\frac12\card(I^{(n)}_{\e_1,\dots,\e_{\al-1}})=2^{n-\al}.
 \end{split}
 \end{equation}

We see the following pattern
\begin{align*}\label{h++}
h^{(n)}_{++}&=\underbrace{0\dots\dots\dots\dots0}_{2^{n-1}}\underbrace{0\dots0}_{ {2^{n-2}}}
\underbrace{+\dots+}_{2^{n-2}\atop{I^{(n)}_{++}}}\underbrace{-\dots-}_{2^{n-2}\atop{\RR(I^{(n)}_{++})}}\underbrace{0\dots0}_{2^{n-2}}\underbrace{0\dots\dots\dots\dots0}_{2^{n-1}},\\
h^{(n)}_{+-}&=\underbrace{0\dots\dots\dots\dots0}_{2^{n-1}}\underbrace{+\dots+}_{2^{n-2}\atop{I^{(n)}_{+-}}}\underbrace{0\dots0}_{2^{n-2}}\underbrace{0\dots0}_{2^{n-2}}
\underbrace{-\dots-}_{2^{n-2}\atop{\RR(I^{(n)}_{+-})}}\underbrace{0\dots\dots\dots\dots0}_{2^{n-1}},\\
h^{(n)}_{-+}&=\underbrace{0\dots0}_{2^{n-2}}\underbrace{+\dots+}_{2^{n-2}\atop{I^{(n)}_{-+}}}\underbrace{0\dots\dots\dots\dots0}_{2^{n-1}}\underbrace{0\dots\dots\dots\dots0}_{2^{n-1}}\underbrace{-\dots-}_{2^{n-2}\atop{\RR(I^{(n)}_{-+})}}\underbrace{0\dots0}_{2^{n-2}},\\
h^{(n)}_{--}&=\underbrace{+\dots+}_{2^{n-2}\atop{I^{(n)}_{--}}}\underbrace{0\dots0}_{2^{n-2}}\underbrace{0\dots\dots\dots\dots0}_{2^{n-1}}\underbrace{0\dots\dots\dots\dots0}_{2^{n-1}}\underbrace{0\dots0}_{2^{n-2}}\underbrace{-\dots-}_{2^{n-2}\atop{\RR(I^{(n)}_{--})}},
\end{align*}
and so on.

By IS and ESA of the
basis, we have, for all $\al=1,\dots,n$, and all
$\{\e_i\}_{i=1}^\al\in\{-1,1\}^\al$,
\begin{equation*}\label{normh+al}
\big\|h^{(n)}_{\e_1,\dots,\e_\al}\big\|=\frac12\big\|h^{(n)}_{\e_1,\dots,\e_{\al-1}}\big\|=\frac1{2^\al}\big\|h^{(n)}\big\|=2^{n-\al}\|e_1-e_2\|.
\end{equation*}

Moreover we have:
\begin{observation}\label{O:OnSupp} The supports of any two vectors
$h^{(n)}_{\e_1,\dots,\e_{\al}}$ and
$h^{(n)}_{\theta_1,\dots,\theta_{\beta}}$ are either contained one
in the other or are disjoint. The support of
$h^{(n)}_{\e_1,\dots,\e_{\al}}$ is contained in the support of
$h^{(n)}_{\theta_1,\dots,\theta_{\beta}}$ if and only if the
string ${\theta_1,\dots,\theta_{\beta}}$ is the initial part of
the string ${\e_1,\dots,\e_{\al}}$. In this case the vector
$h^{(n)}_{\e_1,\dots,\e_{\al}}$ can be regarded as a
coordinate-wise  product
$h^{(n)}_{\theta_1,\dots,\theta_{\beta}}\cdot
\1_{\supp(h^{(n)}_{\e_1,\dots,\e_{\al}})}$.

Let us emphasize that the statement above implies that
$h^{(n)}_{\e_1,\dots,\e_{\al}}$ and
$h^{(n)}_{\theta_1,\dots,\theta_{\beta}}$ are disjointly supported
if and only if there is $\gamma\le\min\{\alpha,\beta\}$ such that
$\e_{\gamma}=-\theta_\gamma$.
\end{observation}

Finally, let $\mathcal{P}$ be the set of all tuples
$(j_1,\dots,j_s)$ of all lengths between $1$ and $n$, where each
$j_i$ is in $\{1,\dots,k\}$, that is, $\mathcal{P}$ is the set of
all labels of branches in the diamond $D_{n,k}$. We will denote
the cardinality of $\mathcal{P}$ by $M$, that is
  \begin{equation*}\lb{defM}
  M\DEF \card({\mathcal{P}})= k+k^2+\dots+k^n.
  \end{equation*}
\medskip

For $A\in\mathcal{P}$, let $r_A$  be the (natural analogues of)
Rademacher functions on $\{0,\dots,2^{M}-1\}$.

\subsection{Definition of the map}

Now we are ready to define a bilipschitz embedding of $D_{n,k}$
into $X$. We shall denote the image of $v^{(n)}_{\la;
j_1,\dots,j_{s(\la)}}$ in $X$ by $x^{(n)}_{\la;
j_1,\dots,j_{s(\la)}}$.

We define the image of the bottom vertex $v_0^{(n)}$ of  $D_{n,k}$
to be zero (that is, $x_0^{(n)}=0)$, and the image of the top
vertex to be the element $x^{(n)}_1$ that is defined as the sum
of $2^M$ disjoint shifted copies of $h^{(n)}$, more precisely,
\begin{equation*}
x^{(n)}_1=\sum_{\nu=0}^{2^{M}-1} S^{2^{n+1}\nu}(h^{(n)}),
\end{equation*}
where, as above, $S$ denotes the shift operator (i.e. $Se_i\DEF e_{i+1}$).

Note that, by IS and ESA  of the basis we have
\begin{equation}\lb{norm-top}
\begin{split}
\|x^{(n)}_1\|&=2^n\Big\|\sum_{\nu=0}^{2^{M}-1}
S^{2^{n+1}\nu}(e_1-e_2)\Big\|=2^n\Big\|\sum_{\nu=0}^{2^{M}-1}
S^{2\nu}(e_1-e_2)\Big\|.
\end{split}
\end{equation}

We will use the notation $S^{2^{n+1}\nu}[1, 2^{n+1}]$,
$S^{2^{n+1}\nu}I^{(n)}$,
$S^{2^{n+1}\nu}I^{(n)}_{\e_1,\dots,\e_{\al}}$ for the
 shifts of the sets $[1, 2^{n+1}]$, $I^{(n)}$,
$I^{(n)}_{\e_1,\dots,\e_{\al}}$, respectively. We will use the term {\it $\nu$-th block}, or {\it block number
$\nu$}, for the restriction of any of the considered vectors to
$S^{2^{n+1}\nu}[1, 2^{n+1}]$.

\begin{remark} \lb{rem-idea}
Our main reason for choosing this $x^{(n)}_1$ is that
 there are many well-separated (exact)  metric midpoints between $0$ and  $x^{(n)}_1$. Namely, when in each block we  replace  $h^{(n)}$ by  either $h^{(n)}_+$ or $h^{(n)}_-$, for all possible choices, we obtain an element in $X$ that
is a metric midpoint between $0$ and $x^{(n)}_1$. Further, if we
use the values of the  Rademacher functions at $\nu\in
\{0,\dots,2^{M}-1\}$, to decide the choice of  $h^{(n)}_+$ or
$h^{(n)}_-$ for the $\nu$-th block, then, by the independence of
the Rademacher functions, we will be able to estimate the distance
between metric midpoints determined by different Rademacher
functions,  similarly as in
 Section~\ref{S:1,k}.    Moreover, each midpoint obtained this way in every block has
 entries structurally very similar to the elements $h^{(n-1)}$. This is vitally important for us, because this structure, together with the ESA property of the basis, will allow us to iterate this procedure $n$ times to obtain the embedding of the diamond $D_{n,k}$.
  We will make this  precise below.
\end{remark}

Since our definition of the map (on vertices different from the
top and the bottom) is rather complicated, we decided to give it
both as an inductive procedure  and as an explicit formula.

\subsubsection{Inductive form of the definition}
\label{S:VerbDef}

Our definition of the map is such that each vector $x^{(n)}_{\la;
j_1,\dots,j_{s(\la)}}$ satisfies the following conditions:
\begin{enumerate}

\item It is a $\{0,+1,-1\}$-valued vector.

\item Its support is contained in the set
$\bigcup_{\nu=0}^{2^M-1}S^{2^{n+1}\nu}[1, 2^{n+1}]$.

\item\label{I:PosNeg} The set $P=P^{(n)}_{\la;
j_1,\dots,j_{s(\la)}}$ of coordinates where the value of  $x^{(n)}_{\la;
j_1,\dots,j_{s(\la)}}$ is equal to $1$ is contained in
$\bigcup_{\nu=0}^{2^M-1}S^{2^{n+1}\nu}I^{(n)}$, and the set  of coordinates with values equal to $-1$
 is contained in the
complement of $\bigcup_{\nu=0}^{2^M-1}S^{2^{n+1}\nu}I^{(n)}$.

\item\label{I:Symm} The values of the element $x^{(n)}_{\la;
j_1,\dots,j_{s(\la)}}$ on the  set
\[\left(\bigcup_{\nu=0}^{2^M-1}S^{2^{n+1}\nu}[1,
2^{n+1}]\right)\setminus \left(\bigcup_{\nu=0}^{2^M-1}S^{2^{n+1}\nu}I^{(n)}\right)\]
are uniquely determined by its values on the set
$\bigcup_{\nu=0}^{2^M-1}S^{2^{n+1}\nu} I^{(n)}$.

Namely: for each
 $j\in S^{2^{n+1}\nu}([1, 2^{n+1}]\setminus
I^{(n)})$ the value on the $j$-th coordinate of $x^{(n)}_{\la;
j_1,\dots,j_{s(\la)}}$ is equal to  $(-1)$ times    the value on the
$(2^{n+1}(\nu+1)-j+2^{n+1}\nu+1)$-th coordinate of $x^{(n)}_{\la;
j_1,\dots,j_{s(\la)}}$ (by definition, this property is clearly satisfied by $x_1^{(n)}$). Intuitively this property says that the negative part of each block of
$x^{(n)}_{\la; j_1,\dots,j_{s(\la)}}$ can be obtained from the
positive part by the composition of  the   negation and the symmetric reflection about the
center of the block.

That is,  for each $\nu$, if in the $\nu$-th block $P(\nu)\DEF
P^{(n)}_{\la; j_1,\dots,j_{s(\la)}}\cap S^{2^{n+1}\nu}I^{(n)}$
then we have
\begin{equation}\lb{I:Symm:E}
x^{(n)}_{\la; j_1,\dots,j_{s(\la)}}\cdot \one_{S^{2^{n+1}\nu}[1,
2^{n+1}]}=\one_{P(\nu)}-\one_{\RRnu(P(\nu))},
\end{equation}
where $\RRnu$ is the symmetric reflection of the $\nu$-th block about the
center of the block, that is for every $j\in  S^{2^{n+1}\nu}[1, 2^{n+1}]$,  $\RRnu(j)\DEF2^{n+1}(\nu+1)-j+2^{n+1}\nu+1$.
\end{enumerate}

By the properties in items \ref{I:PosNeg} and \ref{I:Symm},  the
restriction of the vector $x^{(n)}_{\la; j_1,\dots,j_{s(\la)}}$ to
$S^{2^{n+1}\nu}I^{(n)}$ is a $\{0,1\}$-valued vector and
$x^{(n)}_{\la; j_1,\dots,j_{s(\la)}}$ is completely determined by
all such restrictions. Therefore it is enough to define the set
$P(\nu)=P^{(n)}_{\la; j_1,\dots,j_{s(\la)}}\cap
S^{2^{n+1}\nu}I^{(n)}$, i.e.  the part of the support of
$x^{(n)}_{\la; j_1,\dots,j_{s(\la)}}$ that is contained in
$S^{2^{n+1}\nu}I^{(n)}$, for each $\nu\in\{0,\dots,2^M-1\}$.

For all $\la\in B_n$, $(j_1,\dots,j_{s(\la)})\in \PP$, and
$\nu\in\{0,\dots,2^M-1\}$,  we define the  set $P(\nu)\DEF
P^{(n)}_{\la; j_1,\dots,j_{s(\la)}}\cap S^{2^{n+1}\nu}I^{(n)}$
through the following finite inductive procedure.

We use the notation ${\disp
\lambda=\sum_{\alpha=0}^{s(\lambda)}\frac{\lambda_\alpha}{2^{\al}}}$,
for the binary decomposition of $\la$,  and, for all
$\al\in\{1,\dots,s(\la)\}$, we denote by
$J_\al\DEF(j_1,\dots,j_\al)$, i.e. $J_\al$ is the initial segment
of length $\al$ of the   $s(\la)$-tuple
$(j_1,\dots,j_{s(\la)})$ that labels the branch of the vertex
$v^{(n)}_{\la; j_1,\dots,j_{s(\la)}}$.
\begin{enumerate}
\item (Initial Step) If  $\la_0=1$, we let $P(\nu)\DEF C_0\DEF
S^{2^{n+1}\nu}I^{(n)}$ and STOP.

It is clear that this happens if and only if the vertex is
$v_1^{(n)}$,  $\la=1$, and $s(\la)=0$. Notice that in this case we
have $\card(P(\nu))=\card(C_0)=2^n\la=2^{n-0}\la_0$.

Otherwise, that is, if $s(\la)>0$ (and $\lambda_0=0$), we set
$\al=1$
 $C_0=\emptyset$, (note that  $\card(C_0)=0=2^{n-0}\la_0$)
and go to Step \ref{I:2}.

\item\label{I:2} (Inductive step) Suppose that the following are given:
$\al\ge 1$,  a set $C_{\al-1}\subseteq S^{2^{n+1}\nu}I^{(n)}$ with $\card(C_{\al-1})=\sum_{i=0}^{\al-1}  2^{n-i}\la_i$, and numbers
$\e_1,\dots,\e_{\al-1}\in\{-1,1\}$  such that
\begin{equation}\lb{disjoint-interval}
C_{\al-1}\cap S^{2^{n+1}\nu}I^{(n)}_{\e_1,\dots,\e_{\al-1}}=\emptyset.
\end{equation}
(If $\alpha=1$, we mean that
$I^{(n)}_{\e_1,\dots,\e_{\al-1}}=I^{(n)}$.)

Then
\begin{enumerate}
\item If $\lambda_\al=1$, we set
\begin{equation*}\lb{defC}
C_\al\DEF C_{\al-1}\cup
S^{2^{n+1}\nu}I^{(n)}_{\e_1,\dots,\e_{\al-1},r_{J_\al}(\nu)}.
\end{equation*}

Note that $S^{2^{n+1}\nu}I^{(n)}_{\e_1,\dots,\e_{\al-1},r_{J_\al}(\nu)}\subseteq
S^{2^{n+1}\nu}I^{(n)}_{\e_1,\dots,\e_{\al-1}}$, and thus,
by \eqref{disjoint-interval} and \eqref{cardI}, we have
\begin{equation}\lb{cardC}
\begin{split}
\card(C_\al)&=\card(C_{\al-1})+2^{n-\al}=\sum_{i=0}^{\al}  2^{n-i}\la_i.
\end{split}
\end{equation}
\begin{enumerate}
\item If $\al=s(\lambda)$ we set $P(\nu)=C_\al$ and STOP.

\item\label{I:ii} If $\al<s(\lambda)$ we set
$\e_{\al}=-r_{J_\al}(\nu)$. Since the intervals
$I^{(n)}_{\e_1,\dots,\e_{\al-1},+}$ and
$I^{(n)}_{\e_1,\dots,\e_{\al-1},-}$ are disjoint (see
\eqref{cardI} and the paragraph immediately preceding it), we see
that, in this case, \eqref{disjoint-interval} holds when $\al-1$
is replaced by $\al$. Therefore we can go back to the beginning of
the inductive step for $\al+1$.
\end{enumerate}
\item If $\lambda_\al=0$ (and thus, necessarily, $\al<s(\lambda)$)
we define $C_\al\DEF C_{\al-1}$, and  $\e_{\al}=r_{J_\al}(\nu)$.
Then
\begin{equation*}
\begin{split}
\card(C_\al)&=\card(C_{\al-1})+2^{n-\al}\cdot 0=\sum_{i=0}^{\al}  2^{n-i}\la_i,
\end{split}
\end{equation*}
and, since $I^{(n)}_{\e_1,\dots,\e_{\al-1},\e_\al}\subseteq I^{(n)}_{\e_1,\dots,\e_{\al-1}}$, we see that also in this case, \eqref{disjoint-interval} holds when $\al-1$ is replaced by $\al$. Therefore we can go back to the beginning of the inductive step for $\al+1$.
\end{enumerate}
\end{enumerate}

\begin{observation} \lb{card-support}
Observe that the above inductive procedure will stop precisely  when $\al=s(\lambda)$, and thus, by \eqref{cardC}, we have
\begin{equation}\lb{cardP}
\begin{split}
\card(P(\nu))&=\sum_{i=0}^{s(\la)}  2^{n-i}\la_i=2^{n}\la.
\end{split}
\end{equation}

Moreover, the inductive procedure is defined in such a way, that for every
$\al\le s(\la)$, we have
\begin{equation}\lb{suppalpha}
\begin{split}
P(\nu)&\subseteq C_\al \cup
S^{2^{n+1}\nu}I^{(n)}_{\e_1,\dots,\e_\al}.
\end{split}
\end{equation}
\end{observation}

\begin{observation} \lb{card-edge}
Observe that if two vertices are joined by an
edge in $D_{n,k}$, then one of them has the form $v^{(n)}_{\la;
j_1,\dots,j_{n}}$, where
$\la=\sum_{\alpha=0}^n\frac{\la_\alpha}{2^{\alpha}}$ is such that
$\lambda_n=1$ (i.e. $s(\la)=n$); and the other vertex has the form
 $v^{(n)}_{\mu;
j_1,\dots,j_{s(\mu)}}$, where $|\mu-\lambda|=2^{-n}$, and
$(j_1,\dots,j_{s(\mu)})$ is the initial segment of the label of the branch of the vertex
$v^{(n)}_{\la;
j_1,\dots,j_{n}}$.

Assume $\la>\mu$. If we follow the definition above for these
vertices we see that the positive support of the difference
$x^{(n)}_{\la; j_1,\dots,j_{n}}-x^{(n)}_{\mu;
j_1,\dots,j_{s(\mu)}}$ in the $\nu$-th block  is an interval of
length $1$. The same holds in the case when $\la<\mu$
and we subtract the vectors in the opposite order.

Therefore, by the \ESA\ property of the basis, for endpoints of every edge in
$D_{n,k}$,  we get
\begin{equation}\label{E:DistEdge}
\begin{split}
\left\|x^{(n)}_{\la; j_1,\dots,j_{n}}-x^{(n)}_{\mu;
j_1,\dots,j_{s(\mu)}}\right\|&=2^{-n}\|x_1^{(n)}\|\\&=
\|x_1^{(n)}\|\cdot d_{D_{n,k}}
\left(v^{(n)}_{\la; j_1,\dots,j_{n}},v^{(n)}_{\mu;
j_1,\dots,j_{s(\mu)}}\right).
\end{split}
\end{equation}

Since the metric in $D_{n,k}$ is the shortest path distance,
the equality \eqref{E:DistEdge} implies that for
any two vertices in $D_{n,k}$ we have:
\begin{equation}\label{E:Above}
\|x^{(n)}_{\la; j_1,\dots,j_{s(\la)}}-x^{(n)}_{\mu;
i_1,\dots,i_{s(\mu)}}\|\le
\|x_1^{(n)}\|\cdot  d_{D_{n,k}}\left(v^{(n)}_{\la;
j_1,\dots,j_{s(\la)}}, v^{(n)}_{\mu;
i_1,\dots,i_{s(\mu)}}\right).
\end{equation}
By \eqref{E:DistEdge} and \eqref{E:Above},  our map is Lipschitz
with constant $\|x_1^{(n)}\|$.
\end{observation}

\subsubsection{The formula for the map}

The described above inductive procedure leads to the following formula for
 $x^{(n)}_{\la;
j_1,\dots,j_{s(\la)}}$, where $\la\notin\{0,1\}$, and
$\lambda=\sum_{\alpha=1}^{s(\lambda)}\lambda_\alpha2^{-\alpha}$
is  the binary
representation of $\lambda$:
\begin{equation}\label{E:DirDef}
\begin{split}
 x^{(n)}_{\la;j_1,\dots,j_{s(\la)}} =\sum_{\nu=0}^{2^{M}-1}
S^{2^{n+1}\nu}&\left(\sum_{\alpha=1}^{s(\lambda)}\lambda_\alpha
h^{(n)}_{\theta(\lambda,j_1,\dots,j_\alpha,\nu)}\right),
\end{split}
\end{equation}
where, for each $\al\le s(\la)$, $\theta(\lambda, j_1,\dots,j_\alpha,\nu)$
is an $\alpha$-tuple of $\pm1$'s defined by
\begin{equation*}\label{def-theta}
\begin{split}
\theta(\lambda,j_1,\dots,j_\alpha,\nu)&
\\=\Big((-1)^{\la_1}&r_{(j_1)}(\nu),\dots, (-1)^{\la_{\al-1}}
r_{(j_1,\dots, j_{\alpha-1})}(\nu),
r_{(j_1,\dots,j_\alpha)}(\nu)\Big).
\end{split}
\end{equation*}
In the case when $\alpha=1$, we mean that for all
$\la\notin\{0,1\}$,
$\theta(\la,j_1,\nu)=r_{(j_1)}(\nu).$

Note that the  $\alpha$-tuples $\theta(\lambda,
j_1,\dots,j_\alpha,\nu)$ are defined in such a way that whenever
$\la_\al\ne 0$, and   $(j_1,\dots,j_\alpha)$ is an initial segment
of $(j_1,\dots,j_{\bar{\alpha}})$, then, for every $\nu$, the
elements $h^{(n)}_{\theta(\lambda,j_1,\dots,j_\alpha,\nu)}$ and
$h^{(n)}_{\theta(\lambda,j_1,\dots,j_{\bar{\alpha}},\nu)}$ are
disjoint.

\subsection{An estimate for the distortion}\label{S:EstDistor}

Since, by \eqref{E:DistEdge} and \eqref{E:Above}, our mapping is
Lipschitz with the Lipschitz constant equal to $\|x_1^{(n)}\|$, it
remains to prove that there exists $K\le 8$, so that, for all
$v^{(n)}_{\la; j_1,\dots,j_{s(\la)}}$, $v^{(n)}_{\mu;
i_1,\dots,i_{s(\mu)}}$ in $D_{n,k}$,
\begin{equation}\label{E:below}
\|x^{(n)}_{\la; j_1,\dots,j_{s(\la)}}-x^{(n)}_{\mu;
i_1,\dots,i_{s(\mu)}}\|\ge\frac{\|x_1^{(n)}\|}K\,
d_{D_{n,k}}\left(v^{(n)}_{\la; j_1,\dots,j_{s(\la)}},
v^{(n)}_{\mu; i_1,\dots,i_{s(\mu)}}\right),
\end{equation}
where $\lambda$ and $\mu$ have the   binary
decompositions,
$\lambda=\sum_{\alpha=0}^{s(\lambda)}2^{-\alpha}\lambda_\alpha$,
and $\mu=\sum_{\alpha=0}^{s(\mu)}2^{-\alpha}\mu_\alpha$, respectively.

To estimate the distortion of the embedding we will simultaneously derive  the  formulas for the distances between
vertices in $D_{n,k}$, and the estimates for the distances between their images.

First,  observe that, by Observation~\ref{O:DistTopBot}, if $v_\mu^{(n)}$ is the bottom or the top vertex of the diamond $D_{n,k}$, i.e. if $\mu\in\{0,1\}$, then for every vertex $v^{(n)}_{\la; j_1,\dots,j_{s(\la)}}$, with $\la\ne\mu$, we have
\begin{equation*}\label{E:DistCase0}
d_{D_{n,k}}\left(v^{(n)}_{\la; j_1,\dots,j_{s(\la)}},
v^{(n)}_{\mu}\right)=|\lambda-\mu|.
\end{equation*}

On the other hand, by Observation~\ref{card-support}, by
\eqref{norm-top}, and by IS and ESA of the basis we get, when
$\mu=0$,
\begin{equation}\lb{bottom-v}
\begin{split}
\|x^{(n)}_{\la; j_1,\dots,j_{s(\la)}}-0\|&=2^n\la\Big\|\sum_{\nu=0}^{2^{M}-1}
S^{2^{n+1}\nu}(e_1-e_2)\Big\|=\la\|x^{(n)}_1\|\\
&=\|x^{(n)}_1\|\cdot d_{D_{n,k}}\left(v^{(n)}_{\la; j_1,\dots,j_{s(\la)}},
v^{(n)}_{\mu}\right),
\end{split}
\end{equation}
and, when $\mu=1$,
\begin{equation}\lb{top-v}
\begin{split}
\|x^{(n)}_1-x^{(n)}_{\la; j_1,\dots,j_{s(\la)}}\|&=2^n(1-\la)\Big\|\sum_{\nu=0}^{2^{M}-1}
S^{2^{n+1}\nu}(e_1-e_2)\Big\|=(1-\la)\|x^{(n)}_1\|\\
&=\|x^{(n)}_1\|\cdot d_{D_{n,k}}\left(v^{(n)}_{\la; j_1,\dots,j_{s(\la)}},
v^{(n)}_{\mu}\right).
\end{split}
\end{equation}

Thus, when at least one of the vertices is the bottom or the top
vertex of the diamond $D_{n,k}$, inequality \eqref{E:below} holds
with $K=1$.

We will say that a path in $D_{n,k}$ is a {\it direct vertical
path} if it is a subpath of a geodesic path that connects  the
bottom and the top vertex  in $D_{n,k}$.

Next, suppose that  distinct vertices $v^{(n)}_{\la;
j_1,\dots,j_{s(\la)}}$ and $v^{(n)}_{\mu;i_1,\dots,i_{s(\mu)}}$
are connected by a direct vertical path. Then   $\la\ne\mu$, say
$\la>\mu$. By the triangle inequality, \eqref{bottom-v}, and
\eqref{top-v} we obtain
\begin{equation*}
\begin{split}
\|x^{(n)}_{\la; j_1,\dots,j_{s(\la)}}-x^{(n)}_{\mu;i_1,\dots,i_{s(\mu)}}\|
&\ge \|x^{(n)}_{1}-x^{(n)}_{0}\|-\|x^{(n)}_{\mu;i_1,\dots,i_{s(\mu)}}-x^{(n)}_{0}\|-\|x^{(n)}_{1}-x^{(n)}_{\la; j_1,\dots,j_{s(\la)}}\|\\
&=\|x^{(n)}_1\|\left(1-\mu-(1-\la)\right)\\
&=|\la-\mu|\cdot\|x^{(n)}_1\|.
\end{split}
\end{equation*}
Thus, by Observation~\ref{O:DistTopBot} and the upper estimate \eqref{E:Above}, we get
\begin{equation}\lb{vertical}
\begin{split}
\|x^{(n)}_{\la; j_1,\dots,j_{s(\la)}}-x^{(n)}_{\mu;i_1,\dots,i_{s(\mu)}}\|
&=|\la-\mu|\cdot\|x^{(n)}_1\|\\
&=\|x^{(n)}_1\|\cdot d_{D_{n,k}}\left(v^{(n)}_{\la; j_1,\dots,j_{s(\la)}},
v^{(n)}_{\mu;i_1,\dots,i_{s(\mu)}}\right).
\end{split}
\end{equation}

Therefore, whenever  vertices $v^{(n)}_{\la;
j_1,\dots,j_{s(\la)}}$ and $v^{(n)}_{\mu;i_1,\dots,i_{s(\mu)}}$
are on a direct vertical path, then \eqref{E:below} holds with
$K=1$ (thus we think of our  embedding as {\it vertically
isometric with the multiplicative constant $\|x^{(n)}_1\|$}, that
is  every pair of vertices $u,v$ of $D_{n,k}$ connected by a
direct vertical path is mapped onto a pair of points in the space
$X$ with  distance equal to the original distance
$d_{D_{n,k}}(u,v)$ multiplied by the constant $\|x^{(n)}_1\|$).

In general, we consider  two different vertices $v^{(n)}_{\la; j_1,\dots,j_{s(\la)}}$ and $v^{(n)}_{\mu;
i_1,\dots,i_{s(\mu)}}$   in $D_{n,k}$, with $\la,\mu\notin\{0,1\}$.
We define the set $B=\{\al\le\min\{s(\la),s(\mu)\}: i_\al\ne j_\al\hbox{ or
}\lambda_\alpha\ne\mu_\alpha\}$, and
\begin{equation*}
\be=\begin{cases}\min B,\ \ &{\text {\rm if}}\ \ B\ne \emptyset,\\
\min\{s(\la),s(\mu)\}+1,  &{\text {\rm if}}\ \ B= \emptyset.
\end{cases}
\end{equation*}

If $B\ne \emptyset$, we define $\delta$ to be the largest
integer that does not exceed $\beta-1$ and is such that either
$\lambda_\delta=\mu_\de=1$ or $\delta=0$ (observe that, since  $\la,\mu\notin\{0,1\}$,
we have $\lambda_0=\mu_0=0$), and we define
$\omega=\sum_{\al=0}^{\delta} \frac{\lambda_\al}{2^\al}=\sum_{\al=0}^{\delta} \frac{\mu_\al}{2^\al}$. We
consider the vertex $v^{(n)}_{\omega;
j_1,\dots,j_{\delta}}=v^{(n)}_{\omega; i_1,\dots,i_{\delta}}$, or
the vertex $v_0^{(n)}$ if  $\om=0$ (note that $\om=0$ \wtw \ $\de=0$).

In the remainder of this argument we will  denote the vertices
$v^{(n)}_{\la; j_1,\dots,j_{s(\la)}}$, $v^{(n)}_{\mu;
i_1,\dots,i_{s(\mu)}}$, and $v^{(n)}_{\omega;
j_1,\dots,j_{\delta}}$ by $v_\la$, $v_\mu$, and $v_\omega$,
respectively, and the corresponding images in $X$, by
$x_\la$, $x_\mu$, and $x_\omega$, respectively. For all $\nu\in\{0,\dots,2^M\}$,  we will also use $P_\la(\nu)$, $P_\mu(\nu)$, and $P_\omega(\nu)$, to denote
the subsets of the $\nu$-th block, where the coordinates of the elements $x_\la$, $x_\mu$, and $x_\omega$, respectively, are equal to 1.

Note that   $v_\om$ is the vertex at the highest possible level so
that there exist direct vertical paths passing through $v_\om$ and
connecting the bottom of $D_{n,k}$ to $v_\la$ and $v_\mu$,
respectively. In particular, $v_\om$
 is connected to both $v_\la$ and $v_\mu$ by  direct vertical paths (that are disjoint with the exception of the vertex $v_\om$).

There are several cases to consider:

\begin{enumerate}

\item\label{I:C0} $B= \emptyset$.

\item\label{I:C1} $B\ne \emptyset$, $i_\beta=j_\beta$, and $\la_\beta\ne\mu_\beta$.

\item \label{I:C2a} $B\ne \emptyset$,  $i_\beta\ne j_\beta$, and
$\la_\beta=\mu_\beta=0$.

\item \label{I:C2b} $B\ne \emptyset$,   $i_\beta\ne j_\beta$, and
$\la_\beta=\mu_\beta=1$.

\item \label{I:C3} $B\ne \emptyset$, $i_\beta\ne j_\beta$ and
$\la_\beta\ne\mu_\beta$.
\end{enumerate}

As a part of the proof below, we  will analyze the geometric
meaning of each case.

\noindent{\bf Case~\ref{I:C0}:} $B= \emptyset$.

Since the vertices $v_\la$ and $v_\mu$ are distinct, the condition
$B= \emptyset$ implies that $s(\la)\ne s(\mu)$, say $s(\mu)<
s(\la)$, and $(i_1,\dots,i_{s(\mu)})$ is an initial segment of
$(j_1,\dots,j_{s(\la)})$, that is, the vertices $v_\mu$ and
$v_\la$ are connected by a direct vertical path. Thus  in
Case~\ref{I:C0}, by \eqref{vertical}, inequality \eqref{E:below}
holds with $K=1$.

\noindent{\bf Case~\ref{I:C1}:} $B\ne \emptyset$, $i_\beta=j_\beta$, and $\la_\beta\ne\mu_\beta$.

 Without loss of generality we may and do assume that $\lambda_\beta=1$
and $\mu_\beta=0$. The definitions of $\be$ and $\de$, together
with $\la_\be=1$ and  $\mu_\beta=0$, imply that then
$R_\be(\la)=\om+2^{-\be}$ and $R_\be(\mu)=\om$. Thus, by
Observation~\ref{O:NestSubdiam}, $v_\mu$ belongs to the subdiamond
$\Sigma_{{\beta}}(v_\mu)$, of height $2^{-\beta}$, with the bottom
at $v_\omega$ and the top at $v_{\omega+2^{-\beta};
i_1,\dots,i_\beta}=v_{\omega+2^{-\beta}; j_1,\dots,j_\beta}$.
Moreover, $v_\lambda$ belongs to the subdiamond
$\Sigma_{\beta}(v_\la)$ of height $2^{-\beta}$ whose bottom is at
$v_{\omega+2^{-\beta}; i_1,\dots,i_\beta} =v_{\omega+2^{-\beta};
j_1,\dots,j_\beta}$. Therefore, by Observation~\ref{O:DistTopBot},
we get that $ d_{D_{n,k}}(v_{\la}, v_{\mu})=|\lambda-\mu|, $ and
that  the vertices $v_\mu$ and $v_\la$ are on a direct vertical
path. Thus in Case~\ref{I:C1}, by \eqref{vertical}, inequality
\eqref{E:below} holds with $K=1$.

\noindent{\bf Case~\ref{I:C2a}:}  $B\ne \emptyset$,  $i_\beta\ne j_\beta$, and
$\la_\beta=\mu_\beta=0$.

 In this case $R_\be(\la)=R_\be(\mu)=\om$, and, by
 Observation~\ref{O:NestSubdiam},  the vertices
$v_\mu$ and $v_\lambda$ are in two different subdiamonds of height
$2^{-\beta}$ both with the bottom at $v_\omega$, and since
$\la_\beta=\mu_\beta=0$,  the distance of each of them to $v_\om$
is less than  $2^{-\beta}$, cf. Figure~\ref{F:Case3-4}. Since the smallest subdiamond that contains both
$v_\mu$ and $v_\lambda$ has height $2^{-(\beta-1)}$, the
  shortest path joining $v_\lambda$ and $v_\mu$
passes through $v_\omega$. By Observation~\ref{O:DistTopBot}, the
length of this path is $(\lambda-\om)+(\mu-\omega)$, so
\begin{equation}\label{E:DistCase2A}
d_{D_{n,k}}(v_{\la},
v_{\mu})
=(\lambda-\om)+(\mu-\omega)=
\sum_{\al=\beta+1}^{s(\la)}\frac{\la_\al}{2^\al}+\sum_{\al=\beta+1}^{s(\mu)}\frac{\mu_\al}{2^\al}.
\end{equation}

\begin{figure}[h]
\hspace{-10mm}\includegraphics[scale=0.8]{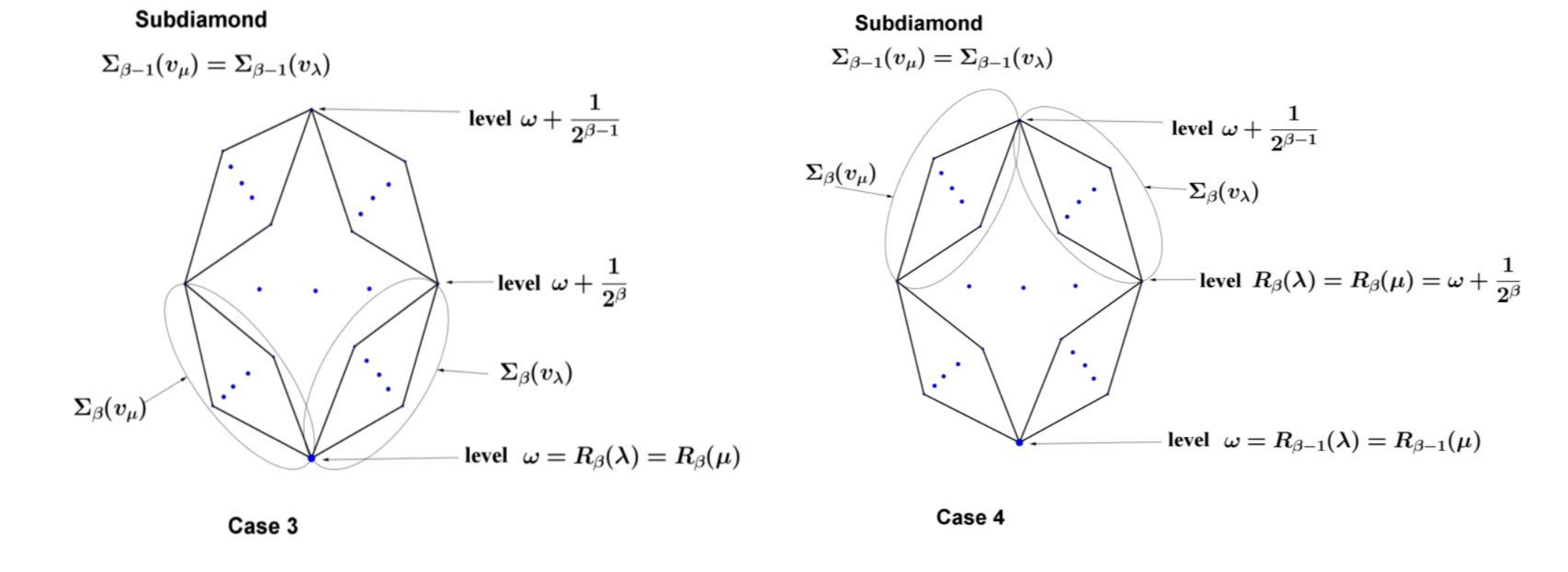}
\caption{Subdiamonds in Case 3 and Case 4.}\label{F:Case3-4}
\end{figure}

In  Case~\ref{I:C2a}, the relative position of the sets
 $P_\la(\nu)$ and $P_\mu(\nu)$ does depend on $\nu$ or, more precisely, on the values of $r_{(j_1,\dots,j_\beta)}(\nu)$ and
 $r_{(i_1,\dots,i_\beta)}(\nu)$.

We suppose, \buoo, that $\la\ge \mu$.

 Let $G$ be the set  consisting  of all
 $\nu$'s for which $r_{(j_1,\dots,j_\beta)}(\nu)=-1$
and $r_{(i_1,\dots,i_\beta)}(\nu)=1$. Note that, by the
independence of the  Rademacher functions, the cardinality the set
$G$ is equal to one fourth of the cardinality of the set of all
$\nu$'s, that is to $2^{M-2}$.

By the SA property of the basis, and since \eqref{I:Symm:E}
implies  that  the sum of all coordinates of $x_\la$ and of
$x_\mu$ in every block is equal to zero, we can replace all
entries in any selected blocks of the
 element $x_\la-x_\mu$ by zeros, without increasing the norm, in particular, we
 have
 \begin{equation}\lb{SAcor}
\begin{split}
\|x_{\la}-x_{\mu}\|&=\Big\|\sum_{\nu=0}^{2^{M}-1} (x_{\la}-x_{\mu})\cdot
\one_{S^{2^{n+1}\nu}[1,2^{n+1}]}\Big\|\\
&\ge \Big\|\sum_{\nu\in G} (x_{\la}-x_{\mu})\cdot
\one_{S^{2^{n+1}\nu}[1,2^{n+1}]}\Big\|.
\end{split}
\end{equation}

Hence we now concentrate on the form of the element $x_\la-x_\mu$  in the blocks whose numbers belong to the set $G$.
By the inductive definition of the sets
 $P_\la(\nu)$, $P_\mu(\nu)$, and $P_\om(\nu)$, we see that, for all $\nu$,
 $P_\om(\nu)\subseteq P_\mu(\nu)\cap P_\la(\nu),$
 and that for every $\nu\in G$, the sets $P_\mu(\nu)\setminus P_\om(\nu)$ and
 $P_\la(\nu)\setminus P_\om(\nu)$ are disjoint. Moreover, for every $\nu\in G$,
since $\lambda_\beta=\mu_\beta=0$, by \eqref{suppalpha} and the
definition of $G$, we have
$$P_\la(\nu)\setminus P_\om(\nu)\subseteq S^{2^{n+1}\nu}I^{(n)}_{\e_1,\dots,\e_{\beta-1},-1} \ \ {\text{\rm and}}\ \ P_\mu(\nu)\setminus  P_\om(\nu)\subseteq S^{2^{n+1}\nu}I^{(n)}_{\e_1,\dots,\e_{\beta-1},1},$$
where, by the definition of $\be$, the numbers $\e_1,\dots,\e_{\beta-1}\in\{-1,1\}$
are the same for both $x^{(n)}_{\la; j_1,\dots,j_{s(\la)}}$ and $x^{(n)}_{\mu; i_1,\dots,i_{s(\mu)}}$.

Therefore, by \eqref{I:Symm:E}, for every $\nu$ we have
\begin{equation*}
 \begin{split}
(x_{\la}&-
x_{\mu})\cdot  \one_{S^{2^{n+1}\nu}[1, 2^{n+1}]}=\\
&\Big(\one_{P_\la(\nu)\setminus P_\om(\nu)}-
\one_{P_\mu(\nu)\setminus P_\om(\nu)}\Big)-
\Big(\one_{\RRnu(P_\la(\nu)\setminus P_\om(\nu))}-
\one_{\RRnu(P_\mu(\nu)\setminus P_\om(\nu))}\Big).
\end{split}
\end{equation*}

Thus, and by \eqref{cardP}, if we consider the restriction of the
difference $x_\lambda-x_\mu$ to the interval
$S^{2^{n+1}\nu}[1,2^{n+1}]$ and omit all zeros, we get a vector of the
following form: first it will have $2^n(\la-\om)$ entries with
values equal to $+1$, then it will have $2^n(\mu-\om)$ entries
equal to $-1$, then it will have $2^n(\mu-\om)$ entries equal to
$+1$, and finally it will have $2^n(\la-\om)$ entries equal to
$-1$:
\begin{align}\lb{form-diff}
\underbrace{+\cdots \dots+}_{2^{n}(\la-\om)}
\underbrace{-\cdots\dots-}_{2^{n}(\mu-\om)}\underbrace{+\cdots\dots+}_{2^{n}(\mu-\om)}\underbrace{-\cdots
\dots-}_{2^{n}(\la-\om)}.
\end{align}

Recall that we assumed that $\lambda\ge\mu$. For each $\nu\in G$,
we will replace by zeros the values  on the coordinates of
$(x_{\la}-x_{\mu})$  in the smallest subinterval of
$S^{2^{n+1}\nu}[1, 2^{n+1}]$ that contains the set
$(P_\mu(\nu)\setminus P_\om(\nu))\cup   \RRnu(P_\mu(\nu)\setminus
P_\om(\nu))$ (the ``central'' set in the diagram
\eqref{form-diff}). Since the sum of all values of  the
coordinates of $(x_{\la}-x_{\mu})$ on this interval is equal to 0,
by the SA  property of the basis, this replacement does not
increase the norm of the element. Thus,
  by \eqref{SAcor} and the ESA property of the basis,   we get
\begin{equation*}\lb{case2a1}
\begin{split}
\|x_{\la}-x_{\mu}\|&\ge \Big\|\sum_{\nu\in G} (x_{\la}-x_{\mu})\cdot
\one_{S^{2^{n+1}\nu}[1,2^{n+1}]}\Big\|\\
&\ge \Big\|\sum_{\nu\in G}
\Big(\one_{P_\la(\nu)\setminus P_\om(\nu)}-
\one_{\RRnu(P_\la(\nu)\setminus P_\om(\nu))}\Big)\Big\|\\
&\stackrel{\rm by \eqref{cardP}}{=}2^n(\la-\om)\Big\|\sum_{\nu=0}^{2^{M-2}-1}
S^{2\nu}(e_1-e_2)\Big\|\\
&\stackrel{(\ast)}{\ge}\frac14(\la-\om)\|x_1^{(n)}\|\\
&\stackrel{(\ast\ast)}{\ge} \frac18\Big((\la-\om)+(\mu-\om)\Big)\|x_1^{(n)}\|\\
&\stackrel{\rm \eqref{E:DistCase2A}}{=}
\frac18\|x_1^{(n)}\|d_{D_{n,k}}(v_\la,v_\mu),
\end{split}
\end{equation*}
where the inequality $(\ast)$ holds by \eqref{norm-top}, and by an application of the triangle inequality similarly as in \eqref{triangle}, and the  inequality $(\ast\ast)$
 holds since $(\mu-\om)\le  (\la-\om)$.

Hence, in Case~\ref{I:C2a}, \eqref{E:below} holds with $K=8$.

\noindent{\bf Case \ref{I:C2b}:} $B\ne \emptyset$,   $i_\beta\ne j_\beta$, and
$\la_\beta=\mu_\beta=1$.

In this case, we also have   $R_{\be}(\la)=R_{\be}(\mu)$, but this
common value is greater than $\om$. Since $\la_\beta=\mu_\beta=1$,
we have that
$s(R_{\be}(\la)+2^{-\be})=s(R_{\be}(\mu)+2^{-\be})<\be$, and thus
by Observation~\ref{O:NestSubdiam},  the subdiamonds
$\Sigma_\be(v_\mu)$ and $\Sigma_\be(v_\la)$, both of height
$2^{-\beta}$, have the same top vertex at the level
$R_\be(\la)+\frac{1}{2^{\be}}=R_\be(\mu)+\frac{1}{2^{\be}}$, and
on the branch labelled by an initial segment of
$(j_1,\dots,j_{\be-1})=(i_1,\dots,i_{\be-1})$, cf.
Figure~\ref{F:Case3-4}.   Note that the distance from this vertex
to either $v_\la$ and $v_\mu$ is smaller than or equal to
$2^{-\beta}$. The bottom vertices of  the subdiamonds
$\Sigma_\be(v_\mu)$ and $\Sigma_\be(v_\la)$ are
$v^{(n)}_{R_{\be}(\mu),j_1,\dots,j_{\be}}$ and
$v^{(n)}_{R_{\be}(\la),i_1,\dots,i_{\be}}$, respectively, that are
different vertices (at the same level) since $i_\beta\ne j_\beta$.
Thus the smallest subdiamond that contains both  $v_\mu$ and
$v_\lambda$ has height $2^{-(\beta-1)}$. Hence there exists
  a shortest path joining $v_\lambda$ and $v_\mu$ that
passes through the common top vertex of the subdiamonds $\Sigma_\be(v_\mu)$
and $\Sigma_\be(v_\la)$, that is
 at the level
$R_\be(\la)+\frac{1}{2^{\be}}=R_\be(\mu)+\frac{1}{2^{\be}}$, and is connected by a direct vertical path to both $v_\lambda$ and $v_\mu$. By
Observation~\ref{O:DistTopBot}, the length of this path is
$(R_{\be}(\la)+\frac{1}{2^{\be}}- \lambda)+(R_{\be}(\mu)+\frac{1}{2^{\be}}-\mu)$, so
\begin{equation}\label{E:DistCase2B}
\begin{split}
d_{D_{n,k}}(v_{\la},
v_{\mu})
&=\frac{2}{2^{\beta}}-
\left[(\lambda-R_{\be}(\la))+(\mu-R_{\be}(\mu))\right]\\
&=\left(\frac{1}{2^{\beta}}-\sum_{\al=\beta+1}^{s(\la)}\frac{\la_\al}{2^\al}\right)+\left(\frac{1}{2^{\beta}}-\sum_{\al=\beta+1}^{s(\mu)}\frac{\mu_\al}{2^\al}\right).
\end{split}
\end{equation}

As in Case~\ref{I:C2a},  \buoo, we   assume that $\la\ge \mu$, and we look first at
the set $G$ consisting of the values of $\nu$ for which $r_{(j_1,\dots,j_\beta)}(\nu)=-1$
and $r_{(i_1,\dots,i_\beta)}(\nu)=1$.

By the inductive definition of the sets $P_\la(\nu)$ and
$P_\mu(\nu)$, and by the definition of $\be$, we see that,  for
every $\nu\in G$, the set $C_{\be-1}$ (cf.
equation~\eqref{disjoint-interval}), and  the numbers
$\e_1,\dots,\e_{\beta-1}\in\{-1,1\}$, are the same for both
$x_{\la}$ and $x_{\mu}$.
  Moreover, for every $\nu\in G$,
since $\lambda_\beta=\mu_\beta=1$, by  \eqref{suppalpha} and the
definitions of $G$ and $C_\beta$, we have
\begin{align*}
&C_\be(x_\la)=
C_{\be-1}\cup S^{2^{n+1}\nu}I^{(n)}_{\e_1,\dots,\e_{\beta-1},-1}, \ \ {\text{\rm and}}\ \
C_\be(x_\mu)=
C_{\be-1}\cup S^{2^{n+1}\nu}I^{(n)}_{\e_1,\dots,\e_{\beta-1},1},\\
&P_\la(\nu)\setminus C_\be(x_\la)\subseteq S^{2^{n+1}\nu}I^{(n)}_{\e_1,\dots,\e_{\beta-1},1} \subseteq C_\be(x_\mu),\\
&P_\mu(\nu)\setminus C_\be(x_\mu)\subseteq S^{2^{n+1}\nu}I^{(n)}_{\e_1,\dots,\e_{\beta-1},-1} \subseteq C_\be(x_\la).
\end{align*}

Therefore if we omit zeros in the block number $\nu$, the
difference $x_\lambda-x_\mu$ will be nonzero on four intervals: it
starts with
$2^{n}\left(\frac1{2^\beta}-\sum_{\alpha=\beta+1}^{s(\mu)}\frac{\mu_\alpha}{2^\alpha}\right)$
entries with values equal to $+1$ (corresponding to the set
$S^{2^{n+1}\nu}I^{(n)}_{\e_1,\dots,\e_{\beta-1},-1}\setminus
P_\mu(\nu)$),
 then it will contain
$2^{n}\left(\frac1{2^\beta}-\sum_{\alpha=\beta+1}^{s(\lambda)}\frac{\lambda_\alpha}{2^\alpha}\right)$
entries with values equal to $-1$ (corresponding to the set
$S^{2^{n+1}\nu}I^{(n)}_{\e_1,\dots,\e_{\beta-1},1}\setminus
P_\la(\nu)$), then it will contain the symmetric images of the
first two sets under the reflection $\RRnu$, which consist of
$2^{n}\left(\frac1{2^\beta}-\sum_{\alpha=\beta+1}^{s(\lambda)}\frac{\lambda_\alpha}{2^\alpha}\right)$
entries equal to $+1$, and finally
$2^{n}\left(\frac1{2^\beta}-\sum_{\alpha=\beta+1}^{s(\mu)}\frac{\mu_\alpha}{2^\alpha}\right)$
entries equal to $-1$:
\begin{align*}
\underbrace{+\cdots\dots+}_{2^{n}(\frac{1}{2^{\beta}}-
(\mu-R_{\be}(\mu)))}\underbrace{-\cdots
\dots-}_{2^{n}(\frac{1}{2^{\beta}}- (\lambda-R_{\be}(\la)))}
\underbrace{+\cdots
\dots+}_{2^{n}(\frac{1}{2^{\beta}}-(\lambda-R_{\be}(\la)))}\underbrace{-\cdots\dots-}_{2^{n}(\frac{1}{2^{\beta}}-
(\mu-R_{\be}(\mu)))}.
\end{align*}

Recall that we assumed that $\lambda\ge\mu$. Similarly, as in
Case~\ref{I:C2a}, for each $\nu\in G$, we will replace by zeros
the values  on the coordinates of $(x_{\la}-x_{\mu})$  in the
smallest subinterval of $S^{2^{n+1}\nu}[1, 2^{n+1}]$ that contains
the two ``central'' sets  above, that contain
$2^{n}\left(\frac1{2^\beta}-\sum_{\alpha=\beta+1}^{s(\lambda)}\frac{\lambda_\alpha}{2^\alpha}\right)$
entries equal to $-1$, and the same amount of entries equal to
$+1$.
 Since   the sum of all replaced values  is equal to 0, by the
SA  property of the basis, this replacement does not increase the norm of the element. Thus, and by \eqref{SAcor}, we get
\begin{equation*}\lb{case2b1}
\begin{split}
\|x_{\la}-x_{\mu}\|&\ge \Big\|\sum_{\nu\in G} (x_{\la}-x_{\mu})\cdot
\one_{S^{2^{n+1}\nu}[1,2^{n+1}]}\Big\|\\
&\ge 2^{n}\left(\frac1{2^\beta}-\sum_{\alpha=\beta+1}^{s(\mu)}\frac{\mu_\alpha}{2^\alpha}\right)\Big\|\sum_{\nu\in G}
S^{2^{n+1}\nu}(e_1-e_2)\Big\|\\
&\stackrel{(\ast)}{\ge}\frac14\left(\frac1{2^\beta}-\sum_{\alpha=\beta+1}^{s(\mu)}\frac{\mu_\alpha}{2^\alpha}\right)\|x_1^{(n)}\|\\
&\stackrel{(\ast\ast)}{\ge} \frac18\|x_1^{(n)}\|d_{D_{n,k}}(v_\la,v_\mu),
\end{split}
\end{equation*}
where the inequality $(\ast)$ holds by \eqref{norm-top}, and by an application of the triangle inequality similarly as in \eqref{triangle}, and the  inequality $(\ast\ast)$
 holds by \eqref{E:DistCase2B}, since $\mu\le  \la$.

Hence, in Case~\ref{I:C2b}, \eqref{E:below} holds with $K=8$.

\noindent{\bf Case \ref{I:C3}:} $B\ne \emptyset$, $i_\beta\ne j_\beta$ and
$\la_\beta\ne\mu_\beta$.

Without loss of generality we    assume
that $\lambda_\beta=1$ and $\mu_\beta=0$.
Then $\la\ge \mu$ and $R_{\be}(\la)=R_{\be}(\mu)+\frac{1}{2^\be}$.
By the definition of $\be$, and since $\mu_\beta=0$, $R_{\be-1}(\la)=R_{\be-1}(\mu)=R_{\be}(\mu)$, and the subdiamonds   $\Sigma_{\be-1}(v_\mu)$ and $\Sigma_{\be-1}(v_\la)$ coincide, we will denote the top of this subdiamond by
$\overline{t}$, and the bottom by $\overline{b}$, cf. Figure~\ref{F:Case5}.

\begin{figure}[h]
\centering
\includegraphics[scale=0.8]{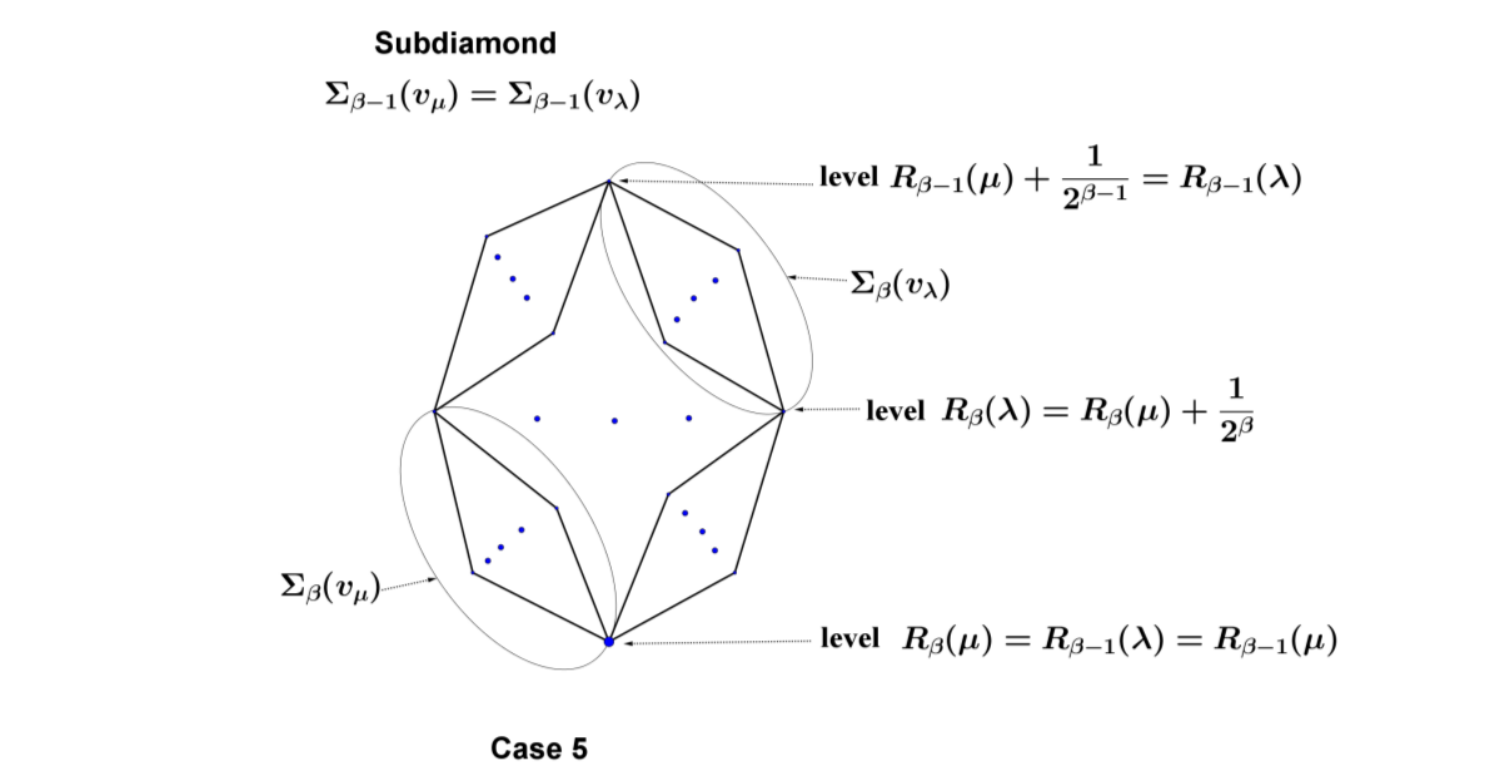}
\caption{Subdiamonds in Case 5.}\label{F:Case5}
\end{figure}

By
Observation~\ref{O:NestSubdiam}, $\overline{b}$ is also the bottom vertex of the subdiamond $\Sigma_\be(v_\mu)$, and the top vertex of  $\Sigma_\be(v_\mu)$ is at the same level as the bottom vertex of the subdiamond $\Sigma_\be(v_\la)$. Note however that since $i_\beta\ne j_\beta$, the top vertex of  $\Sigma_\be(v_\mu)$  and the bottom vertex of the subdiamond $\Sigma_\be(v_\la)$ cannot coincide.
 Therefore
the shortest path between $v_\mu$ and $v_\lambda$ has to pass  either
through $\overline{t}$ or $\overline{b}$. Thus
\begin{equation}\label{E:DistCase3}
\begin{split}
d_{D_{n,k}}(v_{\la},
v_{\mu})&=\min\left\{\sum_{\al=\beta+1}^{s(\mu)}\frac{\mu_\al}{2^\al}+\sum_{\al=\beta}^{s(\la)}\frac{\la_\al}{2^\al},\
\frac{1}{2^{\be-1}}-\left(\sum_{\al=\beta+1}^{s(\mu)}\frac{\mu_\al}{2^\al}+\sum_{\al=\beta}^{s(\la)}\frac{\la_\al}{2^\al}\right)\right\}\\
&=\min\left\{\sum_{\al=\beta}^{s(\mu)}\frac{\mu_\al}{2^\al}+\sum_{\al=\beta}^{s(\la)}\frac{\la_\al}{2^\al},\
\frac{1}{2^{\be-1}}-\left(\sum_{\al=\beta}^{s(\mu)}\frac{\mu_\al}{2^\al}+\sum_{\al=\beta}^{s(\la)}\frac{\la_\al}{2^\al}\right)\right\}.
\end{split}
\end{equation}

To estimate the distance between $x_\la$ and $x_\mu$, as in
previous cases, we look  at the set $G$ consisting of the values
of $\nu$ for which $r_{(j_1,\dots,j_\beta)}(\nu)=-1$ and
$r_{(i_1,\dots,i_\beta)}(\nu)=1$. Then, by the inductive
definition of $x_\mu$ and $x_\la$, and by the definition of $\be$,
we obtain that $C_{\be-1}(x_\mu)=C_{\be-1}(x_\la)\subseteq
P_\la(\nu)\cap P_\mu(\nu)$,
$S^{2^{n+1}\nu}I^{(n)}_{\e_1,\dots,\e_{\beta-1},-1}\subseteq
P_\la(\nu)\setminus C_{\be-1}(x_\la)$, and $P_\mu(\nu)\setminus
C_{\be-1}(x_\mu)\subseteq
S^{2^{n+1}\nu}I^{(n)}_{\e_1,\dots,\e_{\beta-1},+1}.$ Thus
$$S^{2^{n+1}\nu}I^{(n)}_{\e_1,\dots,\e_{\beta-1},-1}\subseteq P_\la(\nu)\setminus P_\mu(\nu).$$
Therefore we have
\begin{equation*}\lb{case5}
\begin{split}
\|x_{\la}-x_{\mu}\|&\ge \Big\|\sum_{\nu\in G} (x_{\la}-x_{\mu})\cdot
\one_{S^{2^{n+1}\nu}[1,2^{n+1}]}\Big\|\\
&\ge 2^{n}\frac1{2^\beta}\Big\|\sum_{\nu\in G}
S^{2^{n+1}\nu}(e_1-e_2)\Big\|\\
&\ge\frac18\frac1{2^{\beta-1}}\|x_1^{(n)}\|\\
&\stackrel{{\rm by \eqref{E:DistCase3}}}{\ge} \frac18\|x_1^{(n)}\|d_{D_{n,k}}(v_\la,v_\mu).
\end{split}
\end{equation*}

This completes the proof of Theorem \ref{T:Main}.

\section{The set of  diamonds of all finite branchings does
not satisfy the factorization assumption \eqref{factor} -- Proof
of Theorem~\ref{no-diamonds-in-X-Delta}}\label{S:NoMultBrXdelta}

The goal of this section is to prove
Theorem~\ref{no-diamonds-in-X-Delta}. Our approach  is the
following. First we show that if for a fixed $k$ and all
$n\in\bbN$,  the diamond graphs $\{D_{n,k}\}_{n=1}^\infty$ can be
embedded into $\ell_1$ with uniformly bounded distortions and so
that the embeddings can be
  factored between the summing and the
$\ell_1$-norm with uniformly bounded factorization constants (see the discussion of \eqref{factor} in the Introduction), then there exists a ``vertically almost isometric'' embedding of  the diamond $D_{1,k}$ into $\ell_1$ that satisfies \eqref{factor} with the same factorization constant, see
Lemma~\ref{stretch} for precise formulation.
 In some contexts arguments of this type are called ``self-improvement'' arguments. Their first usage in Banach space
theory is apparently due to James \cite{Jam64a}, and in non-linear
setting  to Johnson and Schechtman \cite{JS09}. This argument is
by now standard, Lee and Raghavendra \cite[Lemma~4.1]{LR10} prove
essentially the same lemma as ours, but since their terminology is
different, we decided to enclose the following elementary proof
for convenience of the readers.

Next, we prove (Lemma \ref{reduction}) that an embedding of
$D_{1,k}$ into $\ell_1$ that satisfies \eqref{factor} with the
factorization constant $C$ can be further ``improved'' so that it
resembles closely a standard embedding of $D_{1,2}$ into $\ell_1$,
that is, for some $N\in\bbN$, the top vertex of $D_{1,k}$  is
mapped onto a vector in $\ell_1^N$ whose every coordinate is 1 or
$-1$, the bottom vertex is mapped onto $0$, and all other vertices
of $D_{1,k}$ are mapped onto elements of  $\ell_1^N$ such that
their pairwise $c_0$-distance does not exceed 1, and their
pairwise summing norm distance
  is at least $\alpha N$, where $\alpha>0$ depends
only on the factorization constant of the original embedding, see
Lemma~\ref{reduction} for the precise statement.

After reaching this point we use the Ramsey theorem to show that
the number of branches $k$ in this situation is bounded from above by a constant that depends only on $\alpha$ and not on $N$, see
Lemma~\ref{bound-on-branches}.
An outline of this step is
described at the beginning of the proof of Lemma~\ref{bound-on-branches}.

We use the standard notation $c_{00}$ for the linear space of
infinite sequences of real numbers with finite support. We shall
use the following norms on $c_{00}$: the $\ell_1$-norm
$\|\cdot\|_1$ and the summing norm $\|\cdot\|_s$.

\begin{lemma}\label{stretch}  Suppose that there exist  $C> 1$ and $k\in \mathbb{N}$   such that for
every $n\in\mathbb{N}$ there exists an embedding
$f_n:D_{n,k}\to\ell_1$ satisfying
\begin{equation}\label{E:LemStr}\forall u,v\in  D_{n,k}\quad
\|f_n(u)-f_n(v)\|_1\le d_{D_{n,k}}(u,v)< C\cdot
\|f_n(u)-f_n(v)\|_s.\end{equation}

Then for every $\eta\in(0,1)$ there exist nonzero elements
$\{x_i\}_{i=0}^k$ in $\ell_1$,
 so that,  for all $i,j\in\{1,\dots,k\}$ with $i\ne j$, we have
\begin{align}\label{midpoint}
\frac{(1-\eta)}{2}\|x_0\|_1 &\le\|x_i\|_1\le \frac{(1+\eta)}{2}\|x_0\|_1,\\
\frac{(1-\eta)}{2}\|x_0\|_1 &\le\|x_0-x_i\|_1\le
\frac{(1+\eta)}{2}\|x_0\|_1,\label{midpoint2}
\end{align}
and
\begin{equation}\label{far}
\|x_i-x_j\|_s>\frac{1}{C}\|x_i-x_j\|_1\ge\frac{1}{C^2}\|x_0\|_1,
\end{equation}
\end{lemma}

\begin{lemma}\label{reduction} Suppose that there exist  $C> 1$ and $k\in \mathbb{N}$   such that for
every $n\in\mathbb{N}$ there exists an embedding
$f_n:D_{n,k}\to\ell_1$ satisfying
$$\forall u,v\in D_{n,k}\quad \|f_n(u)-f_n(v)\|_1\le d_{D_{n,k}}(u,v)< C\cdot \|f_n(u)-f_n(v)\|_s.$$

Let $\al=\frac{1}{2C^2}>0$. Then, there exist $N\in\bbN$  and
elements $z_i=\sum_{m=1}^\infty z_{im}e_m\in c_{00}$, for
$i\in\{1,\dots,k\}$, so that
 \begin{align}
\label{suppz}
&\forall i\in\{1,\dots,k\}& \supp(z_i) &\subseteq \{1,\dots,N\},\\
\label{zdiff}
&\forall i,j\in\{1,\dots,k\}\ \forall m\in\{1,\dots,N\}&   |z_{im}-z_{jm}|&\le 1,\\
\label{lfarz} &\forall i, j\in\{1,\dots,k\}, \  i\ne j, &
\|z_i-z_j\|_s&\ge\al N,
\end{align}

and
\begin{equation}\label{alphaN}
\al N\ge 2.
\end{equation}
\end{lemma}

\begin{lemma}\label{bound-on-branches}
For every $\al\in(0,1)$, there exists a natural number $k(\al)$,
so that if there exist   $k, N\in\bbN$, and elements
$z_i=\sum_{m=1}^\infty z_{im}e_m\in c_{00}$, for
$i\in\{1,\dots,k\}$, satisfying conditions
\eqref{suppz}--\eqref{alphaN}, then
$$k\le k(\al).$$
\end{lemma}

Recall that the diamond $D_{1,k}$ consists of $(k+2)$ vertices. In
this section we shall use for them notation which is different
from the one used before: the bottom vertex will be denoted by
$v_{-1}$, the top vertex will be denoted by $v_{0}$, and the $k$
vertices, which are midpoints between $v_{-1}$ and $v_0$ will be
denoted by $\{v_i\}_{i=1}^k$.  For all $1\le i,j\le k$, $i\ne j$
the distances between the vertices in $D_{1,k}$ are
\begin{align}\lb{dist1}
1=d_{D_{1,k}}(v_{-1},v_{0})=d_{D_{1,k}}(v_i,v_j)=2d_{D_{1,k}}(v_0,v_i)=2d_{D_{1,k}}(v_i,v_{-1}).
\end{align}

\begin{proof}[Proof of Lemma~\ref{stretch}]
Let $\eta\in(0,1)$ be given, and let $\de\in(0,\eta/5)$.  We denote by $\{v_i^n\}_{i=-1}^k$ the vertices of
$D_{n,k}$ which correspond to vertices $\{v_i\}_{i=-1}^k$ in
$D_{1,k}$. For each $n\in\bbN$, we define $t(n)$ to be the
supremum of $ \|f_n(v_{0}^n)-f_n(v_{-1}^n)\|_1$ over all
bilipschitz embeddings $f_n:D_{n,k}\to \ell_1$ satisfying
\eqref{E:LemStr}. The supremum is finite because
$d_{D_{n,k}}(v_{0}^n,v_{-1}^n)=1$. Note that for every $n\in
\bbN$, the diamond $D_{n+1,k}$ contains an isometric copy of
$D_{n,k}$ with the same top and bottom vertex. Thus, for every
$m\in\bbN$, $t(n+1)\le t(n).$ Since, by \eqref{E:LemStr} and
because $\|\cdot \|_s\le \|\cdot \|_1$, the sequence
$(t(n))_{n\in\bbN}$ is bounded below by $1/C$, it is convergent.
We define
$$t=\lim_{n\to\infty} t(n).$$

Let $n\in\bbN$ be such that
\begin{equation*}
t\le t(n)\le t(n-1)\le (1+\de)t.
\end{equation*}

Then there exists a bilipschitz embedding $f_{n}:D_{n,k}\to
\ell_1$, satisfying \eqref{E:LemStr},   such that
$f_n(v_{-1}^n)=0$, and
\begin{equation}\label{tm}
(1-\de)t\le (1-\de)t(n)\le\|f_n(v_{0}^n)-f_n(v_{-1}^n)\|_1\le
t(n)\le (1+\de)t.
\end{equation}
We put $x_j=f_n(v_{j}^n)$ for $j\in\{-1,0,1,\dots,k\}$.
 Note that, for every $i=1,\dots, k$, the diamond $D_{n,k}$ contains two $1/2$-scaled copies of the diamond $D_{n-1,k}$,  with
top-bottom pairs $(v_{-1}^n,v_{i}^n)$ and $(v_{i}^n,v_{0}^n)$, respectively. Since
$f_n$ restricted to either of these subdiamonds satisfies
\eqref{E:LemStr}, we obtain that
\begin{align}\label{midpointm}
\|x_i\|_1 &\le\frac12 t(n-1)\le\frac12(1+\de)t.\\
\|x_0-x_i\|_1 &\le\frac12
t(n-1)\le\frac12(1+\de)t.\label{midpointm2}
\end{align}
Since $\|x_0\|_1\le \|x_i\|_1+\|x_0-x_i\|_1$, by
\eqref{tm}--\eqref{midpointm2}, we obtain
\begin{align*}
\frac{1}{2}\frac{(1-3\de)}{(1-\de)}\|x_0\|_1 &\le\|x_i\|_1\le
\frac{1}{2}\frac{(1+\de)}{(1-\de)}\|x_0\|_1,\\
\frac{1}{2}\frac{(1-3\de)}{(1-\de)}\|x_0\|_1 &\le\|x_0-x_i\|_1\le
\frac{1}{2}\frac{(1+\de)}{(1-\de)}\|x_0\|_1,
\end{align*}
 and, since
$\de<\frac{\eta}{5}$, we conclude that \eqref{midpoint} and
\eqref{midpoint2} are satisfied.

Further,   by \eqref{dist1} and \eqref{E:LemStr} we get that   $\|x_0\|_1\le 1$, and
therefore
\begin{equation*}
\|x_i-x_j\|_1\ge||x_i-x_j||_s\stackrel{\eqref{E:LemStr}}{>}
\frac{1}{C}\ge\frac{1}{C} \|x_0\|_1.
\end{equation*}
Using \eqref{E:LemStr} again we get \eqref{far}.
\end{proof}

\begin{proof}[Proof of Lemma~\ref{reduction}] Let $\eta=\frac{1}{2C^2}>0$. By
Lemma~\ref{stretch}, there exist
nonzero elements
 $\{x_i\}_{i=0}^k$ in $\ell_1$ satisfying \eqref{midpoint}--\eqref{far}.
\Buoo, we may assume that the vector $x_0\in \ell_1$ has finite
support and rational coefficients.  Thus, after rescaling (which
is applied to all vectors
 $\{x_i\}_{i=0}^k$), we may
assume that all coefficients of $x_0$ are  integers and
\begin{equation}\label{scalex0}
\|x_0\|_1\ge  4C^2.
\end{equation}

Let $p\in\bbN$ and $\{a_m\}_{m=1}^p\subset
 \bbZ$, be such that
$x_0=\sum_{m=1}^p a_me_m$. We define $b_0\DEF0$,  and $b_m\DEF\max\{b_{m-1}+|a_m|, b_{m-1}+1\}$,  for each $m\in\{1,\dots,p\}$. Next we define
an operator $T$ on $c_{00}$ by putting for every
$y=\sum_{m=1}^t y_{m}e_m\in c_{00}$, where $t\in\bbN$,
\begin{equation*}
T\left(\sum_{m=1}^t y_{m}e_m\right)=
\begin{cases}{\disp
\sum_{m=1}^t \left(\sum_{\nu=b_{m-1}+1}^{b_m}\frac{y_{m}}{b_m-b_{m-1}} \ e_\nu\right),} \ \ \ & {\text{\rm \ if\ }} t\le p,\\
{\disp \sum_{m=1}^p
\left(\sum_{\nu=b_{m-1}+1}^{b_m}\frac{y_m}{b_m-b_{m-1}} \
e_\nu\right)+\sum_{m=p+1}^{t}y_{m}e_{b_p+m-p}, }\ \ \ & {\text{\rm
\ if\ }} t> p.
\end{cases}
\end{equation*}

 Notice that both the $\ell_1$-norm and the summing norm are equal-signs-additive (ESA)  on the unit vector basis $\{e_m\}_{m=1}^\infty$ of
 $c_{00}$ (see Definition~\ref{D:ESAetc}). Therefore  the operator $T$ is an isometry on $c_{00}$ in both of these norms. Thus the elements
 $\{T(x_i)\}_{i=0}^k$ in $\ell_1$  also satisfy \eqref{midpoint}--\eqref{far}.

 Note that by the definition of the numbers $\{b_m\}_{m=1}^p$, we have
$$T(x_0)=T\left(\sum_{m=1}^p a_me_m\right)=\sum_{m=1}^p \left( \sum_{\nu=b_{m-1}+1}^{b_m}\e_me_\nu\right),$$
where $\e_m=\sign(a_m)$,  for each $m\in\{1,\dots,p\}$ (we use the
convention that $\sign(0)=0$).Thus all nonzero coordinates of
$T(x_0)$ are equal to 1 or $-1$.

 Thus, after applying all the above operations  if necessary,  we may assume \buoo, that there exist
nonzero elements
 $\{x_i\}_{i=0}^k$ in $\ell_1$ that satisfy \eqref{midpoint}--\eqref{far}, and so that $x_0=\sum_{m=1}^\infty x_{0m}e_m\in c_{00}$ and all nonzero coefficients of  $x_0$ satisfy $|x_{0m}|=1$. Let $N\in\bbN$ be such  that
 $$\|x_0\|_1=\sum_{m\in\supp(x_0)} |x_{0m}|=N.$$

For each $1\le i\le k$, we write
$$x_i=\dot{x_i}+\ddot{x_i},$$
where $\supp(\dot{x_i})\subseteq\supp(x_0)$, and
$\supp(\ddot{x_i})\cap\supp(x_0)=\emptyset$.  Then
\begin{align*}
\|x_i\|_1&=\|\dot{x_i}\|_1+\|\ddot{x_i}\|_1,\\
\|x_0-x_i\|_1&=\|x_0-\dot{x_i}\|_1+\|\ddot{x_i}\|_1,
\end{align*}
and thus, by summing \eqref{midpoint} and \eqref{midpoint2}, we obtain
\begin{equation*}
\begin{split}
(1+\eta)\|x_0\|_1&\ge
\|x_i\|_1+\|x_0-x_i\|_1\\&=\|\dot{x_i}\|_1 + \|\ddot{x_i}\|_1+
\|x_0-\dot{x_i}\|_1+\|\ddot{x_i}\|_1\\&\ge
\|x_0\|_1+2\|\ddot{x_i}\|_1
\end{split}
\end{equation*}

Thus
\begin{equation}\label{hat}
 \|\ddot{x_i}\|_1\le \frac12\eta\|x_0\|_1,
\end{equation}
and
\begin{equation}\label{supportx0}
\begin{split}
 \|\dot{x_i}\|_1 + \|x_0-\dot{x_i}\|_1
&= \sum_{m\in\supp(x_0)} \Big(|x_{im}|+|x_{0m}-x_{im}|\Big)\\
&\le (1+\eta)\|x_0\|_1.
\end{split}
\end{equation}
For each $1\le i\le k$, we define the following sets
\begin{align*}
A_i&=\{m\in \supp(x_0)\ : \ (|x_{im}|\le 1) \wedge (\sign(x_{im})=\sign(x_{0m}))\},\\
B_i&=\{m\in \supp(x_0)\ : \ (|x_{im}|> 1) \wedge (\sign(x_{im})=\sign(x_{0m}))\},\\
C_i&=\{m\in \supp(x_0)\ : \ \sign(x_{im})\ne\sign(x_{0m})\},\\
D_i&= \supp(x_i)\setminus\supp(x_0).
\end{align*}
Since we use the convention that $\sign(0)=0$, the sets $A_i, B_i,
C_i$ are mutually disjoint, and $A_i\cup B_i\cup C_i= \supp(x_0)$.

Note that for every $m\in \supp(x_0)$ and every
$i\in\{1,\dots,k\}$ we have
\begin{equation*}\label{sumk}
|x_{im}|+|x_{0m}-x_{im}|=
\begin{cases}
1 \ \ \ & {\text{\rm \ if\ }} m\in A_i,\\
1+2(|x_{im}|-1) \ \ \ & {\text{\rm \ if\ }} m\in B_i,\\
1+2|x_{im}| \ \ \ & {\text{\rm \ if\ }} m\in C_i.
\end{cases}
\end{equation*}

Thus, by \eqref{supportx0}, we obtain
\begin{equation}\label{supportx01}
\begin{split}
(1+\eta)\|x_0\|_1&\ge
\sum_{m\in\supp(x_0)}
\Big(|x_{im}|+|x_{0m}-x_{im}|\Big)\\
 &= \sum_{m\in\supp(x_0)} 1+ \sum_{m\in B_i} 2(|x_{im}|-1) + \sum_{m\in C_i} 2|x_{im}|  \\
 &= \|x_0\|_1+ 2\left[\sum_{m\in B_i} (|x_{im}|-1) + \sum_{m\in C_i} |x_{im}|\right]
\end{split}
\end{equation}

Now, for $1\le i\le k$, we define elements $z_i=\sum_{m=1}^\infty
z_{im}e_m\in c_{00}$, by setting
\begin{equation*}
z_{im}=
\begin{cases}
x_{im} \ \ \ & {\text{\rm \ if\ }} m\in A_i,\\
x_{0m} \ \ \ & {\text{\rm \ if\ }} m\in B_i,\\
0 \ \ \ & {\text{\rm \ if\ }} m\notin A_i\cup B_i.\\
\end{cases}
\end{equation*}

Thus, for each $1\le i\le k$, $\supp(z_i)\subseteq \supp(x_0)$.
Moreover for each $m\in \supp(x_0)$, and each $1\le i,j \le k$, we
have
\begin{equation*}\label{zijk}
|z_{im}|\le 1, \ \ \ \  {\text{\rm \ and \ \ \ \ }}
|z_{im}-z_{jm}|\le 1.
\end{equation*}

Further, by \eqref{hat} and \eqref{supportx01},  we obtain
\begin{equation}\label{approxzx}
\begin{split}
\|x_i-z_i\|_1&=\sum_{m\in B_i} (|x_{im}|-1) + \sum_{m\in C_i} |x_{im}| + \sum_{m\in D_i} |x_{im}| \\
&\le \eta\|x_0\|_1.
\end{split}
\end{equation}
Thus, using \eqref{approxzx}, \eqref{midpoint}, \eqref{midpoint2},
and \eqref{far}, we obtain for all $1\le i \le k$,
\begin{align*}
\frac{(1-3\eta)}{2}\|x_0\|_1 &\le\|z_i\|_1\le \frac{(1+3\eta)}{2}\|x_0\|_1,\\
\frac{(1-3\eta)}{2}\|x_0\|_1 &\le\|x_0-z_i\|_1\le
\frac{(1+3\eta)}{2}\|x_0\|_1,
\end{align*}
and,  for all $1\le i, j\le k$, $i\ne j$,
\begin{equation*}\label{farz}
\|z_i-z_j\|_s\ge\left(\frac{1}{C^2}-2\eta\right)\|x_0\|_1.
\end{equation*}

Therefore, since  $\eta=\frac1{4C^2}$ and $\al=\frac{1}{2C^2}$,
\eqref{lfarz} holds. Since,
  by \eqref{scalex0}, $N=\|x_0\|_1\ge 4C^2$,  we
have $\al N\ge 2$, that is, \eqref{alphaN} holds.

Finally, since for each $1\le i\le k$, $\supp(z_i)\subseteq \supp(x_0)$, and since
both the $\ell_1$-norm and the summing norm are ESA, we can ``remove all the common gaps'' in the supports of $x_0$ and $\{z_i\}_{i=1}^k$ by applying appropriate shift operators (by ESA, all such shifts are isometries in both the $\ell_1$-norm and the summing norm), that is, we can assume without loss of generality that $\supp(x_0)=\{1,\dots,N\}$.
 Thus all conditions
\eqref{suppz}--\eqref{alphaN} are satisfied, which ends the proof
of Lemma~\ref{reduction}.
\end{proof}

\begin{proof}[Proof of Lemma~\ref{bound-on-branches}]
Let $\al\in(0,1)$. We   will say   that a natural number $k$
satisfies property $P(\al)$ (or $k\in P(\al)$), if there exist
$N\in\bbN$,  and elements $z_i=\sum_{m=1}^\infty z_{im}e_m\in
\ell_1^N$, for all $i\in\{1,\dots,k\}$, that satisfy conditions
\eqref{suppz}--\eqref{alphaN}.

We fix  $k\in \bbN$ so that $k$ satisfies property $P(\al)$.

Let $N=N(k)\in\bbN$ be a corresponding natural number, and let
$z_i=\sum_{m=1}^\infty z_{im}e_m\in \ell_1^N$ for
$i\in\{1,\dots,k\}$ be the
 corresponding elements that satisfy
\eqref{suppz}--\eqref{alphaN} (note that these elements  may
depend on the values of both $k$ and $N$ but since both $k$ and
$N$ are now fixed, and  to avoid excessive subscripts, we do not
reflect this fact in our  notation).

For every $i, j\in\{1,\dots,k\}$ with $i\ne j$, we will denote by
$r(i,j)$ the smallest integer in $\{1,\dots,N\}$ such that
\begin{equation}\label{rij}
\al N\le\Big|\sum_{m=1}^{r(i,j)} (z_{im}-z_{jm})\Big|<\al N +1,
\end{equation}
that is $r(i,j)$ is the smallest index that witnesses the fact that the summing norm distance between $z_i$ and $z_j$ is at least  $\al N$. By our assumptions
 \eqref{zdiff} and \eqref{lfarz},  for every
 $i\ne
j$, the number $r(i,j)$ exists.

Our proof of Lemma~\ref{bound-on-branches} consists of two
essential steps. First, we prove that for every three pairwise
distinct numbers $i,j,l\in\{1,\dots,k\}$ the values of indices
$r(i,j),r(i,l),r(j,l)$ cannot ``stay together'', and at least two
of them are separated by a positive distance independent of
$i,j,l$,  see Lemma~\ref{L:Lemma6}.

In the second step we prove that if, for example, for all triples
$1\le i<j<l\le k$, the maximum of $r(i,j),r(i,l),r(j,l)$ is always
attained at $r(j,l)$, then, for all $i,j$, the values of indices
$r(i,j)$ have to grow by a fixed amount with every increase of $i$
and $j$. As a consequence, we obtain that $r(k-1,k)$ would have to
be much larger than $r(1,2)$, see \eqref{final}. Since all
elements $z_i$ are in $\ell_1^N$, we know that $r(k-1,k)\le N$,
and thus there would be a bound on the size of $k$. We would
obtain similar bounds if the maximum of $r(i,j),r(i,l),r(j,l)$  is
always  equal to $r(i,j)$, or  always equal to  $r(i,l)$. This
leads us to a 3-coloring of triples from $\{1,\dots,k\}$, and
using the Ramsey theorem we conclude that for large $k$ there
exist large subsets of $\{1,\dots,k\}$ with monochromatic triples,
which leads us to an upper bound for $k$.

\begin{lemma}\label{L:Lemma6}
For every pairwise distinct triple of numbers $i,j,l\in\{1,\dots,k\}$ we have:
\begin{equation}\label{Lemma6}
\max\{r(i,j),r(i,l),r(j,l)\}- \min\{r(i,j),r(i,l),r(j,l)\}\ge
\frac{\al N-1}{2}>0.
\end{equation}
\end{lemma}

\begin{proof}
Let $\tau(i)$, $\tau(j)$ and $\tau(l)$ be the sums of the respective
sequences up to the term number $r(i,j)$. That is, for example,
$\tau(l)=\sum_{m=1}^{r(i,j)} z_{lm}$. Then, by the definition of
$r(i,j),$ we have
\begin{equation*}
 \alpha N\le
|\tau(i)-\tau(j)|<\alpha N+1.
\end{equation*}

Suppose, \buoo, that
$$|\tau(l)-\tau(i)|\le |\tau(l)-\tau(j)|.$$

Then we have one the following  two possibilities, cf.
Figure~\ref{F:Lemma6},
\begin{equation}\label{E:4-4-case1}
|\tau(l)-\tau(i)|\le\frac{\alpha N+1}2,
\end{equation}
or
\begin{equation}\label{E:4-4-case2}
|\tau(l)-\tau(i)|>\frac{\alpha N+1}2, \ \ \text{\rm and}\ \ |\tau(l)-\tau(j)|>\frac{3\alpha N+1}2.
\end{equation}

\begin{figure}[h]
\centering
\includegraphics[scale=1]{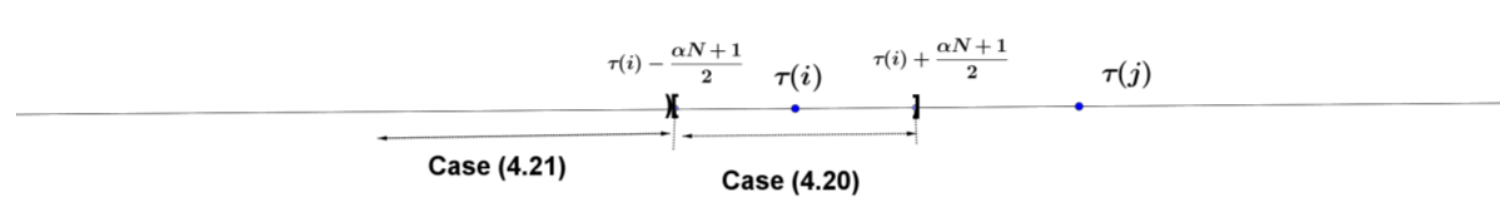}
\caption{Intervals containing $\tau(l)$.}\label{F:Lemma6}
\end{figure}

In the following computation we will use the convention that if $a,b\in\bbN$,  $a<b$, and $\{c_i\}_{i=a}^b$ are real numbers, then
$$\sum_{m=b}^a c_i\DEF -\sum_{m=a}^b c_i.$$

Suppose first that \eqref{E:4-4-case1} holds. Then, by the
definition of $r(i,l)$, and the above convention, we have
\begin{equation*}
\begin{split}
\al N&\le\Big|\sum_{m=1}^{r(i,l)} (z_{im}-z_{lm})\Big|\\
&\le \Big|\sum_{m=1}^{r(i,j)} (z_{im}-z_{lm})\Big| + \Big|\sum_{m=r(i,j)+1}^{r(i,l)} (z_{im}-z_{lm})\Big|\\
&\le |\tau(i)-\tau(l)| +
\Big|\sum_{m=r(i,j)+1}^{r(i,l)} |z_{im}-z_{lm}|\ \Big|\\
&\stackrel{{\rm by \eqref{E:4-4-case1}\, and\, \eqref{zdiff}}}{\le} \frac{\alpha N+1}2 +
|{r(i,j)}-{r(i,l)}|,
\end{split}
\end{equation*}
and therefore $|{r(i,j)}-{r(i,l)}|> \frac{\alpha N-1}2,$ so \eqref{Lemma6} holds in this case.

Next we suppose  that \eqref{E:4-4-case2} holds. Then, by the definition of $r(j,l)$, we have
\begin{equation*}
\begin{split}
\al N+1&>\Big|\sum_{m=1}^{r(j,l)} (z_{jm}-z_{lm})\Big|\\
&\ge \Big|\sum_{m=1}^{r(i,j)} (z_{jm}-z_{lm})\Big| - \Big|\sum_{m=r(i,j)+1}^{r(j,l)} (z_{jm}-z_{lm})\Big|\\
&\stackrel{{\rm by \eqref{E:4-4-case2}\, and\, \eqref{zdiff}}}{\ge} \frac{3\alpha N+1}2 -
|{r(i,j)}-{r(j,l)}|,
\end{split}
\end{equation*}
and therefore $|{r(i,j)}-{r(j,l)}|> \frac{\alpha N-1}2,$ so \eqref{Lemma6} holds also in this case, which ends the proof of Lemma~\ref{L:Lemma6}.
\end{proof}

We are now ready for the final step of    the proof of Lemma~\ref{bound-on-branches}   and
Theorem~\ref{no-diamonds-in-X-Delta}.

For every $1\le i<j<l \le k$ we define
\begin{equation*}\label{E:Mijl}
M_{ijl}=\max\{r(i,j),r(i,l),r(j,l)\}.
\end{equation*}

We will color triples $(i,j,l)\in\{1,\dots,k\}^3$ with   $1\le i<j<l
\le k$ as
\begin{itemize}
\item[--] {\it red} - if $M_{ijl}=r(j,l)$,
\item[--] {\it blue} - if $M_{ijl}=r(i,j)$, and $r(i,j)>r(j,l)$,
\item[--] {\it green} - if $M_{ijl}=r(i,l)$, and $r(i,l)>\max\{r(i,j),r(j,l)\}$.
\end{itemize}

We refer to \cite[Section 1.2]{GRS90} for basic facts of Ramsey
theory. By the Ramsey Theorem, for every $s\in\bbN$, there exists
a natural number denoted $R_3(s,3)\in\bbN$,   so that   for all $k\ge
R_3(s,3)$ the set $\{1,\dots,k\}$ contains a subset $B$ with
$\card(B)\ge s$  such that every triple $(i,j,l)\in B^3$ is of
the same color.

Let $s\in\bbN$, $s\ge 3$,  be such that there exists a  subset
$B=\{b_1,\dots,b_s\}$, listed in the increasing order, of
$\{1,\dots,k\}$ so that  every triple in $B^3$ is of the same
color. We will consider the three possible colors separately.

First we assume that the color of any triple in
$B^3$ is red. We show that in this case for every
$q\in\{1,\dots,s-1\}$ we have
\begin{equation}\label{final}
\forall t>q\ \ \ \ \ r(b_q,b_t)\ge \frac{(q+1)(\al N-1)}{2}.
\end{equation}

We prove \eqref{final} by induction on $q$.

When $q=1$, by \eqref{rij} and \eqref{zdiff}, for all $t>1$, we have
\begin{equation*}
\al N\le \sum_{m=1}^{r(b_1,b_t)} |z_{b_1m}- z_{b_tm}|\le
r(b_1,b_t),
\end{equation*}
so \eqref{final} is satisfied for $q=1$.

As the Inductive Hypothesis, we assume that \eqref{final} holds
for some $q<s-1$.

By \eqref{Lemma6}, the assumption that all triples in $B$ are
red, and the Inductive Hypothesis, for all $t>q+1$ we have
\begin{equation*}
\begin{split}
r(b_{q+1},b_t)&= M_{b_q,b_{q+1},b_t}\\
&\stackrel{\eqref{Lemma6}}{\ge} \min\{r(b_q,b_{q+1}),r(b_q,b_t)\} +\frac{\al N-1}{2}\\
&\ge \frac{(q+1)(\al N-1)}{2} + \frac{\al N-1}{2}\\
& = \frac{(q+2)(\al N-1)}{2}.
\end{split}
\end{equation*}
By induction, this ends the proof that  \eqref{final} holds for
every $q\in\{1,\dots,s-1\}$.

Since $z_{b_{s-1}}, z_{b_s}\in\ell_1^N$, as a consequence of
\eqref{final} we get
$$N\ge r(b_{s-1},b_s)\ge \frac{s(\al N-1)}{2}\stackrel{\eqref{alphaN}}{\ge} \frac{s\al N}{4} .$$

Thus $s\le \left\lfloor\frac{4}{\al}\right\rfloor,$ and, if all
triples in $B^3$ were red, we obtain that
\begin{equation}\lb{redest}
\card(B)\le \left\lfloor\frac{4}{\al}\right\rfloor.
\end{equation}

The case when all triples in $B^3$ are blue can be considered in
the same way, we  just list the elements in $B$ in the decreasing order. Thus  \eqref{redest} is also valid in this case.

It remains to consider the case when all triples are green. In
this case we prove by induction on $q\in\{0,\dots,\lfloor\log_2s\rfloor-1\}$
that for all $t,u \in \{1,\dots,s\}$ with $t<u$ and
$\log_2|u-t|\ge q$ we have
\begin{equation}\label{E:FinGreen}
r(b_t,b_u)\ge (q+2)\left(\frac{\al N-1}{2}\right).
\end{equation}

 By \eqref{rij} and\eqref{zdiff}, for any $t<u$ we
have
\begin{equation*}
\al N\le \sum_{m=1}^{r(b_u,b_t)} |z_{b_um}- z_{b_tm}|\le
r(b_u,b_t),
\end{equation*}
so \eqref{E:FinGreen} is satisfied for $q=0$.

As the Inductive Hypothesis, we assume that \eqref{E:FinGreen} holds
for some $q<\lfloor\log_2s\rfloor-1$.

 Assume that
$t<u$ and $\log_2|t-u|\ge q+1$. Let $w$ be such that $t<w<u$,
$\log_2|t-w|\ge q$,  and $\log_2|u-w|\ge q$. Then the assumption that the triple $(b_t,
b_w, b_u)$ is green,     the Inductive Hypothesis, and \eqref{Lemma6} imply that
\[\begin{split}r(b_t,b_u)&\ge \min\{r(b_t,b_w), r(b_w,b_u)\}+\frac{\al
N-1}{2}\\&\ge (q+3)\left(\frac{\al N-1}{2}\right),
\end{split}\]
which, by induction, proves \eqref{E:FinGreen}.

Therefore, since $z_{b_{1}}, z_{b_s}\in\ell_1^N$, and since by \eqref{alphaN}, $\al N\ge 2$,  we get
$$N\ge r(b_1,b_s)\ge  \left(\Big\lfloor\log_2|s-1|\Big\rfloor+2\right)\left(\frac{\al N-1}{2}\right)\stackrel{\eqref{alphaN}}{\ge} \frac{(\log_2s)\al N}{4} .$$

Thus $\log_2s\le \left\lceil\frac{4}{\al}\right\rceil,$ and in the
case when all triples in $B^3$ are green we obtain that
\begin{equation*}\lb{greenest}
\card(B) \le 2^{\left\lceil\frac{4}{\al}\right\rceil}.
\end{equation*}
Together with \eqref{redest}, by the Ramsey theorem, this implies
that
$$k\le k(\alpha)\DEF
R_3\left(2^{\left\lceil\frac{4}{\al}\right\rceil},3\right),$$
which  ends the proof of Lemma~\ref{bound-on-branches} and
Theorem~\ref{no-diamonds-in-X-Delta} (with $k(C)\DEF
R_3\big(2^{\lceil8C^2\rceil},3\big)$).
\end{proof}

\section{Laakso graphs}\lb{S:Laakso}

In this    section we outline  a proof of an analog of
Theorem~\ref{T:Main} for multi-branching Laakso graphs. Recall
that Johnson and Schechtman \cite{JS09} proved that the set of all
(binary) Laakso graphs, is a set of test spaces for
super-reflexivity. These graphs were introduced by Lang and Plaut
\cite{LP01} whose construction was based on some ideas of Laakso
\cite{La00}. Laakso graphs have many similar properties with
diamond graphs, but, in addition,   are doubling, that is, every
ball in any of these graphs can be covered by a finite number of
balls of half the radius, and that finite number does not depend
on either the graph or the radius of the ball, see
\cite{La00,LP01}. Here we will consider a natural generalization
of Laakso graphs to graphs with an arbitrary finite number of
branches.

\begin{definition}\label{D:multiLaakso}(cf. \cite{LR10})
For  any integer  $k\ge 2$, we define   $L_{1,k}$ to be a graph consisting of $k+4$ vertices $\{s,t, s_1,t_1\}\cup\{v_i\}_{i=1}^k$ joined by the following $(2k+2)$ edges: $(s, s_1)$, $(t_1,t)$, and, for all $i\in\{1,\dots,k\}$, $(s_1,v_i)$, $(v_i,t_1)$.
  see Figure~\ref{F:L1}.
For any integer $n\ge 2$, if the graph $L_{n-1,k}$  is defined, we
define the graph $L_{n,k}$ as the graph obtained from $L_{n-1,k}$
by replacing each edge $uv$ in $L_{n-1,k}$ by a copy of the graph
$L_{1,k}$. We put uniform weights on all edges of $L_{n,k}$, and
we endow $L_{n,k}$ with the shortest path distance. We call
$\{L_{n,k}\}_{n=0}^\infty$ the {\it Laakso graphs of branching
$k$}.
\end{definition}
\begin{figure}[h]
\centering
\includegraphics[scale=0.6]{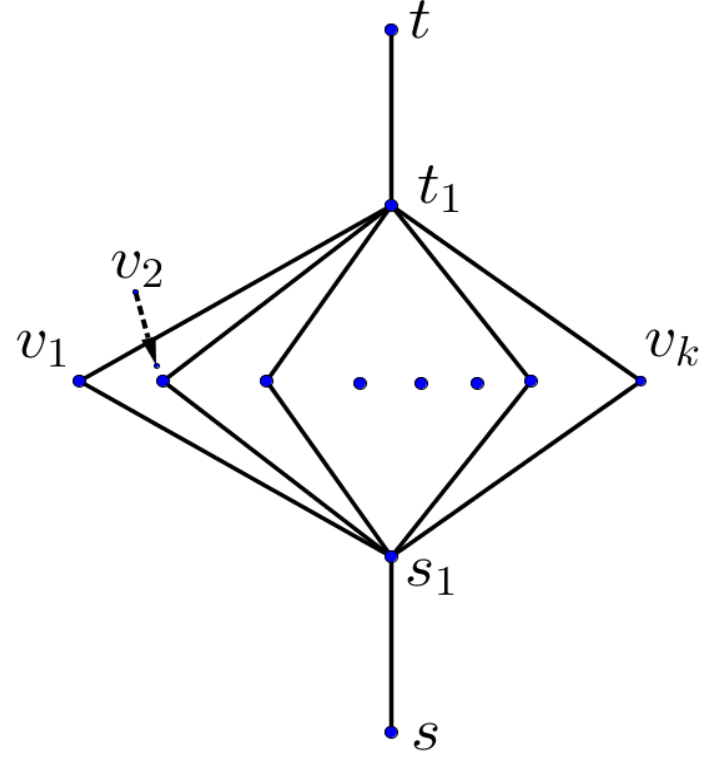}
\caption{The Laakso graph $L_{1,k}$}\label{F:L1}
\end{figure}

We refer to vertex $s$, as the {\it bottom}, and to the vertex $t$, as the {\it top} of the graph  $L_{n,k}$.
Similarly as in the case of the diamonds, we use the normalization of Laakso graphs so that   the distance from the top to the bottom vertex of   $L_{n,k}$ is equal to 1. In this normalization  $L_{n-1,k}$ is embedded isometrically in $L_{n,k}$.

We have the following   result, whose proof is very similar to the
proof of Theorem~\ref{T:Main}. We outline its proof below.

\begin{theorem}\label{T:Laakso} For every $\ep>0$, any non-superreflexive Banach
space $X$, and any $n,k\in \mathbb{N}$, $k\ge 2$, there exists a
bilipschitz embedding of $L_{n,k}$ into $X$ with distortion at
most $8+\ep$.
\end{theorem}

We start with an adaptation of the construction  in
Section~\ref{S:1,k}, to show
 an embedding of $L_{1,k}$ into  spaces with an \ESA\
basis, that  illustrates the pattern that will be iterated to
construct embeddings of $L_{n,k}$, for arbitrary $n\in \bbN$. We
note, that similarly as in the case of diamonds one can easily
find a bilipschitz embeddings  with small  distortions of
$L_{1,k}$ into any infinite-dimensional Banach space. As in
Section~\ref{S:D1k}, the usefulness of the construction described
below is  in the existence of a suitable iteration, that leads to
a low-distortion embedding of $L_{n,k}$, and in building an
intuition for a general embedding.

Using the same notation as before, we start from the element analogous to the element $h$ in Section~\ref{S:1,k}, but with twice as many nonzero coordinates.
$$h^{(2)}=++++----.$$

As before, let $M\ge k$ and
$r_1,\dots,r_M$ be the (natural analogues of) Rademacher
functions on $\{1,2,3,\dots,2^M\}$. We define the image of the bottom vertex $s$ to be $0$, and the images of the  vertices $t$,  $t_1$, $s_1$ as follows
\begin{align*}
x_t&=\sum_{\nu=0}^{2^M-1} S^{8\nu}(h^{(2)})=++++----......++++----,\\
x_{t_1}&=\sum_{\nu=0}^{2^M-1} S^{8\nu}(h^{(2)}_{--}+h^{(2)}_{+-}+h^{(2)}_{++} ) =+0++--0-......+0++--0-,\\
x_{s_1}&=\sum_{\nu=0}^{2^M-1} S^{8\nu}(h^{(2)}_{--}) =+000000-......+000000-.
\end{align*}

For every $i\in\{1,\dots,k\}$, the images of the  vertices $v_i$,
are defined to be
\begin{equation*}
\begin{split}
x_{v_i}&=\sum_{\nu=0}^{2^M-1} S^{8\nu}(h^{(2)}_{--}+h^{(2)}_{+,r_i(\nu)}).
\end{split}
\end{equation*}

Thus we define the embedding so that the independent random selection of elements in the middle,  which mimics the properties of the embedding in Section~\ref{S:1,k}, occurs on the supports of shifted copies of $h^{(2)}_{+}$.

 By IS and ESA of the basis we have
\begin{align*}
\|x_{t}\|&=4\left\|\sum_{\nu=0}^{2^M-1} S^{2\nu}(e_1-e_2)\right\|
=4\|\underbrace{+-+-....+-}_{2^M {\text {\small\  pairs}}}\|,\\
\|x_{t_1}\|&= \|x_{t}-x_{s_1}\|= 3\left\|\sum_{\nu=0}^{2^M-1} S^{2\nu}(e_1-e_2)\right\|=\frac34\|x_{t}\|,\\
\|x_{v_i}\|&= \|x_{t_1}-x_{s_1}\|= \|x_{t}-x_{v_i}\|=2\left\|\sum_{\nu=0}^{2^M-1} S^{2\nu}(e_1-e_2)\right\|=\frac24\|x_{t}\|,\\
\|x_{s_1}\|&= \|x_{t}-x_{t_1}\|= \|x_{t_1}-x_{v_i}\|= \|x_{v_i}-x_{s_1}\|= \left\|\sum_{\nu=0}^{2^M-1} S^{2\nu}(e_1-e_2)\right\|=\frac14\|x_{t}\|.
\end{align*}

Moreover, as in Section~\ref{S:1,k}, we  observe that when $i\ne j$, for   one
quarter of the values of $\nu$, we have $r_i(\nu)=1, r_j(\nu)=-1$. By SA, without increasing the norm, we can replace all the remaining blocks by zeros, and we obtain
\begin{equation*}\lb{D1level12}
\begin{split}
\|x_{v_i}-x_{v_j}\|&=\|\sum_{\nu=0}^{2^M-1} S^{8\nu}(h^{(2)}_{+,r_i(\nu)}-h^{(2)}_{+,r_j(\nu)})\|
\ge\|\sum_{\nu=0}^{2^{M-2}-1} S^{8\nu}(h^{(2)}_{+,+}-h^{(2)}_{+,-})\|\\
&=\|\sum_{\nu=0}^{2^{M-1}-1} S^{2\nu}(-e_1+e_2)\|\ge\frac18\|x_t\|.
\end{split}
\end{equation*}

Thus, as in Section~\ref{S:1,k}, we obtained an embedding of
$L_{1,k}$ with  distortion $\le 4$.
As before, the most important feature of this construction is that it can be
iterated without large increase of distortion, as we outline
below.

To describe an embedding of $L_{n,k}$ for any $n,k\in\bbN$, we develop a method of labelling the vertices of $L_{n,k}$, similar to that in Section~\ref{S:DescrD}.

We will say that a  vertex of $L_{n,k}$  is at the {\it level}
$\la$, if its distance from the bottom vertex is equal to $\la$.
Then $T_n\DEF\{\frac{t}{4^n} : 0\le t \le 4^n\}$ is the set of all
possible levels. For each $\lambda\in T_n$ we consider its tetradic
expansion
\begin{equation}\lb{tetradic}
\la=\sum_{\al=0}^{t(\la)} \frac{\la_\al}{4^\al},
\end{equation}
where $0\le t(\la)\le n$, $\la_\al\in\{0,1,2,3\}$ for each
$\al\in\{0,\dots,t(\la)\}$, and
$\la_{t(\lambda)}\ne 0$ for all $\lambda\ne 0$. We will use the convention
$t(0)=0$. Note that $1\in T_n$ is the only value of $\la\in T_n$
with $\la_0\ne 0$.

As in diamonds, we will say that a path in $L_{n,k}$ is a {\it direct vertical path} if it is a subpath of
a geodesic path that connects  the bottom and the
top vertex  in $L_{n,k}$.

We will say that a vertex $v\in L_{n,k}$ is {\it directly above} a vertex $u\in L_{n,k}$
if there exists a direct vertical path that passes through both vertices $v$ and $u$, and the level of $v$ is greater then the level of $u$. We define similarly  the notion that $v$ is  {\it directly below} a vertex $w$, and we say that a vertex $v$ is {\it between vertices} $u$ and $w$, if $v$ is directly below one of them and directly above the other.

\begin{remark}\lb{uniqueness}
We note here that   the Laakso graphs $L_{n,k}$ have  the following uniqueness property: If $v\in L_{n,k}$ is at the level $\la$ with
$0<t(\la)\le n$,  then there exists a unique vertex in $L_{n,k}$, that we will denote by $v^+$, so that
\begin{equation*}\lb{v+}
d_{ L_{n,k}}(v,v^+)=\frac{4-\la_{t(\la)}}{4^{t(\la)}},
\end{equation*}
and every geodesic path in $L_{n,k}$ that connects $v$ and the top of the graph $L_{n,k}$ has to pass through $v^+$.   Note that $v^+$ is at the level
\begin{equation*}\lb{la+}
\la_+\DEF \la+\frac{4-\la_{t(\la)}}{4^{t(\la)}},
\end{equation*}
and $t(\la_+)<t(\la)$.
Observe that there are many distinct vertices $v, u$ in $L_{n,k}$, even with $v$ and $u$ at different levels, so that $v^+=u^+$, cf. Figure~\ref{F:L2}.

Similarly, there exists a unique vertex $v^-\in L_{n,k}$,   so that
\begin{equation*}\lb{v-}
d_{ L_{n,k}}(v,v^-)=\frac{ \la_{t(\la)}}{4^{t(\la)}},
\end{equation*}
and every geodesic path in $L_{n,k}$ that connects $v$ and the bottom of the graph $L_{n,k}$ has to pass through $v^-$.   Note that $v^-$ is at the level
\begin{equation*}\lb{la-}
\la_-\DEF \la-\frac{\la_{t(\la)}}{4^{t(\la)}},
\end{equation*}
and $t(\la_-)<t(\la)$.

The  property that every vertex $v$ has the unique vertices $v^+$ and $v^-$ defined above is crucial for our construction of an embedding of $L_{n,k}$.
Note that diamonds $D_{n,k}$ have  a similar uniqueness property, cf.
Observation~\ref{O:NestSubdiam}. It is possible that our construction may be adapted for every family of series parallel graphs that posses such a uniqueness property, but we have not checked this carefully.
\end{remark}

We will  label each vertex $v$ of the graph $L_{n,k}$  by its
level $\la$, and by an ordered $\g$-tuple  $J=J(v)$ of numbers
from the set $\{1,\dots,k\}$, where $\g\in\{0,\dots,n\}$,
$v=v^{(n)}_{\la,J}$ (we do allow $J=\emptyset$).  We call $J$
   the {\it label of the branch of the vertex $v$}. We will define labels $J$ inductively
   on the  value of $t(\la)$ of the level  $\la$ of the vertex, so that for every vertex $v$ at the level $\la$ the length of the label $J(v)$ is smaller than or equal to $t(\la)$, and with the property that
 two vertices $u,v$ are connected by a direct vertical path \wtw\ either $J(v)=J(u)$,
or one of them is an initial segment of the other (the higher
vertex does not necessarily have a longer label).  The inductive
procedure is as follows, cf. Figure~\ref{F:L2}:

\begin{figure}[h]
\includegraphics[scale=1.1]{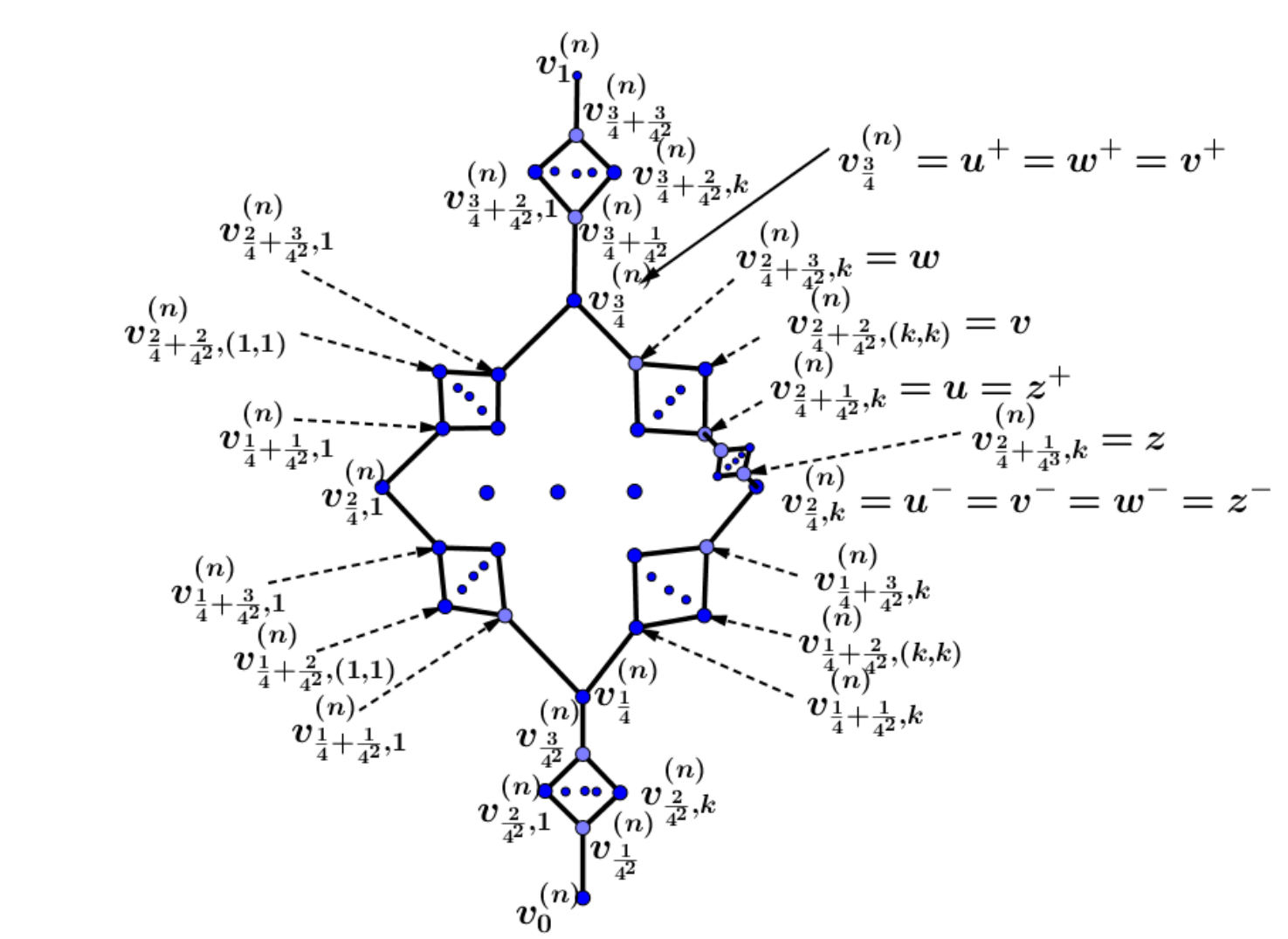}
\caption{The labels of the Laakso graph $L_{2,k}$ (with a small
part of $L_{3,k}$)}\label{F:L2} \end{figure}

\begin{itemize}
\item $t(\la)=0$: The bottom vertex is labelled $v^{(n)}_0$, and the top
vertex is labelled $v^{(n)}_1$.

\item   $t(\la)=1$: There are two possibilities:

$(i)$ $\la=\frac14$ or  $\la=\frac34$: at this  level there is one vertex, to which we  assign $J=\emptyset$.

$(ii)$ $\la=\frac24$: at this level there are $k$ vertices, to each of which we  assign $J=(j)$, for each vertex a different value of $j\in\{1,\dots,k\}$,

\item $t(\la)=t+1$, where $1\le t<n$,  and for
all $\mu\in T_n$ with $t(\mu)\le t$, all vertices at level $\mu$
have been labelled.

Since $t(\la)=t+1$, $\la_{t+1}\ne 0$. We consider vertices $v^-$
and $v^+$ defined in Remark~\ref{uniqueness}. Then $v$ is between
$v^-$ and $v^+$, also $t(\la_-)\le t$ and  $t(\la_+)\le t$. Thus
both vertices $v^-$, $v^+$ already have labels $J(v^-)$, $J(v^+)$,
respectively, and either $J(v^-)=J(v^+)$, or one of them is an
initial segment of the other. Denote by $J^\bigstar$ a longer of the
two labels $J(v^-)$, $J(v^+)$. Now we consider two possibilities:

$(i)$ if $\la_{t+1}=1$ or  $\la_{t+1}=3$, then we assign to the branch of $v$ the label  $J^\bigstar$,

$(ii)$ if $\la_{t+1}=2$, then we assign to  the branch of $v$   the label  whose
initial segment is $J^\bigstar$, which is followed by one of the numbers $j\in\{1,\dots,k\}$, thus the length of the label is increased by 1 compared to $J^\bigstar$.
\end{itemize}

Note that the set of all possible labels of branches is equal to  $ \{\emptyset\}\cup\mathcal{P}$, where $\mathcal{P}$ is the set  of all tuples
$(j_1,\dots,j_s)$ of all lengths between $1$ and $n$, with
$j_i$ is in $\{1,\dots,k\}$. Recall that
$M\DEF \card({\mathcal{P}})= k+k^2+\dots+k^n$.

Now we are ready to define a bilipschitz embedding of $L_{n,k}$
into a Banach space $X$ with an ESA basis. We shall denote the image of
$v^{(n)}_{\la;J}$ in $X$ by
 $x^{(n)}_{\la;J}$, and we use the same notation as in Section~\ref{S:n,k}.

We define the image of the bottom vertex $v_0^{(n)}$ of  $L_{n,k}$
to be zero (that is, $x_0^{(n)}=0)$, and the image of the top
vertex to be the element $x^{(n)}_1$ that is defined as the sum of
$2^M$ disjoint shifted copies of $h^{(2n)}$, that is,
\begin{equation}\lb{defLtop}
x^{(n)}_1=\sum_{\nu=0}^{2^{M}-1} S^{2^{2n+1}\nu}(h^{(2n)}).
\end{equation}

Note that, by IS and ESA of the basis we have
\begin{equation*}\lb{norm-top-Laakso}
\begin{split}
\|x^{(n)}_1\|&=4^n\Big\|\sum_{\nu=0}^{2^{M}-1}
S^{2^{2n+1}\nu}(e_1-e_2)\Big\|=4^n\Big\|\sum_{\nu=0}^{2^{M}-1}
S^{2\nu}(e_1-e_2)\Big\|.
\end{split}
\end{equation*}

Next we describe an inductive process to define elements
$x^{(n)}_{\la,J}$ for all vertices $v^{(n)}_{\la,J}\in L_{n,k}$.

Note   that in Section~\ref{S:VerbDef} the  definition of the
image of a vertex in $D_{n,k}$ was obtained in several steps,
through an inductive procedure that used the dyadic representation
of the level of the vertex  in $D_{n,k}$ and the label of the
branch, and did not explicitly use images of any other  vertices
in the diamond.

In the inductive procedure described below, the induction is also
on the level of the vertex (or, more precisely, on the length of
the tetradic representation of the level of the vertex), but it
depends on the images of other elements in the graph $L_{n,k}$
whose levels have shorter tetradic representations, and the label
of the branch of the vertex  is not used explicitly.

The structure of elements $x^{(n)}_{\la,J}$ will be similar   to
the images of vertices under the embedding of the diamonds into
$X$, that is, $x^{(n)}_{\la,J}$  also will be composed of $2^M$
blocks, and each block is a finite sum of disjointly supported
elements of the form $S^{2^{2n+1}\nu}h^{(2n)}_\e$, where $\e$ is a
$\g$-tuple of $\pm1$, for some $\g\in\{0,1,\dots,2n\}$, where the
tuples $\e$   may depend on the number of the block $\nu$. Recall
that by  definition $h^{(2n)}_\e$ is $\{0,\pm1\}$-valued, its
support is contained in $[1, 2^{2n+1}]$, symmetric about the
center of this interval, and coordinates of $h^{(2n)}_\e$ are
equal to $+1$ on the first half of the support, and to $-1$ on the
second half of the support. Thus, similarly as in
Section~\ref{S:VerbDef}, each  element
$S^{2^{2n+1}\nu}h^{(2n)}_\e$ is uniquely determined by the portion
of its support contained in the interval
$S^{2^{2n+1}\nu}I^{(2n)}$, and since we are summing disjointly
supported elements  of this form, it is enough to describe the
sets $P(v^{(n)}_{\la, J},\nu)\DEF P(v^{(n)}_{\la, J})\cap
S^{2^{2n+1}\nu}I^{(2n)}$, for each $\nu\in\{0,\dots,2^M\}$, where
$P(v^{(n)}_{\la, J})\subset\mathbb{N}$ is the set where
$x^{(n)}_{\la, J}$ is equal to $1$. Recall, that for any tuple
$\e$, we denoted by $I^{(2n)}_\e$ the part of the support of
$h^{(2n)}_\e$ where the values of the coordinates of $h^{(2n)}_\e$
are equal to $+1$, and that
 $I^{(2n)}_\e\subseteq I^{(2n)} $.

 Next we will define  inductively the sets
$P(v,\nu)\subseteq I^{(2n)} $ so that for every $v=v^{(n)}_{\la,J}\in L_{n,k}$, and
every $\nu\in\{0,\dots,2^{M}-1\},$
\begin{equation}\lb{cardLP}
\begin{split}
\card(P(v^{(n)}_{\la,J},\nu))= 4^n\la,
\end{split}
\end{equation}
and
with the property that if the vertex $v$ is directly above the vertex $u\in  L_{n,k}$, then
\begin{equation}\lb{setsLP}
\begin{split}
P(v,\nu)= P(u,\nu)\cup S^{2^{2n+1}\nu} \left(\bigcup_{\e\in A} I^{(2n)}_\e)\right),
\end{split}
\end{equation}
for some subset $A=A(v,u,\nu)$     of the set ${\mathcal{H}}$  of  all $\pm1$-valued
tuples of all lengths between 0 and $2n$   uniquely determined by the following conditions:
\begin{itemize}
\item[(C1)] if the  tuples $\e$ and $\de$ are in   $A$, then the sets
$I^{(2n)}_\e$ and $I^{(2n)}_\de$ are disjoint,
\item[(C2)] if  a tuple $\e\in A$, then
$I^{(2n)}_\e$ is disjoint with the set $P(u,\nu)$,
\item[(C3)] for every $\e\in{\mathcal{H}}$, at most one of the tuples $(\e,+)$ and $(\e,-)$ can belong to $A$, that is, if
 $$ S^{2^{2n+1}\nu} \left(I^{(2n)}_{\e,+}\cup I^{(2n)}_{\e,-}\right)= S^{2^{2n+1}\nu} \left(I^{(2n)}_{\e}\right)\subseteq P(v,\nu)\setminus P(u,\nu),$$ then $\e\in A$, and $(\e,+)\not\in A$,  $(\e,-)\not\in A$; that is
the set $A$ has the smallest possible cardinality,
\item[(C4)] if  there exists a vertex $w\in L_{n,k}$, so that $v=w^+$ and $u=w^-$, then the set     $A$ consists of exactly one tuple which we will denote by
$\e(w^-,w^+,\nu)$, i.e.
for every $w\in L_{n,k}$, if $w$ is neither the top nor the bottom of $L_{n,k}$, then there exists a   tuple $\e(w^-,w^+,\nu)\in {\mathcal{H}}$  so that $I^{(2n)}_{\e(w^-,w^+,\nu)}$ is disjoint with $P(w^-,\nu)$, and
\begin{equation*}
P(w^+,\nu)= P(w^-,\nu)\cup   S^{2^{2n+1}\nu}
\left(I^{(2n)}_{\e(w^-,w^+,\nu)}\right).
\end{equation*}
\end{itemize}

The induction is on the value of $t(\la)$ of the level $\la$ that is represented in the form \eqref{tetradic}.

\begin{enumerate}
\item  $t(\la)=0$

(i) If  $\la_0=1$, for every $\nu$, we let $P(v,\nu)\DEF
S^{2^{2n+1}\nu}I^{(2n)}$.

It is clear that this happens if and only if $\la=1$ and the
vertex is $v_1^{(n)}$. Notice that this agrees with the formula
\eqref{defLtop}. In this case we have
$\card(P(v_1^{(n)},\nu))=\card(I^{(2n)})=4^n=4^n\la$, so
\eqref{cardLP} and \eqref{setsLP} are satisfied.

(ii) If  $\la_0=0$, we let $P(v,\nu)\DEF
S^{2^{2n+1}\nu}\emptyset=\emptyset$.

It is clear that this happens if and only if $\la=0$ and the vertex is
$v_0^{(n)}$. This agrees with the condition that $x_0^{(n)}=0$, and
\eqref{cardLP} and \eqref{setsLP} clearly hold.

It is clear that conditions (C1)-(C3) are satisfied in both cases
(i) and (ii). Moreover, we agree that
$h^{(2n)}=h^{(2n)}_{\emptyset}$, so   the condition (C4)  holds
for every $w\in L_{n,k}$ with $w^-=v^{(n)}_0$ and $w^+=v^{(n)}_1$,
that is  for every $w\in L_{n,k}$ at the level $\al$ with
$t(\al)=1$.

\item $t(\la)=t+1$

Suppose that for all $\mu\in T_n$ with $t(\mu)\le t$, for all
$\nu$, and all vertices $u=v^{(n)}_{\mu,J(u)}$ the sets $P(u,\nu)$
are defined in such a way that the conditions \eqref{cardLP},
\eqref{setsLP}, and (C1)-(C3) are satisfied, and the condition
(C4)  holds for every $w\in L_{n,k}$ at the level $\al$ with
$t(\al)\le t$.

 Let $v=v_{\la,J}=v^{(n)}_{\la,J}\in L_{n,k}$, $t(\lambda)=t+1$, and let
$\la_-$, $\la_+$,   $v^-$, $v^+$ be defined as in Remark~\ref{uniqueness}. Since
$t(\la_-)\le t$ and  $t(\la_+)\le t$, by      (C4), for every $\nu$, there exists a   tuple $\e(v,\nu)\in {\mathcal{H}}$  so that
$I^{(2n)}_{\e(v^-,v^+,\nu)}$ is disjoint with the set $P(v^-,\nu)$, and
\begin{equation*}
P(v^+,\nu)= P(v^-,\nu)\cup  S^{2^{2n+1}\nu} \left(I^{(2n)}_{\e(v^-,v^+,\nu)}\right).
\end{equation*}

Moreover, by  \eqref{cardLP}, we have
\begin{equation}\lb{extracard}
  \begin{split}
  \card\Big(I^{(2n)}_{\e(v^-,v^+,\nu)}\Big)
  &=\card(P(v^+,\nu))-\card(P(v^-,\nu))\\
&=4^n(\la^+-\la^-)=4^{n-t}
\end{split}
\end{equation}

We define the set $P(v,\nu)$ depending on the value of
$\la_{t+1}$, as follows:

(i) If $\la_{t+1}=1$, we define for all values of $\nu$,
\begin{equation*}
P(v^{(n)}_{\la,J},\nu)\DEF P(v^-,\nu)\cup S^{2^{2n+1}\nu}
\left(I^{(2n)}_{\e(v^-,v^+,\nu),-,r_J(\nu)}\right),
\end{equation*}
where, as in Section~\ref{S:D1k},
for $J\in\mathcal{P}$,  $r_J$  denotes the (natural analogue of)
Rademacher function on $\{0,\dots,2^{M}-1\}$, and $r_{\emptyset}(\nu)\DEF+1$ for all
$\nu\in\{0,\dots,2^{M}-1\}$.

Thus, in this case,
$
A(v^{(n)}_{\la,J},v^-,\nu)=\{(\e(v^-,v^+,\nu),-1,r_J(\nu))\},
$
 and
\begin{equation*}
\begin{split}
P(v^+,\nu)&=P(v^{(n)}_{\la,J},\nu) \cup
S^{2^{2n+1}\nu} \left(I^{(2n)}_{\e(v^-,v^+,\nu)}\setminus I^{(2n)}_{\e(v^-,v^+,\nu),-,r_J(\nu)}\right)\\
&=P(v^{(n)}_{\la,J},\nu) \cup S^{2^{2n+1}\nu}
\left(I^{(2n)}_{\e(v^-,v^+,\nu),-,-r_J(\nu)}\cup
I^{(2n)}_{\e(v^-,v^+,\nu),+}\right),
\end{split}
\end{equation*}
which implies that
\begin{equation*}\lb{added-block-level1u}
A(v^+,v^{(n)}_{\la,J},\nu)=\{(\e(v^-,v^+,\nu),-1,-r_J(\nu)),
(\e(v^-,v^+,\nu),+1)\}.
\end{equation*}

 (ii) If $\la_{t+1}=2$, we define
\begin{equation*}
P(v^{(n)}_{\la,J},\nu)\DEF  P(v^-,\nu)\cup  S^{2^{2n+1}\nu}
\left(I^{(2n)}_{\e(v^-,v^+,\nu),-,r_{J^\bigstar}(\nu)}\cup
I^{(2n)}_{\e(v^-,v^+,\nu),+,r_J(\nu)}\right),
\end{equation*}
where $J^\bigstar$ is the initial segment of $J$ of length one
less than the length of $J$, that is, by the definition of
labelling, $J^\bigstar$ is the longer of the labels $J(v^-)$ and
$J(v^+)$ (note that $J^\bigstar$ is also the label of the branch
of the vertex  on the level $\la-\frac{1}{4^{t+1}}$ that is
directly below $v$, which is the vertex  that we discussed in item
(i)).

Similarly as in case (i), we obtain
\begin{equation*}
\begin{split}
P(v^+,\nu)&=P(v^{(n)}_{\la,J},\nu) \cup S^{2^{2n+1}\nu}
\left(I^{(2n)}_{\e(v^-,v^+,\nu)}\setminus \Big(I^{(2n)}_{\e(v^-,v^+,\nu),-,r_{J^\bigstar}(\nu)}\cup I^{(2n)}_{\e(v^-,v^+,\nu),+,r_J(\nu)}\Big)\right)\\
&=P(v^{(n)}_{\la,J},\nu) \cup S^{2^{2n+1}\nu}
\left(I^{(2n)}_{\e(v^-,v^+,\nu),-,-r_{J^\bigstar}(\nu)}\cup
 I^{(2n)}_{\e(v^-,v^+,\nu),+,-r_J(\nu)})\right).
\end{split}
\end{equation*}
Thus
\begin{eqnarray*}\lb{added-block-level2}
A(v^{(n)}_{\la,J},v^-,\nu)&=&\{(\e(v^-,v^+,\nu),-1,r_{J^\bigstar}(\nu)),
(\e(v^-,v^+,\nu),+1,r_J(\nu))\},\\
\lb{added-block-level2u}
 A(v^+,v^{(n)}_{\la,J},\nu)&=&\{(\e(v^-,v^+,\nu),-1,-r_{J^\bigstar}(\nu)), (\e(v^-,v^+,\nu),+1,-r_J(\nu)) \}.
 \end{eqnarray*}

(iii) If $\la_{t+1}=3$, we define for all values of $\nu$,
\begin{equation*}
P(v^{(n)}_{\la,J},\nu)\DEF  P(v^-,\nu)\cup S^{2^{2n+1}\nu}
\left(I^{(2n)}_{\e(v^-,v^+,\nu),-,r_J(\nu)}\cup
I^{(2n)}_{\e(v^-,v^+,\nu),+}\right).
\end{equation*}

We note that when $\la_{t+1}=3$, then the label $J$ of the branch of the vertex $v$ is the same as the label
 of the branch of the vertex  on the level $\la-\frac{2}{4^{t+1}}$ that is directly below $v$.

Similarly as in previous cases, we obtain
\begin{equation*}
\begin{split}
P(v^+,\nu)&=P(v^{(n)}_{\la,J},\nu) \cup S^{2^{2n+1}\nu}\left(
I^{(2n)}_{\e(v^-,v^+,\nu)}\setminus\Big(I^{(2n)}_{\e(v^-,v^+,\nu),-,r_J(\nu)}\cup I^{(2n)}_{\e(v^-,v^+,\nu),+}\Big)\right)\\
&=P(v^{(n)}_{\la,J},\nu) \cup S^{2^{2n+1}\nu}
\left(I^{(2n)}_{\e(v^-,v^+,\nu),-,-r_J(\nu)}\right).
\end{split}
\end{equation*}
Thus \[A(v^{(n)}_{\la,J},v^-,\nu)=\{(\e(v^-,v^+,\nu),-1,r_J(\nu)), (\e(v^-,v^+,\nu),+1)\},\]
 \[A(v^+,v^{(n)}_{\la,J},\nu)=\{(\e(v^-,v^+,\nu),-1,-r_J(\nu))  \}.\]

Since for all $\de_1,\de_2\in\{\pm1\}$ and all $\e\in
{\mathcal{H}}$, we have
$I^{(2n)}_{\e,\de_1,\de_2} \subseteq
I^{(2n)}_{\e,\de_1} \subseteq  I^{(2n)}_{\e} $,
in all cases (i)-(iii),  each of the sets
$A(v^{(n)}_{\la,J},v^-,\nu)$ and $A(v^+,v^{(n)}_{\la,J},\nu)$
satisfies the conditions (C1)-(C3). Thus \eqref{setsLP} is satisfied for
both pairs of vertices $(v^{(n)}_{\la,J},v^-)$ and
$(v^+,v^{(n)}_{\la,J})$, and by the Inductive Hypothesis and
Remark~\ref{uniqueness}, \eqref{setsLP} and (C1)-(C3) hold also
for all other pairs of  vertices.

It is also clear from the above definitions and the
Inductive Hypothesis  that the condition (C4) will hold for all vertices $w$ that are directly between $v^-$ and $v^+$ and such that $w$ is at the level $\al$ with $t(\al)\le t+2$.

Moreover, by \eqref{cardI}, for all $\de_1,\de_2\in\{\pm1\}$ and all $\e\in {\mathcal{H}}$, we have
\begin{align*}
\card\Big(I^{(2n)}_{\e,\de_1}\Big)&=\frac12 \card\Big(I^{(2n)}_{\e}\Big)\\
\card\Big(I^{(2n)}_{\e,\de_1,\de_2}\Big)&=\frac14 \card\Big(I^{(2n)}_{\e}\Big).
\end{align*}

Thus, by \eqref{extracard}, in each of the cases (i)-(iii) we have
\begin{equation*}
  \begin{split}
 \card(P(v^{(n)}_{\la,J},\nu))&=\card(P(v^-,\nu))+
 \frac{\la_{t+1}}{4} =4^n\Big(\la^-+\frac{\la_{t+1}}{4^{t+1}}\Big)=4^n\la,
\end{split}
\end{equation*}
 and
 \begin{equation*}
  \begin{split}
 \card(P(v^+,\nu))&=\card(P(v^{(n)}_{\la,J},\nu))+
 \frac{4-\la_{t+1}}{4} =4^n\Big(\la+\frac{4-\la_{t+1}}{4^{t+1}}\Big)=4^n\la^+.
\end{split}
\end{equation*}
Thus \eqref{cardLP} holds for both pairs of  vertices
$(v^{(n)}_{\la,J},v^-)$ and $(v^+,v^{(n)}_{\la,J})$. By the
Inductive Hypothesis and Remark~\ref{uniqueness}, \eqref{cardLP}
holds for all pairs of  vertices  that include the new vertex
$v^{(n)}_{\la,J}$.
\end{enumerate}

Note that it follows from \eqref{setsLP}  and from  the ESA property of the basis that the embedding defined by this inductive procedure maps endpoints of any direct vertical path in $L_{n,k}$ to points in $X$ that are at the distance equal to
$\|x^{(n)}_1\|$ times the distance between the original points. Thus, similarly as in the case of diamonds, our embedding of Laakso graphs is vertically isometric with the multiplicative constant $\|x^{(n)}_1\|$. In particular, this implies that this embedding is  Lipschitz  with the Lipschitz constant equal to $\|x^{(n)}_1\|$.

To prove that this embedding satisfies the  co-Lipschitz estimate,
it remains to estimate from below the distances between images of
vertices $u, v$  that do not lie on a direct vertical path in
$L_{n,k}$. Note that then every geodesic path $\wp$ that connects
$u$ and $v$ passes through a vertex $w$ that splits the path $\wp$
into two subpaths, each of which is a direct vertical path. Thus
both pairs of vertices $(u,w)$ and $(v,w)$ are connected by a
direct vertical path,
\begin{equation}\lb{LGdistance-uv}
d_{L_{n,k}}(u,v)= d_{L_{n,k}}(u,w)+d_{L_{n,k}}(w,v),
\end{equation}
and  $w$ is either directly below both $u$ and $v$, or $w$ is   directly above both $u$ and $v$.

We will denote the levels of vertices $u,v,w$, by $\mu,\la,\om$, respectively, and the lengths of labels $J(u), J(v), J(w)$ of their branches, by $l(u), l(v), l(w)\in\{0,1,\dots,n\}$, respectively. We denote $J(u)=(i_1,\dots,i_{l(u)})$ if $l(u)>0$ (if $l(u)=0$, $J(u)=\emptyset$),  and $J(v)=(j_1,\dots,j_{l(v)})$ if $l(u)>0$   (if $l(v)=0$, $J(v)=\emptyset$).

We outline the proof in  the case when $w$ is  directly below both $u$ and $v$ (the other case is very similar).

In this case $\om<\la$ and $\om<\mu$. Note also that $w$ belongs to a direct vertical path that connects the bottom $v_0^{(n)}$ of the Laakso graph $L_{n,k}$ with $u$ and to a direct vertical path that connects   $v_0^{(n)}$  with $v$, and the vertex $w$ has the highest possible level among all common vertices of these two direct vertical paths. By Definition~\ref{D:multiLaakso}, cf. Figure~\ref{F:L2}, this maximality property implies that
$$\om_{t(\om)}=1,$$
when the level $\om$ of the vertex $w$ is represented in its tetradic expansion $\om=\sum_{\al=0}^{t(\om)} \frac{\om_\al}{4^\al}$, cf. \eqref{tetradic}.
Moreover, since $u$ is neither directly above nor directly below $v$, and by maximality of the level of $w$, none of the vertices $u$, $v$ can be equal to or be directly above
the vertex $\tilde{w}$ that is directly above $w$ at the level
$\om+\frac{2}{4^{t(\om)}}$, cf. Figure~\ref{F:L3}.
Hence the  levels of both $u$ and $v$ have to be strictly smaller than
$\om+\frac{2}{4^{t(\om)}}$, and since there exists a geodesic path
that connects  $u$ and $v$ and passes through the vertex $w$ at the level $\om$,
at least one of the levels of $u$ and $v$ cannot exceed $\om+\frac{1}{4^{t(\om)}}$.

\begin{figure}[h]
\centering
\includegraphics[scale=1]{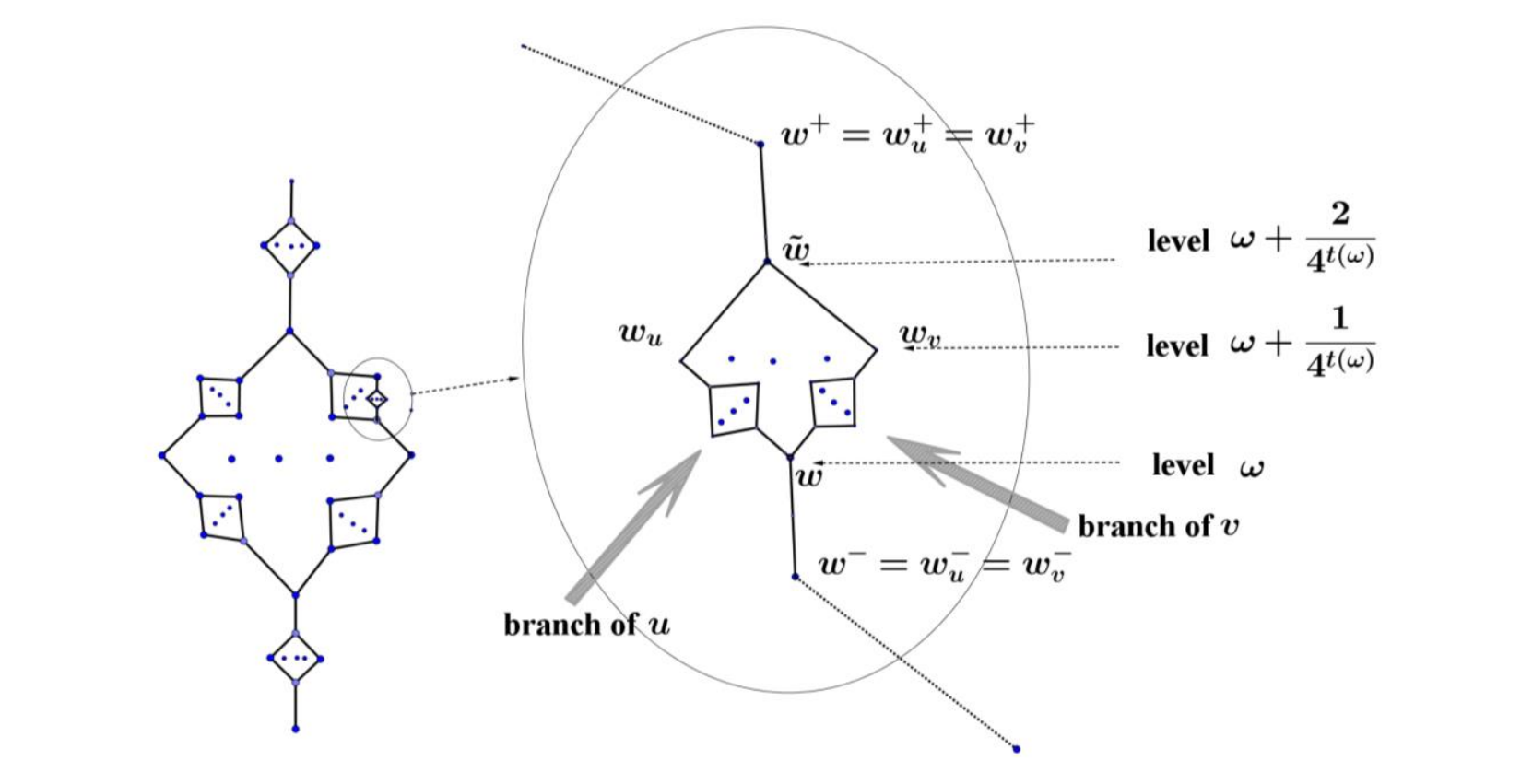}
\caption{Sample positions of vertices $u,v, w, w_u,w_v, \tilde{w}, w^-,w^+$.}\label{F:L3} \end{figure}

Before obtaining a lower estimate on the distance between images
of $u$ and $v$, we need to also look at the structure of the
labels of branches of $u$, $v$,and $w$. Notice that the label
$J(w)$ is the largest common initial segment of labels $J(u)$ and
$J(v)$. Thus $l(u), l(v)$ are both greater than or equal to
$l(w)+1\ge 1$, $(i_1,\dots,i_{l(w)})=(j_1,\dots,j_{l(w)})=J(w),$
where, if $l(w)=0$ then $J(w)=\emptyset$, and
\[i_{l(w)+1}\ne j_{l(w)+1}.\]

We denote $J^\Box(u)\DEF (i_1,\dots,i_{l(w)+1})$, $J^\Box(v)\DEF (j_1,\dots,j_{l(w)+1})$,
$w_u\DEF v^{(n)}_{\om+\frac{1}{2^{t(\om)}},J^\Box(u)}$,
and $w_v\DEF v^{(n)}_{\om+\frac{1}{2^{t(\om)}},J^\Box(v)}$, cf.
Figure~\ref{F:L3}. Note that
\begin{equation*}
w^-=w_u^-=w_v^-\ \  {\text{ and}} \ \ w^+=w_u^+=w_v^+.
\end{equation*}

Recall that, by (C4),  for all $\nu\in \{0,\dots,2^M-1\}$,
  there exists a   tuple $\e(w^-,w^+,\nu)\in {\mathcal{H}}$  so that $I^{(2n)}_{\e(w^-,w^+,\nu)}$ is disjoint with $P(w^-,\nu)$, and
\begin{equation*}
P(w^+,\nu)= P(w^-,\nu)\cup   S^{2^{2n+1}\nu}
\left(I^{(2n)}_{\e(w^-,w^+,\nu)}\right).
\end{equation*}

By the definition of our embedding we have,  for all $\nu\in \{0,\dots,2^M-1\}$,

\begin{equation}\lb{wu}
P(w_u,\nu)=  P(w,\nu)\cup  S^{2^{2n+1}\nu}
\left(
I^{(2n)}_{\e(w^-,w^+,\nu),+,r_{J^\Box(u)}(\nu)}\right),
\end{equation}

\begin{equation}\lb{wv}
P(w_v,\nu)=  P(w,\nu)\cup  S^{2^{2n+1}\nu}
\left(
I^{(2n)}_{\e(w^-,w^+,\nu),+,r_{J^\Box(v)}(\nu)}\right).
\end{equation}

We are now ready to obtain our estimates. \Buo we assume that $\mu\le\la$, and we denote by $x,y$ and $z$ the images of $u,v$, and $w$, respectively. We proceed similarly as in Section~\ref{S:EstDistor}.

Let $G$ be the set  consisting  of all
 $\nu$'s for which $r_{J^\Box(v)}(\nu)=-1$
and $r_{J^\Box(u)}(\nu)=1$. Since both tuples $J^\Box(v)$ and
$J^\Box(u)$ are nonempty and different from each other, by
independence of the  Rademacher functions, the cardinality the set
$G$ is equal to one fourth of the cardinality of the set of all
$\nu$'s, that is to $2^{M-2}$.

We first consider the  situation  similar to Case~\ref{I:C2a} in
Section~\ref{S:EstDistor}, that is  we suppose that both
$$
\la, \mu\le \om+\frac{1}{4^{t(\om)}}.
$$

In this situation, by  \eqref{setsLP},  \eqref{wu}, and  \eqref{wv}, for all
$\nu\in \{0,\dots,2^M-1\}$, the intersection of the support of  $y-x$ with $S^{2^{2n+1}\nu}I^{(2n)}$ is contained in the set
\[ S^{2^{2n+1}\nu}
\left(
I^{(2n)}_{\e(w^-,w^+,\nu),+,r_{J^\Box(v)}(\nu)}\right)\cup S^{2^{2n+1}\nu}
\left(
I^{(2n)}_{\e(w^-,w^+,\nu),+,r_{J^\Box(u)}(\nu)}\right),\]
and, the coordinates of $y-x$ are non-negative on the first portion of this set, and  non-positive  on the second portion of this set.
By \eqref{cardLP},  exactly $4^n(\la-\om)$ coordinates of $y-x$ are non-zero on the first portion of the set, and exactly $4^n(\mu-\om)$ coordinates of $y-x$ are non-zero on the second portion. Thus, for all $\nu\in G$,
 if we consider the restriction of the
difference $y-x$ to the interval
$S^{2^{2n+1}\nu}[1,2^{2n+1}]$ and omit all zeros, we get a vector of the
following form: first it will have $4^n(\la-\om)$ entries with
values equal to $+1$, then it will have $4^n(\mu-\om)$ entries
equal to $-1$, then it will have $4^n(\mu-\om)$ entries equal to
$+1$, and finally it will have $4^n(\la-\om)$ entries equal to
$-1$:
\begin{align*}
\underbrace{+\cdots \dots+}_{4^{n}(\la-\om)}
\underbrace{-\cdots\dots-}_{4^{n}(\mu-\om)}\underbrace{+\cdots\dots+}_{4^{n}(\mu-\om)}\underbrace{-\cdots
\dots-}_{4^{n}(\la-\om)}.
\end{align*}

  As in Case~\ref{I:C2a} in
Section~\ref{S:EstDistor}, we replace by zeros all entries of $y-x$ in all blocks $\nu\not\in G$, and for each $\nu\in G$,
we will replace by zeros the values  on the coordinates of
$y-x$  in the  ``central'' set in the $\nu$-th block. By the SA  property of the basis, this replacement does not
increase the norm of the element. Thus,
  by  the ESA property of the basis,  \eqref{LGdistance-uv}, and a computation very similar to the one  at the end of Case~\ref{I:C2a} in
Section~\ref{S:EstDistor} we get that
\begin{equation*}
\begin{split}
\|y-x\|&\ge  \frac18\|x_1^{(n)}\|d_{L_{n,k}}(v,u).
\end{split}
\end{equation*}

Next we first consider the  situation  similar to Case~\ref{I:C3} in
Section~\ref{S:EstDistor}, that is  we suppose that
\begin{equation}\lb{la-big}
 \mu\le \om+\frac{1}{4^{t(\om)}}<\la.
\end{equation}

Note that since there exists a geodesic path from $u$ to $v$ that
passes through the vertex $w$ on the level $\om\le \mu<\la$, we
have, cf. Figure~\ref{F:L3},
\begin{equation*}
\begin{split}
d_{L_{n,k}}(u,v)=(\mu-\om)+(\la-\om)&\le d_{L_{n,k}}(u,\tilde{w}) +d_{L_{n,k}}(\tilde{w},v)\\
&=\Big((\om+\frac2{4^{t(\om)}})-\mu\Big)+\Big((\om+\frac2{4^{t(\om)}})-\la\Big),
\end{split}
\end{equation*}
and thus
\begin{equation}\lb{dist-uv}
d_{L_{n,k}}(u,v)=(\mu-\om)+(\la-\om)\le \frac2{4^{t(\om)}}.
\end{equation}

Moreover, when \eqref{la-big} holds, the vertex $v$ is directly above $w_v$, and by the definition of our embedding
we have,  for all $\nu\in \{0,\dots,2^M-1\}$,
\begin{equation}\lb{Pv-la-big}
P(w_v,\nu)\subseteq P(v,\nu)\subseteq P(\tilde{w},\nu)=  P(w_v,\nu)\cup  S^{2^{2n+1}\nu}
\left(
I^{(2n)}_{\e(w^-,w^+,\nu),+,-r_{J^\Box(v)}(\nu)}\right).
\end{equation}

As in the
previous case, let $G$ be the set   of all the values
of $\nu$ for which $r_{J^\Box(v)}(\nu)=-1$ and
$r_{J^\Box(u)}(\nu)=1$. Then, by \eqref{wu}, \eqref{wv}, and \eqref{Pv-la-big}, we obtain that for all $\nu\in G$,
$$S^{2^{2n+1}\nu}\left(
I^{(2n)}_{\e(w^-,w^+,\nu),+-}\right)\subseteq P(v,\nu)\setminus P(u,\nu).$$

Note that, by \eqref{extracard} and \eqref{cardI}, $\card(I^{(2n)}_{\e(w^-,w^+,\nu),+-})=4^{n-t(\om)-1}$.
Therefore, similarly as in  Case~\ref{I:C3} in
Section~\ref{S:EstDistor},      we have
\begin{equation*}\lb{case5}
\begin{split}
\|y-x\|&\ge \Big\|\sum_{\nu\in G} (y-x)\cdot
\one_{S^{2^{2n+1}\nu}[1,2^{2n+1}]}\Big\|\\
&\ge 4^{n-t(\om)-1}\Big\|\sum_{\nu\in G}
S^{2^{2n+1}\nu}(e_1-e_2)\Big\|\\
&\ge\frac18\frac2{4^{t(\om)}}\|x_1^{(n)}\|\\
&\stackrel{{\rm by \eqref{dist-uv}}}{\ge} \frac18\|x_1^{(n)}\|d_{L_{n,k}}(u,v).
\end{split}
\end{equation*}

This ends the outline of the proof of the lower estimate in the case when  $w$ is  directly below both $u$ and $v$. The case when $w$ is  directly above both $u$ and $v$ is proved similarly, we omit the details.

\section{Acknowledgement}

The first named author gratefully acknowledges the support by the
National Science Foundation grant DMS-1201269 and by the Summer Support of
Research program of St. John's University during different stages
of work on this paper. Part of the work on this paper was done
when both authors were participants of the NSF supported Workshop
in Analysis and Probability at Texas A\&M University, 2016.

We thank Steve Dilworth, Bill Johnson, James Lee, and Assaf Naor
for useful comments related to the subject of this paper, and the
anonymous referee for very careful reading of the manuscript and
for numerous helpful suggestions that led to the improvement of
the presentation.

\begin{small}

\renewcommand{\refname}{

\textsc{Department of
Mathematics and Computer Science, St. John's University, 8000 Utopia Parkway, Queens, NY 11439, USA} \par
  \textit{E-mail address}: \texttt{ostrovsm@stjohns.edu} \par
  \smallskip

\textsc{Department of Mathematics, Miami University,
Oxford, OH 45056, USA} \par
  \textit{E-mail address}: \texttt{randrib@miamioh.edu} \par

\end{small}


\section{References}}

\begin{thebibliography}{999}

\bibitem{AR98} Y.~Aumann, Y.~Rabani, An $O(\log k)$ approximate
min-cut max-flow theorem and approximation algorithm, {\it  SIAM
J. Comput.}, {\bf 27}  (1998),  no. 1, 291--301.

\bibitem{Bal13} K.~Ball, The Ribe programme. S\'eminaire Bourbaki. Vol. 2011/2012.
Expos\'es 1043--1058. Ast\'erisque No. {\bf 352} (2013), Exp. No.
1047, viii, 147--159.

\bibitem{Bau07} F.~Baudier, Metrical characterization of
super-reflexivity and linear type of Banach spaces, {\it Archiv
Math.},  {\bf  89}  (2007),  no. 5, 419--429.

\bibitem{BCDKRSZ16+} F.~Baudier, R.~Causey,  S.\,J.~Dilworth, D.~Kutzarova,
N.\,L.~Randrianarivony, Th.~Schlum\-precht, S.~Zhang, On the
geometry of the countably branching diamond graphs, preprint, {\tt
arXiv:1612.01984}.

\bibitem{Bea79} B.~Beauzamy, Banach-Saks properties and spreading models.
{\it Math. Scand.} {\bf 44} (1979), no. 2, 357--384.

\bibitem{Bou86} J.~Bourgain, The metrical interpretation of superreflexivity in
Banach spaces, {\it Israel J. Math.}, {\bf 56} (1986), no. 2,
222--230.

\bibitem{BMW86} J.~Bourgain, V.~Milman, H.~Wolfson, On type of
metric spaces, {\it Trans. Amer. Math. Soc.}, {\bf 294} (1986),
no. 1, 295--317.

\bibitem{BS74} A.~Brunel, L.~Sucheston,
On $B$-convex Banach spaces. {\it Math. Systems Theory} {\bf 7}
(1974), no. 4, 294--299.

\bibitem{BS75} A.~Brunel, L.~Sucheston,
On $J$-convexity and some ergodic super-properties of Banach
spaces. {\it Trans. Amer. Math. Soc.} {\bf 204} (1975), 79--90.

\bibitem{BS76} A.~Brunel, L.~Sucheston, Equal signs additive
sequences in Banach spaces. {\it J. Funct. Anal.} {\bf 21} (1976),
no. 3, 286--304.

\bibitem{CJLV08} A. Chakrabarti, A. Jaffe, J.\,R. Lee, J. Vincent,
Embeddings of Topological Graphs: Lossy Invariants, Linearization,
and $2$-Sums, {\it Proceedings of the 2008 49th Annual IEEE
Symposium on Foundations of Computer Science}, p.~761--770,
October 25--28, 2008.

\bibitem{CCN15} J.~Chalopin, V.~Chepoi, G.~Naves, Isometric embedding of Busemann
surfaces into $L_1$. {\it Discrete Comput. Geom.} {\bf 53} (2015),
no. 1, 16--37.

\bibitem{DL97} M.\,M.~Deza, M.~Laurent, {\it Geometry of cuts and
metrics}. Algorithms and Combinatorics, {\bf 15}. Springer-Verlag,
Berlin, 1997.

\bibitem{Die00} R.~Diestel,  {\em Graph theory.} Second edition. Graduate Texts in Mathematics, {\bf 173}, Springer, Heidelberg, 2000.


\bibitem{GRS90} R.\,L.~Graham, B.\,L.~Rothschild, J.\,H.~Spencer, {\it Ramsey
theory.} Second edition. Wiley-Interscience Series in Discrete
Mathematics and Optimization. A Wiley-Interscience Publication.
John Wiley \&\ Sons, Inc., New York, 1990.

\bibitem{GNRS04} A.~Gupta, I.~Newman, Y.~Rabinovich,
A.~Sinclair, Cuts, trees and $\ell_1$-embeddings of graphs, {\it
Combinatorica}, {\bf 24} (2004) 233--269; Conference version in:
{\it 40th Annual IEEE Symposium on Foundations of Computer
Science}, 1999, pp.~399--408.

\bibitem{Ind01} P.~Indyk, Algorithmic applications
of low-distortion geometric embeddings, in: {\it Proc 42nd IEEE
Symposium on Foundations of Computer Science}, 2001, pp.~10--33,
availble at: {\tt http://theory.lcs.mit.edu/$\sim$indyk/}


\bibitem{Jam64a} R.\,C.~James, Uniformly non-square Banach spaces,
{\it Annals of Math.}, {\bf 80} (1964), 542--550.

\bibitem{Jam64} R.\,C.~James, Weak compactness and reflexivity, {\it Israel J. Math.}, {\bf
2} (1964), 101--119.

\bibitem{JS09} W.\,B.~Johnson, G. Schechtman, Diamond graphs and
super-reflexivity, {\it J. Topol. Anal.}, {\bf 1} (2009),
177--189.

\bibitem{Klo14} B.~Kloeckner, Yet another short proof of the Bourgain's distortion estimate
for embedding of trees into uniformly convex Banach spaces, {\it
Israel J. Math.},  {\bf 200} (2014), no. 1, 419--422; DOI:
10.1007/s11856-014-0024-4.

\bibitem{La00} T.\,J.~Laakso,   Ahlfors $Q$-regular spaces with arbitrary $Q>1$ admitting weak Poincar\'e inequality, {\it Geom. Funct. Anal.} {\bf 10} (2000),
111--123.

\bibitem{LP01} U.~Lang, C.~Plaut,
Bilipschitz embeddings of metric spaces into space forms.
{\it Geom. Dedicata} {\bf  87} (2001), no. 1-3, 285--307.

\bibitem{LM10} J.\,R.~Lee, M.~Moharrami,  Bilipschitz snowflakes and metrics of negative type. STOC'10--{\it Proceedings of the 2010 ACM International Symposium on Theory of Computing}, 621--630, ACM, New York, 2010.

\bibitem{LP13} J.\,R.~Lee, D.\,E.~Poore, On the
$2$-sum embedding conjecture. {\it Computational geometry}
(SoCG'13), 197--206, ACM, New York, 2013.

\bibitem{LR10} J.\,R.~Lee, P.~Raghavendra, Coarse differentiation and
multi-flows in planar graphs. {\it Discrete Comput. Geom.} {\bf
43} (2010), no. 2, 346--362.

\bibitem{LS09} J.\,R.~Lee,  A.~Sidiropoulos,  On the geometry of graphs with a forbidden minor. STOC'09--{\it Proceedings of the 2009 ACM International Symposium on Theory of Computing}, 245--254, ACM, New York, 2009.

\bibitem{LS13} J.\,R.~Lee,  A.~Sidiropoulos, Pathwidth, trees, and random embeddings. {\it Combinatorica} {\bf 33} (2013), no. 3, 349--374.

\bibitem{LNOO15+} S.\,L.~Leung, S.~Nelson, S. Ostrovska, M.\,I.~Ostrovskii, Distortion of embeddings of binary trees into diamond
graphs, preprint, {\tt arXiv:1512.06438}.

\bibitem{Lin02} N.~Linial, Finite metric spaces--combinatorics,
geometry and algorithms, in: {\it Proceedings of the International
Congress of Mathematicians}, Vol. {\bf III} (Beijing, 2002),
573--586, Higher Ed. Press, Beijing, 2002.

\bibitem{LLR95} N.~Linial, E.~London, Yu.~Rabinovich,
The geometry of graphs and some of its algorithmic applications,
{\it Combinatorica}, {\bf 15} (1995), no. 2, 215--245.

\bibitem{Mat99} J.~Matou\v sek, On embedding trees into
uniformly convex Banach spaces, {\it Israel J. Math.}, {\bf 114}
(1999), 221--237.

\bibitem{Mat02} J.~Matou\v sek, {\it Lectures on Discrete
Geometry},  Graduate Texts in Mathematics, {\bf 212},
Springer-Verlag, New York, 2002.

\bibitem{MatN-open-list} J.~Matou\v sek (Editor), starting June 2010 maintained jointly with A.~Naor, Open problems
on embeddings of finite metric spaces, last update August 2011,
available at:\\ {\tt http://kam.mff.cuni.cz/$\sim$matousek/}.

\bibitem{MM65} D.\,P.~Milman,  V.\,D.~Milman,  The geometry of nested families with
empty intersection. Structure of the unit sphere of a nonreflexive
space (Russian), {\it Matem. Sbornik}, {\bf 66} (1965), no.  1,
109--118; English transl.: {\it Amer.  Math.  Soc.  Transl.}  (2)
v.~{\bf 85} (1969), 233--243.

\bibitem{Nao12} A.~Naor,
An introduction to the Ribe program, {\it Jpn. J. Math.}, {\bf 7}
(2012), no.~2, 167--233.

\bibitem{OO15+} S.~Ostrovska, M.\,I.~Ostrovskii,
Nonexistence of embeddings with uniformly bounded distortions of
Laakso graphs into diamond graphs, {\it Discrete Math.},  {\bf
340} (2017), no. 2, 9--17.

\bibitem{Ost11} M.\,I.~Ostrovskii, On metric characterizations of some classes of Banach
spaces, {\it  C. R. Acad. Bulgare Sci.}, {\bf 64} (2011), no. 6,
775--784.

\bibitem{Ost13} M.\,I.~Ostrovskii, {\it Metric Embeddings: Bilipschitz and Coarse Embeddings into Banach
Spaces}, de Gruyter Studies in Mathematics, {\bf 49}. Walter de
Gruyter \&\ Co., Berlin, 2013.

\bibitem{Ost13a} M.\,I.~Ostrovskii, Different forms of metric
characterizations of classes of Banach spaces, {\it Houston. J.
Math.}, {\bf 39} (2013), no. 3, 889--906.

\bibitem{Ost14} M.\,I.~Ostrovskii,
Metric characterizations of superreflexivity in terms of word
hyperbolic groups and finite graphs, {\it Anal. Geom. Metr.
Spaces} {\bf 2} (2014), 154--168.

\bibitem{Ost16}  M.\,I.~Ostrovskii, Metric characterizations of some classes of Banach
spaces, in: {\it Harmonic Analysis, Partial Differential
Equations, Complex Analysis, Banach Spaces, and Operator Theory},
Celebrating Cora Sadosky's life, M.\,C.~Pereyra, S.~Marcantognini,
A.\,M.~Stokolos, W.\,U.~Romero (Eds.), Association for Women in
Mathematics Series, Vol. {\bf 4}, pp.~307--347, Springer-Verlag,
Berlin, 2016.

\bibitem{OR16} M.\,I.~Ostrovskii, B.~Randrianantoanina, Metric spaces admitting low-distortion embeddings into all $n$-dimensional Banach spaces.
{\it Canad. J. Math.} {\bf 68} (2016), no. 4, 876--907.

\bibitem{Pel62} A.~Pe\l czy\'nski, A note on the paper of I.~Singer ``Basic
sequences and reflexivity of Banach spaces'', {\it Studia Math.},
{\bf 21} (1961/1962), 371--374.

\bibitem{Pis86} G.~Pisier,
Probabilistic methods in the geometry of Banach spaces, in: {\it
Probability and analysis (Varenna, 1985)}, 167--241, {\it Lecture
Notes in Math.}, {\bf 1206}, Springer, Berlin, 1986.

\bibitem{Pis16} G.~Pisier, {\it Martingales in Banach
spaces}, Cambridge Studies in Advanced Mathematics {\bf 155}.
Cambridge, Press Cambridge University Press, 2016.

\bibitem{Pta59} V.~Pt\'ak, Biorthogonal systems and
reflexivity of Banach spaces, {\it Czechoslovak Math. J.}, {\bf 9}
(1959), 319--326.

\bibitem{Sid13} A.~Sidiropoulos, Non-positive curvature and the planar
embedding conjecture. 2013 {\it IEEE 54th Annual
 Symposium on Foundations of Computer Science FOCS} 2013, 177--186, IEEE Computer Soc., Los Alamitos, CA, 2013.

\bibitem{Sin62} I.~Singer, Basic sequences and reflexivity of Banach spaces,
{\it Studia Math.}, {\bf 21} (1961/1962), 351--369.

\bibitem{WS11}  D.\,P.~Williamson, D.\,B.~Shmoys,
{\it The Design of Approximation Algorithms}, Cambridge University
Press, 2011.

\end{thebibliography}
\end{document}